\documentclass[a4paper,reqno]{amsart}
\usepackage{etex}
\usepackage{amscd,amssymb,amsmath,mathrsfs,latexsym,upgreek,url,mathtools}

\usepackage[all]{xy}
\usepackage{color,graphicx}
\usepackage{caption,subcaption,pictex}
\usepackage{hyperref}
\hypersetup{breaklinks=true}

\def\ov#1{{\overline{#1}}}

\def\wh#1{{\widehat{#1}}}
\def\wt#1{{\widetilde{#1}}}
\def\?{\ ???\ \immediate\write16{}%
\immediate\write16{Warning: There was still a question mark . . . }%
\immediate\write16{}}

\newcommand{\conv}{\operatorname{conv}}

\newcommand{\vol}{\operatorname{vol}}

\newcommand{\Cb}{\cC_{\rm b}}
\newcommand{\Cc}{\cC_{\rm c}}
\renewcommand{\and}{{\quad \text{ and } \quad}}
\newcommand{\norm}{\operatorname{N}}

\newcommand{\supp}{\operatorname{supp}}

\newcommand{\capacity}{\operatorname{cap}}

\newcommand{\Gal}{\operatorname{Gal}}

\newcommand{\avol}{\operatorname{\wh{vol}}}

\newcommand{\exv}{\operatorname{E}}

\newcommand{\dd}{\hspace{1pt}\operatorname{d}\hspace{-1pt}}

\renewcommand{\div}{\operatorname{div}}

\newcommand{\Spec}{\operatorname{Spec}}

\newcommand{\stab}{\operatorname{stab}}

\newcommand{\chern}{\operatorname{c}}

\renewcommand{\i}{\operatorname{i}}
\newcommand{\ri}{\operatorname{ri}}

\newcommand{\e}{\operatorname{e}}
\newcommand{\h}{\operatorname{h}}

\newcommand{\ball}{{{\mathcal B}}}

\newcommand{\Hom}{\operatorname{Hom}}

\newcommand{\can}{{{\rm can}}}
\newcommand{\ord}{{\operatorname{ord}}}

\newcommand{\val}{{\operatorname{val}}}

\newcommand{\an}{{\operatorname{an}}}

\newcommand{\hooklongrightarrow}{\lhook\joinrel\longrightarrow}

\newcommand{\abs}{{\operatorname{abs}}}
\newcommand{\ess}{{\operatorname{ess}}}

\newcommand{\KK}{\operatorname{K}}
\newcommand{\hphi}{\widehat{\phi}}
\newcommand{\Lip}{\operatorname{Lip}}

\def \A{\mathbb{A}}

\def \C{\mathbb{C}}
\def \F{\mathbb{F}}
\def \G{\mathbb{G}}

\def \K{\mathbb{K}}

\def \P{\mathbb{P}}
\def \Q{\mathbb{Q}}
\def \R{\mathbb{R}}
\def \SS{\mathbb{S}}

\def \T{\mathbb{T}}
\def \Z{\mathbb{Z}}

\def\cC {{\mathcal C}}
\def\cD {{\mathcal D}}
\def\cE {{\mathcal E}}

\def\cH {{\mathcal H}}

\def\cM {{\mathcal M}}
\def\fM {{\mathfrak M}}

\def\cO {{\mathcal O}}
\def\cP {{\mathcal P}}

\def\cX {{\mathcal X}}

\newcommand{\bfu}{{\boldsymbol{u}}}

\newcommand{\bfA}{{\boldsymbol{A}}}
\newcommand{\bfE}{{\boldsymbol{E}}}
\newcommand{\bfU}{{\boldsymbol{U}}}

\newcommand{\bfdelta}{{\boldsymbol{\delta}}}

\newcommand{\bfnu}{{\boldsymbol{\nu}}}

\newcommand{\bfzero}{\boldsymbol{0}}

\numberwithin{equation}{section}

\theoremstyle{definition}
\newtheorem{defn}{Definition}
\numberwithin{defn}{section}
\newtheorem{notn}[defn]{Notation}
\newtheorem{rem}[defn]{Remark}
\newtheorem{exmpl}[defn]{Example}

\theoremstyle{plain}
\newtheorem{lem}[defn]{Lemma}

\newtheorem{prop}[defn]{Proposition}
\newtheorem{thm}[defn]{Theorem}
\newtheorem{cor}[defn]{Corollary}

\newtheorem{prop-def}[defn]{Proposition-Definition}

\begin{document}

\title[Small points in toric varieties]{The distribution of Galois
  orbits of points of small height in toric varieties}

\author[Burgos Gil]{Jos\'e Ignacio Burgos Gil}
\address{Instituto de Ciencias Matem\'aticas (CSIC-UAM-UCM-UCM3).
  Calle Nicol\'as Ca\-bre\-ra~15, Campus UAB, Cantoblanco, 28049 Madrid,
  Spain} 
\email{burgos@icmat.es}
\urladdr{\url{http://www.icmat.es/miembros/burgos}}
\author[Philippon]{Patrice Philippon}
\address{Institut de Math{\'e}matiques de
Jussieu -- U.M.R. 7586 du CNRS, \'Equipe de Th\'eorie des Nombres.
BP 247, 4
place Jussieu, 75005 Paris, France}
\email{patrice.philippon@imj-prg.fr}
\urladdr{\url{http://www.math.jussieu.fr/~pph}}
\author[Rivera-Letelier]{Juan Rivera-Letelier}
\address{Department of Mathematics, University of Rochester. Hylan
  Building, Rochester, NY 14627, U.S.A.} 
\email{riveraletelier@gmail.com}
\urladdr{\url{http://rivera-letelier.org}}
\author[Sombra]{Mart{\'\i}n~Sombra}
\address{ICREA \& Departament d'{\`A}lgebra i Geometria, Universitat
  de Barcelona. 
Gran Via 585, 08007 Bar\-ce\-lo\-na, Spain}
\email{sombra@ub.edu}
\urladdr{\url{http://atlas.mat.ub.es/personals/sombra}}

\date{\today} 
\thanks{Burgos was partially supported by the
  MINECO research project MTM2013-42135-P.  Philippon was partially
  supported by the CNRS research project PICS 6381 ``G\'eom\'etrie
  diophantienne et calcul formel'' and the ANR research project
  ``Hauteurs, modularit\'e, transcendance''. Rivera-Letelier was partially
  supported by FONDECYT grant 1141091. Sombra was partially
  supported by the MINECO research project MTM2012-38122-C03-02.}

\begin{abstract}
In this text, we address the distribution properties of points of small
height on proper toric varieties and applications to the related
Bogomolov property. We introduce the notion of monocritical toric
metrized divisor and  we prove that equidistribution occurs for every
generic, small 
sequence with respect to a toric metrized divisor, for every
place if and only if the divisor is monocritical. Furthermore, when
this is the case, the limit measure is a translate of the natural
measure on the compact torus sitting in the principal orbit of the
ambient toric variety. 

We also study the $v$-adic modulus distribution of Galois orbits of small
points. We
characterize, in terms of the given toric semipositive metrized
divisor, the cluster measures of $v$-adic 
valuations of  Galois orbits of generic small sequences.  

The Bogomolov property now says that a subvariety of the principal
orbit of a proper toric
variety that has the same essential minimum than the toric variety
with respect to a monocritical toric metrized divisor, must be a
translate of a subtorus. We also give several examples, including a
non-monocritical
divisor for which the Bogomolov property does not hold.
\end{abstract}

\vspace*{-8mm}

\maketitle

\setcounter{tocdepth}{1}
\tableofcontents

\vspace*{-10mm}

\section[Introduction]{Introduction}

The study of the limit distribution of Galois orbits of points of
small height was initiated by Szpiro, Ullmo and Zhang in their seminal
paper \cite{SzpiroUllmoZhang:epp}. They proved the equidistribution of
the Galois orbits of sequences of points in an abelian variety over a
number field, with N\'eron-Tate height converging to zero, over the
Archimedean places.  It was motivated by the Bogomolov conjecture on
abelian varieties, and eventually applied in the affirmative answer to
this conjecture by Ullmo \cite{Ullmo:pdpa} and Zhang
\cite{Zhang:esab}, see also~\cite{Cin11,Ghioca:pshvdff,Gub07,Yam13}
for similar results in the function field case.

This equidistribution result was widely generalized, in particular to
other varieties and other height functions and, with the
introduction of Berkovich spaces, to the equidistribution over
non-Archimedean places \cite{Bilu:ldspat, ChambertLoir:pph,
  FavreRivera:eqpph, 
  ChambertLoir:meeB, BakerRumely:esp, Yuan:blbav,
  BermanBoucksom:gbhsee, Chen:davf}.  We next introduce the necessary
background to explain this generalization.

Let $\K$ be a global field, that
is, a field which is either a number field or the function field of a
regular projective curve over an arbitrary field, and $\fM_{\K}$ its
set of places.  Let $X$ be a proper algebraic variety over $\K$ of
dimension $n$, and $\ov D=(D,(\|\cdot\|_{v})_{v\in \fM_{\K}})$ a semipositive
metrized (Cartier) divisor with $D$ big.  Let
\begin{displaymath}
  \h_{\ov D}\colon X(\ov \K)\longrightarrow \R
\end{displaymath}
be the associated height function on the set of algebraic points of
$X$, see \S\ref{sec:preliminaries} for details. It is a generalization
of the notion of height of algebraic points considered by Weil,
Northcott and others.

The \emph{essential minimum} of $X$ with respect to $\ov D$, denoted
by $\upmu^{\ess}_{\ov D}(X)$, is the smallest possible limit value of
the height of a generic net of algebraic points of $X$.  Consequently,
we say that a net $(p_{l})_{l\in I}$ is \emph{$\ov D$-small} if
\begin{displaymath}
  \lim_{l}\h_{\ov D}(p_{l})=\upmu^{\ess}_{\ov D}(X).
\end{displaymath}

A fundamental inequality by Zhang \cite{Zhang:plbas} shows that the
essential minimum can be bounded below in terms of the height and the
degree of $\ov D$:
\begin{equation} \label{eq:72}
  \upmu^{\ess}_{\ov D}(X)\ge \frac{\h_{\ov D}(X)}{(n+1)\deg_{D}(D)}.
\end{equation}
We say that $\ov D$ is \emph{quasi-canonical} if this lower bound for
the essential minimum is an equality (Definition~\ref{def:1}).
Examples of quasi-canonical metrized divisors are given by the
canonical metrics on divisors of toric and abelian varieties and, more
generally, by 
the metrics coming from algebraic dynamical systems.

For a place $v\in \mathfrak{M}_{\K}$, we denote by $X_{v}^{\an}$ the
$v$-adic analytification of $X$.  If $v$ is Archimedean, it is a
complex space whereas, if $v$ is non-Archimedean, it is a Berkovich
space over $\C_{v}$, the completion of the algebraic closure of the
local field~$\K_{v}$.  We endow the space of probability measures on
$X_{v}^{\an}$ with the weak-$\ast$ topology with respect
to the space of  continuous functions on~$X_{v}^{\an}$.

For an algebraic point $p$ of $X$, we denote by $\Gal(p)_{v}$ its
$v$-adic Galois orbit, that is, the orbit of $p$ in
$X^{\an}_{v}$ under the action of the absolute Galois group of
$\K$. We set
\begin{equation*}
 \mu_{p,v}   =\frac{1}{\#\Gal(p)_{v}} \sum_{q\in \Gal(p)_{v}}\delta_{q}
\end{equation*}
for the uniform probability measure on $\Gal(p)_{v}$. We also denote
by $\chern_{1}(D,\|\cdot\|_{v})^{\wedge n}$ the {$v$-adic
  Monge-Amp\`ere measure} of $\ov D$, see for instance \cite[\S
1.4]{BurgosPhilipponSombra:agtvmmh}. It is a measure on $X^{\an}_{v}$
of total mass $\deg_{D}(X)$.

The following statement is representative of several equidistribution
theorems for Galois orbits of small points in the literature. In this
form, it is due to Yuan \cite[Theorem 3.1]{Yuan:blbav} for number
fields and to Gubler \cite[Theorem 1.1]{Gubler:Eff} for
function fields.

\begin{thm} \label{thm:9} Let $X$ a projective variety over $\K$ of
  dimension $n$, and $\ov D$ a quasi-canonical semipositive metrized
  divisor on $X$ with $D$ ample.  Let $(p_{l})_{l\in I}$ be a
  generic $\ov D$-small net of algebraic points of $X$. Then, for every $v\in
  \fM_{\K}$,  the
  net of probability measures $(\mu_{p_{l},v})_{l\in I}$ converges
  to $ \frac{1}{\deg_{D}(X)}\chern_{1}(D,\|\cdot\|_{v})^{\wedge n}$, 
    the normalized $v$-adic Monge-Amp\`ere measure of $\ov D$.
\end{thm}

A common feature of this result and its variants and generalizations,
is the assumption that the lower bound \eqref{eq:72} is an equality
or, in other words, that the metrized divisor $\ov D$ is
quasi-canonical. This severely restricts their range of
application. Nonetheless, these results do apply to the important case
of metrics arising from algebraic dynamical systems and moreover, they
have a very strong thesis: not only the Galois orbits of points of
small height do converge, but the limit measure is given by the normalized
$v$-adic Monge-Amp\`ere measure.

The motivation of this paper is to start the study what happens when
we remove the hypothesis that $\ov D$ is quasi-canonical. Some of our
typical questions are: is there always an equidistribution phenomenon
for Galois orbits of $\ov D$-small points? If not, can we give
conditions on $\ov D$, beyond being quasi-canonical, under which such
a phenomenon occurs?  When equidistribution occurs, can we describe
the limit measure?

We address these questions and some of its continuations in
the toric setting. 
Our approach is based on the techniques developed in the series
\cite{BurgosPhilipponSombra:agtvmmh,BurgosMoriwakiPhilipponSombra:aptv,BurgosPhilipponSombra:smthf}. These 
techniques are well-suited for the study of toric metrics and their
associated height functions. In the sequel, we recall the basic
constructions.

Let $X$ be a proper toric variety over $\K$ of dimension $n$, given by
a complete fan~$\Sigma$ on a vector space $N_{\R}\simeq \R^{n}$, and a
big toric divisor $D$ on $X$, given by a virtual support function
$\Psi_{D}\colon N_{\R}\to \R $.  This toric divisor also defines an
$n$-dimensional polytope~$\Delta_{D}$ in the dual space
$M_{\R}\coloneqq N_{\R}^{\vee}$.

Let $\ov D=(D,(\|\cdot\|_{v})_{v\in \fM_{\K}})$ be a semipositive
toric metrized divisor on $X$. To it we associate 
an adelic family of concave functions $\psi_{\ov D, v}\colon N_{\R}\to
\R$, $v\in \mathfrak{M}_{\K}$, called the \emph{metric
  functions} of~$\ov D$.  They satisfy that $|\psi_{\ov D,v}-\Psi_{D}|$
is bounded on $N_{\R}$ for all~$v$, and that $\psi_{\ov D,v}=\Psi_{D}$ for
all~$v$ except for a finite number. We also associate to $\ov D$ an
adelic family of
continuous concave functions on the polytope $\vartheta_{\ov D,
  v}\colon \Delta_{D}\to \R$, $v\in \mathfrak{M}_{\K}$, called the
\emph{local roof functions} of $\ov D$. They verify that
$\vartheta_{\ov D,v}$ is the zero function for all $v$ except for a
finite number.  The \emph{global roof function} is a concave function
$\vartheta_{\ov D}\colon \Delta _{D}\to \R$ defined as a weighted
sum of the local roof functions.

The metric functions and the roof functions convey lot of information
of the pair $(X,\ov D)$. For instance, the essential minimum of $X$
with respect to $\ov D$ can be computed as the maximum of the global
roof function \cite[Theorem~1.1]{BurgosPhilipponSombra:smthf}:
\begin{equation} \label{eq:35}
  \upmu^{\ess}_{\ov D}(X)= \max_{x\in \Delta_{D}}\vartheta_{\ov
    D}(x). 
\end{equation}

In the toric setting, the condition that the metrized divisor $\ov D$
is quasi-canonical is very restrictive, since it is equivalent to the 
condition that its global roof function is constant (Proposition
\ref{prop:10}).  Thus, the only toric metrics to which the
equidistribution theorem~\ref{thm:9} applies are those whose global
roof function is constant.

To identify the toric metrics having good equidistribution properties,
we introduce the notion of \emph{monocritical} toric metrized
divisor. The semipositive toric metrized divisor  $\ov D$ is
monocritical if a certain function on a
space of measures attains its infimum at a unique measure (Definition
\ref{def:16}). When this
is the case, the minimizing measure determines a point
\begin{displaymath}
  \bfu=(u_{v})_{v\in \fM_{\K}}\in \bigoplus_{v\in \fM_{\K}}N_{\R}
\end{displaymath}
with $\sum_{v}n_vu_v=0$, called the \emph{critical point} of $\ov D$. 

The condition for $\ov D$ of being
monocritical can be characterized in terms of its global roof
function: given a point $x_{\max}\in \Delta_{D}$ maximizing
$\vartheta_{\ov D}$, the sup-differential $\partial \vartheta_{\ov
  D}({x_{\max}}) $ is a convex subset of $N_{\R}$ containing the point
$0$. Then $\ov D$ is monocritical if and only if $0$ is a vertex of
this convex subset and, when this is the case, the critical point of $\ov
D$ can be computed from the sup-differential of the local roof
functions at~$x_{\max}$ (Proposition~\ref{prop:14}).

Let $\T\simeq \G_{{\rm m}, \K}^{n}$ be the torus of $X$, which can be
identified with $X_{0}$, the principal open subset of $X$.  For each
$v\in \fM_{\K}$, we denote by $\SS_{v}$ the compact subtorus
of~$\T^{\an}_{v}$.  To a monocritical toric metrized divisor $\ov D$
with critical point $ \bfu\in \bigoplus_{v\in \fM_{\K}} N_{\R} $, we
associate a probability measure $\lambda_{\SS_{v},u_{v}} $ on
$X_{{v}}^{\an}$ (Definition~\ref{def:10}). When $v$ is Archimedean, it
is the uniform measure on a translate of $\SS_{v}\simeq (S^{1})^{n}$
whereas, when $v$ is non-Archimedean, it is the Dirac measure at a
translate of the Gauss point of $\T^{\an}_{v}$.

The following is the main result of this paper (Theorem~\ref{thm:5}). 

\begin{thm} \label{thm:10} Let $X$ be a proper toric variety over $\K$
  and $\overline D$ a semipositive toric metrized divisor on $X$ with
  $D$ big.  Then $\ov D$ is monocritical if and only if for every
  place $v\in \fM_{\K}$ and every
  generic $\ov D$-small net $(p_{l})_{l\in I}$ of algebraic points of
  $X_{0}$, the net of probability measures
  $(\mu_{p_l,v})_{l\in I}$ on $X_{{v}}^{\an}$ converges. 

  When this is the case, the limit measure agrees with $\lambda
  _{\SS_{v},u_{v}}$, where $u_{v}\in N_{\R}$ is the $v$-adic component
  of the critical point of $\ov D$.
\end{thm}

Quasi-canonical toric metrized divisors are monocritical, and
Theorem~\ref{thm:10} reduces to Theorem~\ref{thm:9} in this case.  Our
result produces a wealth of new examples of metric satisfying the
equidistribution property, that were not covered by the previous
results. In particular, this is the case for toric divisors over a
number field~$\K$ with positive smooth metrics at the Archimedean
places and canonical metrics at the non-Archimedean ones (Theorem
\ref{thm:8}). Here we state a simplified version for the case when
$\K=\Q$.

  \begin{cor}
    \label{cor:1}
    Let $X$ be a proper toric variety over $\Q$ and $\ov D$ a
    semipositive toric metrized $\R$-divisor with $D$ big.  We assume
    that the $v$-adic metric of $\ov D$ is, when $v$ is the Archimedean
    place, smooth and positive and, when $v$ is non-Archimedean, equal to the
    $v$-adic canonical metric of $D$.  Then, for every generic $\ov
    D$-small sequence  $(p_{l})_{l\ge 1}$ of algebraic points of $X_{0}$ and
    every place $v\in \fM_{\Q}$, the sequence $(\mu_{p_l,v})_{l\ge 1}$ on
    $X_{{v}}^{\an}$ converges to the probability measure $\lambda
    _{\SS_{v},0}$.
  \end{cor}

  This corollary covers many typical examples of metrics on toric
  varieties like weighted projective spaces and toric bundles, see
  \S~\ref{sec:posit-arch-metr}.  For instance, let $X=\P^{1}_{\Q}$ and
  $\ov D$ the divisor of the point at infinity equipped with the
  Fubini-Study metric at the Archimedean place and the canonical
  metric at the non-Archimedean places. Its essential minimum is
  \begin{displaymath}
    \upmu^{\ess}_{\ov D}(X)=\frac{\log(2)}{2}
  \end{displaymath}
  and, for every generic sequence of algebraic points of $\P^{1}_{\Q}$
  with height converging to this quantity, its $\infty$-adic Galois
  orbits converge to the Haar probability measure on $S^{1}$, the unit
  circle of the Riemann sphere (Example~\ref{exm:11}). This an example
  where equidistribution does occur, but the limit measure is
  \emph{not} given by the $v$-adic Monge-Amp\`ere measure as in
  Theorem~\ref{thm:9}.

  In the other extreme, classical examples of translates of subtori
  with the canonical metric can behave badly with respect to
  equidistribution. For instance let $X$ be the line of $\P^{2}_{\Q}$
  of equation $2z_{1}-z_{2}=0$ and $\ov D$ the metrized divisor on $X$
  given by the restriction of the canonical metrized divisor at
  infinity of $\P^{2}_{\Q}$. As explained in Example~\ref{exm:5},
  Theorem~\ref{thm:10} implies that $\ov D$ does not satisfy the
  equidistribution property in the sense of Definition~\ref{def:4}.

We also study the modulus distribution of the $v$-adic Galois
orbits of $\ov D$-small nets of algebraic points in the general, non
necessarily monocritical, toric case. 
There
is a valuation map $\val_{v}\colon \T_{v}^{\an}\to N_{\R}$, defined,
in any given splitting, by
  \begin{equation*}
    \val_{v}(x_{1},\dots,x_{n})=(-\log|x_{1}|_{v},\dots,-\log|x_{n}|_{v}).
  \end{equation*}
  Hence, for an algebraic point $p$ of $X_{0}=\T$, the direct image
  measure $$\nu_{p,v}\coloneqq(\val_{v})_* \mu_{p,v}$$ is a probability measure on
  $N_{\R}$ that gives the modulus distribution of its $v$-adic Galois
  orbit. 

To each semipositive toric
metrized divisor $\ov D$ with $D$ big, we associate an adelic family
of nonempty subsets of $N_{\R}$
\begin{equation}
  \label{eq:74}
( B_{v}, F_{v})_{v\in
\fM_{\K}},
\end{equation}
with $B_{v}\subset F_{v}$
(Notation \ref{def:14}). We endow the space of probability measures on
$N_{\R}$ with the weak-$\ast$ topology with respect to the bounded
continuous functions on $N_{\R}$. For a probability measure $\nu$ on
$N_{\R}$, we denote by $\supp(\nu)\subset N_{\R}$ its support and,
if~$\nu$ has finite first moment, we denote by $\exv[\nu]$ its
expected value.

The next result characterizes the limit behavior of the modulus
distribution for $\ov D$-small nets (Theorem~\ref{thm:1 bis} and
Corollary~\ref{cor:2}). 

\begin{thm} \label{thm:11} Let $X$ be a proper toric variety over
  $\K$,  $\overline D$ a semipositive toric metrized divisor on $X$ with
  $D$ big, and $v\in \mathfrak{M}_{\K}$. For every $\ov{D}$-small
  net~$(p_{l})_{l \in I}$ of algebraic points in~$X_{0}$, the net of
  probability measures~$(\nu_{p_l,v})_{l \in I}$ has at least one
  cluster point. Every such cluster point is a measure~$\nu_{v}$ with
  finite first moment that satisfies
\begin{equation} \label{eq:75}
  \supp(\nu_{v}) \subset F_v
\quad \text{ and } \quad
\exv[\nu_{v}] \in B_v.
\end{equation}
Conversely, for every probability measure $\nu_{v}$ on $N_{\R}$ which
has finite first moment and satisfies \eqref{eq:75}, there is a $\ov{D}$-small net~$(p_{l})_{l
  \in I}$ of algebraic points of~$X_{0}$ such that $\nu_{v}$ is the limit of the
net~$(\nu_{p_{l},v})_{l\in I}$. 
\end{thm}

In the situation
of Theorem~\ref{thm:11}, when $F_{v}$ consist of only one
point $u_v$, the net
$(\nu_{p_l,v})_{l \in I}$  converges to the measure $\delta
_{u_{v}}$. In this case we say that $\ov D$ satisfies the \emph{modulus
  concentration property} at the 
place $v$.

In fact, $\ov D$ is monocritical if and only if for every place $v$,
the set $F_{v}$ consists of only one point.  Hence, our results imply
that $\ov D$ satisfies the equidistribution property at every place if
and only if it satisfies the modulus concentration property at every
place.  Observe also that, in contrast with Theorem \ref{thm:10}, the
net of algebraic points in Theorem \ref{thm:11} does not need to be
generic.
 

In the absence of modulus concentration, there is a wealth of limit
measures of $v$-adic Galois orbits of $\ov D$-small nets of algebraic
points.  For instance, consider the projective line over a number
field $\K$ and any adelic set $\bfE=(E_{v})_{v\in\fM_{\K}} $ of global
capacity 1, whose associated equilibrium measures are compatible with
the collection of sets in \eqref{eq:74} (see Theorem~\ref{thm:13} for
the precise condition). Using Rumely's
Fekete-Szeg\H{o} theorem \cite{Rumely:fstsc}, we show that, for all
$v$, the equilibrium measure of $E_{v}$ can be realized as the limit
measure of a sequence of $v$-adic Galois orbits of $\ov D$-small
points (Theorem~\ref{thm:13}).


As we already mentioned, the original motivation in
\cite{SzpiroUllmoZhang:epp} to search for equidistribution results of
Galois orbits of small points was to prove the Bogomolov conjecture. 
The Bogomolov conjecture for toric varieties can be stated as follows:
let~$X$ be a 
toric variety over $\K$ and $\ov D^{\can}$ an ample toric divisor on $X$
equipped with the canonical metric. Let $V\subset X_{0,\ov
  \K}$ be a subvariety which is not a translate of a subtorus by a
torsion point. Then there exists $\varepsilon >0$ such that the subset
of algebraic points of $V$ of canonical height bounded above by
$\varepsilon$, is not dense in $V$. Equivalently, if $V\subset X_{0,\ov
  \K}$ is a subvariety with  $\upmu^{\ess}_{\ov D^{\can
  }}(V)=0$, then $V$ is a translate of a subtorus by a torsion point.
It is the  toric counterpart of the Bogomolov conjecture for
abelian varieties proved by Ullmo and Zhang.  

This conjecture was proved by Zhang \cite{Zhang:plbas} for number
fields, and later Bilu gave a proof using his equidistribution
theorem~\cite{Bilu:ldspat}. 
Here we extend this result to an
arbitrary monocritical metric on a toric variety over a number field
(Theorem~\ref{thm:2}). 

\begin{thm}\label{thm:12}
  Let $X$ be a proper toric variety over a number field $\K$ and
  $\overline D$ a monocritical toric metrized divisor on $X$ with
  critical point $\bfu=(u_{v})_{v\in\fM_{\K}}$. Let
  $V$ be a subvariety of $X_{0,\ov \K}$ with
\begin{displaymath}
  \upmu^{\ess}_{\ov D}(V) =  \upmu^{\ess}_{\ov D}(X). 
\end{displaymath}
Then $V$ is a translate of a subtorus.  Furthermore, if $ u_v\in
\val_{v}(\T(\K))\otimes \Q$ for all~$v$, then $V$ is the translate of a
subtorus by an algebraic point $p$ of $X_{0}$ with $\h_{\ov D}(p)=
\upmu^{\ess}_{\ov D}(X)$.
\end{thm}

A subvariety of $X_{0,\ov \K}$ with
\begin{displaymath}
  \upmu^{\ess}_{\ov D}(V) =  \upmu^{\ess}_{\ov D}(X) 
\end{displaymath}
is called a $\ov D$-special subvariety.
We say that a given toric metrized divisor $\ov D$ satisfies the
\emph{Bogomolov property}\footnote{not to be confused with the
  property (B) introduced by Bombieri and Zannier, and studied by
  Amoroso, David and other authors.} if every $\ov D$-special
subvariety is a translate of
  a subtorus (Definition~\ref{def:15}).  This property is 
intimately related with the equidistribution property. Indeed, we give
an example of a metrized divisor $\ov D$ on $\P^{2}_{\Q}$, such that the
line of equation $z_{0}+z_{1}+z_{2}=0$ is $\ov D$-special
(Example~\ref{exm:10}).  This line is certainly not a 
translate of a subtorus, and so $\ov D$ does not satisfy the Bogomolov
property.  This metrized divisor is a variant of the one in Example
\ref{exm:5}, and does not verify modulus concentration nor
equidistribution for any place of $\Q$.

We comment briefly on the techniques of proof. 
To prove Theorem~\ref{thm:11}, for each place $v$ we construct an
upper-semicontinuous concave functional $\Phi_{v}$ on the space of
probability measures with finite first moment on $N_{\R}$. Using the
toric dictionary from \cite{BurgosPhilipponSombra:agtvmmh,
  BurgosMoriwakiPhilipponSombra:aptv, BurgosPhilipponSombra:smthf}, we
reduce the study of the modulus distribution of $v$-adic Galois orbits
to the optimization of this functional, which we treat by applying
Prokorov's compactness theorem.

Once we know that a monocritical divisor $\ov D$ satisfies modulus
concentration at every place, we can construct another toric metric
on~$D$ that is quasi-canonical and such that the $\ov D$-small points
are also small with respect to this new metric. With this trick, the
toric equidistribution theorem~\ref{thm:10} is derived from
Theorem~\ref{thm:9}. 

The Bogomolov property for monocritical metrics (Theorem~\ref{thm:12})
follows from Theorem 
\ref{thm:7}, a variant of the toric equidistribution theorem for $\ov
D$-small nets that are not necessarily generic but only \emph{strict},
in the sense that they eventually avoid any fixed {translate of
  subtorus}.

These results arise several interesting questions.  For instance: is
it possible that a given semipositive toric metrized divisor $\ov D$
satisfies the equidistribution property at one place and not at
another? We study this for the projective line showing that, under a
natural rationality hypothesis, the equidistribution property holds at
a given place if and only if it holds at every place (Proposition
\ref{prop:8}). However, this conclusion is not true without this
rationality hypothesis (Remark~\ref{rem:3}) and we have neither
settled this question for the projective line in full generality, nor
treated toric varieties of higher dimension.

It would also be interesting to see if the converse of Theorem
\ref{thm:12} holds: Let $X$ be  proper toric variety with $\dim X\ge
2$. Given a semipositive toric metrized divisor $\ov
D$ on X,
with $D$ big satisfying the Bogomolov property, is $\ov D$ necessarily 
monocritical? In Proposition~\ref{prop:2} we show that this is true in
a very particular case.  Extending this to the general case would
reinforce the link between the equidistribution and the Bogomolov
properties.

The results of this paper also inspire questions for general varieties
and metrized divisors.  For instance, from Corollary~\ref{cor:1}, it
is plausible to conjecture that a toric divisor equipped with a
positive smooth, but not necessarily toric, Archimedean metric and
canonical non-Archimedean metrics, does satisfy the equidistribution
property.  A puzzling question is that of computing the essential
minimum, with a formula generalizing \eqref{eq:35} to the general,
non-toric, case.  Even more challenging seems the problem of
generalizing the crucial notion of monocritical metrized divisor.

Several of the results presented in this introduction hold in greater
generality and their thesis are stronger. We refer to the body of the paper
for these versions.  The structure of the paper is as follows. In
\S\ref{sec:preliminaries} we give the preliminaries on Galois orbits
and height of points. In \S\ref{sec:auxil-results-conv} we introduce
the upper semi-continuous concave functional $\Phi_{v}$ and study its
properties. In \S\ref{sec:modul-distr-toric} 
we study the modulus distribution of $v$-adic Galois orbits of $\ov
D$-small nets of points in toric varieties. In
\S\ref{sec:equid-galo-orbits} we prove the toric equidistribution
theorem~\ref{thm:10} and its variants, together with the Bogomolov
property for monocritical toric metrized divisors. In
\S\ref{sec:examples} we give examples illustrating a number of
phenomena, including a non-monocritical toric metrized divisor not
verifying the Bogomolov property.  Finally, in
\S\ref{sec:potent-theory-proj} we use potential theory to study the
limit measures that appear in the absence of modulus concentration.

\medskip \noindent {\bf Acknowledgments.} We thank Qing Liu, Ricardo Menares,
Joaquim Ortega-Cerd\'a and Tom Tucker for useful discussions and
pointers to the literature.

Part of this work was done while the authors met at the Universitat de
Barcelona, the Instituto de Ciencias Matem\'aticas (Madrid), the
Institut de Math\'ematiques de Jussieu (Paris),  the Universidad
Cat\'olica de Chile (Santiago) and the Beijing International Center
for Mathematical Research. 

\section{Galois orbits, height of points and essential
  minimum}\label{sec:preliminaries}

By a \emph{global field} $\K$ we mean either a number field or the
function field of a regular projective curve over an arbitrary field,
equipped with a certain set of places, denoted by $\mathfrak{M}_{\K}$.
Each place $v\in \mathfrak{M}_{\K}$ is a pair consisting of an
absolute value $|\cdot|_{v}$ on $\K$ and a positive weight $n_{v}\in
\Q_{> 0}$, defined as follows.

The places of the field of rational numbers $\Q$ consist of the
Archimedean and the $p$-adic absolutes values, normalized in the
standard way, and with all weights equal to 1. For the function field
$\KK(C)$ of a regular projective curve $C$ over a  field $k$,
the set of places is indexed by the closed points of $C$. For each
closed point $v_{0}\in C$, we consider the
absolute value and weight given, for $\alpha\in \KK(C)^{\times}$, by
\begin{displaymath}
  |\alpha|_{v_{0}}=c_{k}^{-\ord_{v_{0}}(\alpha)}\and  n_{v_{0}}=[k(v_{0}):k], 
\end{displaymath}
with $c_{k}=\# k$ if the base field $k$ is finite and $c_{k}=\e$
otherwise, and where $\ord_{v_{0}}(\alpha)$ denotes the order of
$\alpha$ in the discrete valuation ring $\cO_{C,v_{0}}$.

Let $\K_{0}$ denote either $\Q$ or $\KK(C)$. In the general case when
$\K$ is a finite extension of $\K_{0}$, the set of places of $\K$ is
formed by the pairs $v=(|\cdot|_{v},n_{v})$ where $|\cdot|_{v}$ is an
absolute value on $\K$ extending an absolute value $|\cdot|_{v_{0}}$
with $v_{0}\in \fM_{\K_{0}}$ and
\begin{equation}\label{eq:100}
  n_{v}=\frac{[\K_{v}:\K_{0,v_{0}}]}{[\K:\K_{0}]} n_{v_{0}},
\end{equation}
where $\K_{v}$ denotes the completion of $\K$  with respect to
$|\cdot|_{v}$, and similarly for $\K_{0,v_{0}}$.  
This set of places satisfies the following  basic properties. 

\begin{prop}
  \label{prop:15}
Let $\K_{0}$ denote either $\Q$ or $\KK(C)$, the function field of a regular
projective curve $C$ over a field $k$. Let 
$\K$ be a finite extension of $\K_{0}$ and $\fM_{\K}$ the associated set
of places as above. Then
\begin{enumerate}
\item \label{item:19} for all $v_{0}\in \fM_{\K_{0}}$, we have $
  \sum_{v\mid v_{0}} n_v = n_{v_0}$; 
\item  \label{item:20}
for all $\alpha\in\K^{\times}$, we have $  \sum_{v\in \fM_{\K}}
n_{v}\log|\alpha|_{v}=0$ \emph{(product formula)}.
\end{enumerate}
\end{prop}

\begin{proof}
  These properties are classical, see for instance \cite[Theorems 2
  and 3]{AW:acpf}.

  In the function field case there is a subtlety, due to the fact that a given
  field may have different structures of global field depending on the choice of
  base curve.

  Let $C$ be a regular projective curve over $k$ and
  $K(C)\hookrightarrow \K$ a finite extension of fields. Then 
  there is a regular projective curve $B$ over $k$ and a finite morphism
  $\pi\colon B\to C$ such that $\K\simeq K(B)$ and the previous
  extension can be identified with $\pi ^{\ast}\colon
  K(C)\hookrightarrow K(B)$, see for instance \cite[Proposition~7.3.13
  and Lemma~7.3.10]{Liu:agac}.

  We could give to $\K$ the structure of global field defined directly
  by the curve $B$, but the obtained absolute values of $\K$ would not
  be extensions of those of $\K_{0}$. To remedy this, we renormalize
  these absolute values of $\K$ and, to preserve the product formula,
  we also change the weights.

  From the valuative criterium of properness, for each closed point
  $v_{0}\in C$,
  the absolute values of $\K$ extending $|\cdot|_{v_{0}}$ are in
  bijection with the closed points of the fiber above $v_0$. 
  Moreover,  since the map $\pi$ is
  finite, for each closed point $v\in \pi^{-1}(v_{0})$, the ring
  $\cO_{B,v}$ is a finite module over $\cO_{C,v_{0}}$. 
  It follows from \cite[Chapitre
  6, Proposition 2  in  \S8.2 and Theorem 2 in \S8.5]{Bourbaki:ACvi}
  that the absolute value and weight corresponding to $v$ are given, for
  $\alpha\in \KK(B)^{\times}$, by 
  \begin{equation} \label{eq:23}
    |\alpha|_v=c_{k}^{-\frac{\ord_v(\alpha)}{e_{v/v_{0}}}}, 
    \quad n_{v}= \frac{e_{v/v_{0}}[k(v):k]}{[\KK(B):\KK(C)]},
  \end{equation}
with  $e_{v/v_{0}}$  the ramification index of $v$ over $v_{0}$. The
same results in \emph{loc. cit.} give the  formula in \eqref{item:19}. 

For the product formula in \eqref{item:20}, we obtain from
\eqref{eq:23} that
\begin{displaymath}
    \sum_{v}
n_{v}\log|\alpha|_{v}= -\log(c_{k})  \sum_{v}
\frac{[k(v):k]\ord_v(\alpha) }{[\KK(B):\KK(C)]} = 
\frac{-\log(c_{k})}{[\KK(B):\KK(C)]} \deg(\div(\alpha))=0,
\end{displaymath}
because the degree of a principal divisor on $B$ is zero, which concludes the
proof.
\end{proof}
 
For $v\in \fM_{\K}$, we choose an algebraic closure $\K_{v}\subset \ov
\K_{v}$ of $\K_{v}$. The absolute value~$|\cdot|_{v}$ on $\K_{v}$ has
a unique extension to $\ov \K_{v}$.  We denote
by $\C_{v}$ the completion of $\ov \K_{v}$ with respect to this extended
absolute value. We also choose an algebraic closure $\ov \K$ of $\K$
and an embedding  $\jmath_{v}\colon \ov \K \to \C_{v}$.  

Let $X$ be a variety over $\K$, that is, a
reduced and irreducible separated scheme of finite type over $\K$.
The elements of $X(\ov \K)$ are called the \emph{algebraic points} of
$X$. For $p\in X(\ov \K)$, its \emph{Galois orbit} is
$\Gal(p)\coloneqq \Gal(\ov
\K/\K)\cdot p\subset X(\ov \K)$, that is, the orbit of $p$ under the
action of the 
absolute Galois group of $\K$.

For $v\in \mathfrak{M}_{\K}$, we denote by $X_{\K_{v}}^{\an}$ the
$v$-adic analytifications of $X$ over $\K_{v}$ and by 
$X^{\an}_{v}$ the $v$-adic analytifications of $X$ over $\C_{v}$. If
$v$ is Archimedean, they both coincide
with a complex space ($X_{\K_{v}}$ is equipped with an anti-linear
involution if~$\K_{v}\simeq \R$). If $v$ is non-Archimedean, they are 
Berkovich spaces over $\K_{v}$ and $\C_{v}$, respectively.  These
spaces are related by (\cite[Corollary 1.3.6]{Berkovich:stag})
\begin{displaymath}
 X_{\K_{v}}^{\an}=X^{\an}_{v}/\Gal(\overline \K_{v}/\K_{v}).
\end{displaymath}
We denote by
\begin{equation}
  \pi_{v}\colon X^{\an}_{v}\to X_{\K_{v}}^{\an} 
\end{equation}
the projection.

There is a map  
\begin{displaymath}
  X(\C_{v}){\hooklongrightarrow}
X_{{v}}^{\an}.
\end{displaymath}
Using the chosen inclusion $ \jmath_{v}\colon \ov
\K\hookrightarrow \C_{v}$, we obtain a map $X(\ov
\K)\hookrightarrow X(\C_{v})$ and, by composition the previous
map, an inclusion
\begin{equation*}
\iota_{v}\colon X(\ov
\K){\hooklongrightarrow} X_{v}^{\an}.  
\end{equation*}

The \emph{v-adic Galois orbit} of an algebraic point $p\in X(\ov \K)$,
denoted by $\Gal(p)_{v}$, is defined as the image of $\Gal(\ov
\K/\K)\cdot p$
under $\iota_{v}$. It is a finite subset which does not depend on
the choice of the inclusion $ \jmath_{v}$.  We
also denote by $\mu_{p,v}$ the uniform discrete probability measure on
$X_{{v}}^{\an}$ supported on $\Gal(p)_{v}$, that is
\begin{equation} \label{eq:4}
  \mu_{p,v}=\frac{1}{\# \Gal(p)_{v}}\sum_{q\in  \Gal(p)_{v}} \delta_{q},
\end{equation}
where $\delta _{q}$ is the Dirac measure at  the point
$q\in X^{\an}_{{v}}$. Hence, for a continuous function $f\colon X^{\an}_{{v}}\to \R$,
\begin{displaymath}
  \int f\dd \mu_{p,v}= \frac{1}{\# \Gal(p)_{v}}\sum_{q\in
    \Gal(p)_{v}} f(q). 
\end{displaymath}

An \emph{$\R$-divisor} on $X$ is a linear combination of Cartier divisors on $X$ with real
coefficients.  A \emph{metrized}
$\R$-divisor $\ov D$ on $X$ is an $\R$-divisor $D$ on $X$ equipped
with a quasi-algebraic family of $v$-adic metrics
$(\|\cdot\|_{v})_{v\in \fM_{\K}}$, see \cite[\S
3]{BurgosMoriwakiPhilipponSombra:aptv} for details.  In
\emph{loc. cit.}, for each $v\in \fM_{\K}$ the metric $\|\cdot\|_{v}$
is defined over the analytic space~$X_{\K_{v}}^{\an}$. Note that this space was
denoted ``$X_{v}^{\an}$" in \emph{loc. cit.} but since we will study
equidistribution problems 
of Galois orbits of points that are defined over varying extensions of
$\K$ of arbitrary large degree it is more convenient to work on
the space $X_{v}^{\an}$  instead that in the space
$X_{\K_{v}}^{\an}$. Hence we have changed the notation accordingly.
With this point of view, every object on $X_{\K_{v}}^{\an}$ will be
seen as an object on $X_{{v}}^{\an}$ by 
 taking its inverse image under the projection $\pi_{v}$.
For instance let $\ov D$ be a metrized $\R$-divisor on $X$ and  $s$ a
rational $\R$-section of~$D$~\cite[\S
3]{BurgosMoriwakiPhilipponSombra:aptv}. In \emph{loc.  cit.}, the
$v$-adic metric $\|\cdot\|_{v}$ is described by a continuous function
$\|s\|_{v}\colon X^{\an}_{\K_v}\setminus |\div(s)|\to \R_{>0}$. In
the current paper we denote by $\|s\|_v$ the function on
$X^{\an}_{v}\setminus |\div(s)|$ given by the composition
\begin{displaymath}
  \|s(p)\|_{v}=\|s(\pi_{v}(p))\|_{v}.
\end{displaymath}
Clearly this function is invariant under the action of $\Gal(\ov
\K_{v},\K_v)$.

To a metrized $\R$-divisor $\ov D$ on $X$ we can  associate a height function
\begin{displaymath}
  \h_{\ov D}\colon X(\ov \K)\longrightarrow \R
\end{displaymath}
as follows.  

Given $p\in X(\ov \K)$, choose a rational $\R$-section $s$ of $D$ such that
$p\not\in |\div(s)|$. Choose a finite extension $\F$ of $\K$ such that
$p\in X(\F)$. For each $w\in
\mathfrak{M}_{\F}$ over a place $v\in \fM_{\K}$, we can choose an
embedding $\sigma _{w}\colon \F\hookrightarrow \C_{v}$ such that the
restriction of the absolute value $|\cdot|_{v}$ of $\C_{v}$ agrees with
$|\cdot|_{w}$. We denote also by $\sigma _{w}$ the induced map
$X(\F)\to X^{\an}_{v}$.  

\begin{defn}\label{def:2} 
  Let $X$ be a variety over $\K$, $\ov D$ a metrized $\R$-divisor on
  $X$, and $p\in X(\ov \K)$. With the above notation, the
  \emph{height} of $p$ with respect to $\ov D$ is defined as
  \begin{displaymath}
    \h_{\ov D}(p)= -\sum_{w\in \mathfrak{M}_{\F}}n_{w} \log
    \|s\circ\sigma_{w}(p)\|_{v}. 
  \end{displaymath}
\end{defn}
The height  is independent of the choice of the rational $\R$-section 
$s$, the extension~$\F$ and the embeddings $\sigma _{w}$. 

Instead of choosing a finite extension where the point $p$ is defined,
it is possible to express the height of an algebraic point in terms of
its Galois orbit. 

\begin{prop}\label{prop:12} With the previous hypothesis and
  notation, the height of $p$ with respect to $\ov D$ is given by
  \begin{displaymath}
    \h_{\ov D}(p)=-\sum_{v\in \mathfrak{M}_{\K}}\frac{n_{v}}{\# \Gal(p)_{v}}
    \sum_{q\in \Gal(p)_{v}}\log\|s(q)\|_{v}.
  \end{displaymath}
\end{prop}

\begin{proof}  Choose a finite normal extension
  $\F\subset \ov \K$ of
  $\K$ such that $p\in X(\F)$. For
  each $v\in \mathfrak{M}_{\K}$ we denote $\mathfrak{M}_{\F,v}$ the
  set of places of $\F$ above $v$. 

Write $G=\Gal(\F,\K)$ and let $\F^{G}$ be the fixed field. Then
$\F/\F^{G}$ is a Galois extension with Galois group $G$ and
$\F^{G}/\K$ is purely inseparable.  Hence, for $v\in
\mathfrak{M}_{\K}$,
  \begin{displaymath}
    \frac{[\F_{w}:\K_{v}]}{[\F:\K]}=
    \frac{[\F_{w}:(\F^{G})_{v}]}{[\F:\F^{G}]}=\frac{1}{\#\mathfrak{M}_{\F,v}}.
  \end{displaymath}

Then, from the definition of the height
of $p$ in Definition~\ref{def:2} and Proposition
\ref{prop:15}\eqref{item:19},  it follows that
\begin{multline}
  \label{eq:63}
  \h_{\ov D}(p)=-\sum_{v\in \mathfrak{M}_{\K}}n_{v}\sum_{w\mid v}
  \frac{[\F_{w}:\K_{v}]}{[\F:\K]} \log
    \|s\circ\sigma_{w}(p)\|_{v}\\
  =-\sum_{v\in \mathfrak{M}_{\K}}\frac{n_{v}}{\# \mathfrak{M}_{\F,v}}
  \sum_{w\mid v}\log
    \|s\circ\sigma_{w}(p)\|_{v}.
\end{multline}

The group $G$ acts on $\mathfrak{M}_{\F,v}$ and, since $p$ is defined
over~$\F$, also on
$\Gal(p)_{v}$. Both actions are transitive. Therefore, choosing 
$w_{0}\in \mathfrak{M}_{\F,v}$, 
\begin{align*}
\frac{1}{\# \mathfrak{M}_{\F,v}} \sum_{w\mid v}\log
    \|s\circ\sigma_{w}(p)\|_{v}  &= \frac{1}{\# G} \sum_{g\in G}\log
    \|s\circ\sigma_{w_{0}}(g(p))\|_{v}  \\ &= \frac{1}{\# \Gal(p)_{v}} 
    \sum_{q\in \Gal(p)_{v}}\log\|s(q)\|_{v}.
\end{align*}
The statement follows from this together with \eqref{eq:63}.
\end{proof}

  The \emph{essential minimum} of $X$ with respect to $\ov D$ is
  defined as  
\begin{equation}
  \label{eq:14}
  \upmu^{\ess}_{\ov D}(X)=\sup_{\substack{Y\subsetneq X\\Y \text{
        closed}}}\inf_{p\in (X\setminus Y)(\ov \K)} \h_{\ov D}(p).
\end{equation}
Roughly speaking, the essential minimum is the generic infimum of the height
function.

\begin{defn}\label{def:3}
  Let $X$ be a variety over $\K$ and $\ov D$ a metrized $\R$-divisor
  on $X$. A net  $(p_{l})_{l\in I}$  of algebraic points of $X$ is
  \emph{$\ov D$-small} if
  \begin{displaymath}
    \lim_{l}\h_{\ov D}(p_{l})=\upmu^{\ess}_{\ov D}(X).
  \end{displaymath}
  The net $(p_{l})_{l\in I}$ is \emph{generic} if, for every
  closed subset  $Y\subsetneq X$,  there is $l_{0}\in I$ such
  that $p_{l}\not \in Y(\ov \K)$ for $l\ge l_{0}$.
\end{defn}

\begin{prop}\label{prop:4}   Given a  variety $X$ over $\K$ and $\ov
  D$ a metrized $\R$-divisor on $X$,  there exists a generic $\ov
  D$-small net of algebraic 
  points of $X$.   Moreover, every generic net
  $(p_{l})_{l\ge1}$ of algebraic points of $X$ satisfies
 \begin{displaymath} 
    \liminf_{l}\h_{\ov D}(p_{l})\ge \upmu^{\ess}_{\ov D}(X).
  \end{displaymath}
\end{prop}

\begin{proof}
The second statement is clear from the definition of the essential
minimum. 

For the first statement, let $I$ be the set of hypersurfaces of $X$,
ordered by inclusion. This is a directed set. For each $Y\in I$,
denote by $c(Y)$ its number of irreducible components and choose a
point $p_{Y}\in (X\setminus Y)(\ov \K)$ with
\begin{displaymath}
  \h_{\ov D}(p_{Y}) \le \upmu_{\ov D}^{\ess}(X) + \frac{1}{c(Y)}.
\end{displaymath}
Clearly, the net $(p_{Y})_{Y\in I}$ is generic and $\ov D$-small.   
\end{proof}

\begin{rem}
  \label{rem:7}
  When $\K$ is a number field, the collection of subvarieties of $X$
  is countable. Using this fact, we can strengthen
  Proposition~\ref{prop:4} to show the existence of a generic $\ov
  D$-small \emph{sequence} of algebraic points.
\end{rem}

Suppose now that the variety  $X$ is  proper
over $\K$ and of dimension $n$. 
A metrized $\R$-divisor $\ov D$ on $X$ is
\emph{semipositive} if it can be written as
\begin{displaymath}
  \ov D= \sum_{i=1}^{r}\alpha_{i}\ov D_{i}
\end{displaymath}
with $\ov D_{i}$ a semipositive metrized divisor and $\alpha_{i}\in
\R_{\ge 0}$, $i=1,\dots, r$. Recall that $\ov D_{i}$ is semipositive
if each of its $v$-adic metrics is a uniform limit of a sequence of
semipositive smooth (respectively, algebraic) metrics in the
Archimedean (respectively, non-Archimedean) case. 

Given a semipositive metrized $\R$-divisor $\ov D$, we can extend the
notion of height of points to subvarieties of higher dimension. In
particular, the {height} of $X$, denoted by $\h_{\ov D}(X)$, is
defined. Moreover, for each $v\in \fM_{\K}$ we can consider the
associated {$v$-adic Monge-Amp\`ere measure}, denoted by
$\chern_{1}(D,\|\cdot\|_{v})^{\wedge n}$. It is a measure on
$X^{\an}_{{v}}$ of total mass $\deg_{D}(X)$, see for instance
\cite[\S 1.4]{BurgosPhilipponSombra:agtvmmh} for the case when $D$ is
a divisor. The {$v$-adic Monge-Amp\`ere measure} of an $\R$-divisor is
defined from that of divisors by polarization and multilinearity. 

A theorem of Zhang shows that, when $\K$ is a number field, $D$ is an
ample divisor and $\ov D $ is semipositive, the essential minimum can
be bounded below in terms of the height of $X$ and the degree of $D$
\cite[Theorem 5.2]{Zhang:plbas}:
\begin{equation} \label{eq:2}
  \upmu^{\ess}_{\ov D}(X)\ge \frac{\h_{\ov D}(X)}{(n+1)\deg_{D}(D)}.
\end{equation}
This inequality can be generalized to global fields and semiample big
divisors, see for instance \cite[Proposition 5.10]{Gubler:Eff}.

In some cases, the inequality \eqref{eq:2} is an equality. For
instance, this happens for the canonical metric on divisors of toric
and abelian varieties, and for the canonical metrics coming from
dynamical systems. This motivates the following definition.

\begin{defn} \label{def:1}
  Let $X$ be a proper variety over $\K$ of dimension $n$, and 
$\ov D$ a semipositive metrized $\R$-divisor on $X$ with $D$
big. Then 
  $\ov D$ is \emph{quasi-canonical} if
  \begin{displaymath}
    \upmu^{\ess}_{\ov D}(X)=\frac{\h_{\ov D}(X)}{(n+1)\deg_{D}(X)}.
  \end{displaymath}
  In other words, quasi-canonical metrized $\R$-divisors  are those
  for which Zhang's 
  lower bound for the essential minimum is attained.
\end{defn}

As we will see in \S \ref{sec:equid-galo-orbits}, the condition for a
toric metric of being quasi-canonical is very restrictive.  The
following observation is a direct consequence of
Proposition~\ref{prop:4} and of the inequality \eqref{eq:2}.

\begin{prop} \label{prop:9}
  Let $X$ be a proper variety over $\K$ of dimension $n$ and $\ov
  D$ a semipositive metrized divisor on $X$ with $D$ big and
  semiample. Then there exists a generic net $(p_{l})_{l\in I}$ of
  algebraic points of $X$ with
  \begin{equation}\label{eq:54}
    \lim_{l}\h_{\ov D}(p_{l})=\frac{\h_{\ov D}(X)}{(n+1)\deg_{D}(X)}
  \end{equation}
  if and only if $\ov D$  is quasi-canonical.
\end{prop}

We discuss now the equidistribution of Galois orbits of 
points of small height. 
  
Let $X$ be a proper variety over $\K$ and $v\in \fM_{\K}$. We endow
the space of probability measures on $X_{{v}}^{\an}$ with the
weak-$\ast$ topology with respect to the space of continuous functions
on $X_{{v}}^{\an}$. In particular,  a net of probability measures $(\mu
_{l})_{l\in I}$ converges to a probability measure
$\mu $ if, for every continuous function~$f\colon
X^{\an}_{{v}} \to \R$,
\begin{displaymath}
  \lim_{l} \int f\dd
  \mu_{l}=\int f\dd \mu.
\end{displaymath}

\begin{defn}
  \label{def:4}
  Let $\ov D$ be a metrized
  $\R$-divisor on $X$. A probability measure
  $\mu$ on $X_{{v}}^{\an}$ is a \emph{$v$-adic limit measure} for $\ov
  D$ if there exists a~generic $\ov D$-small net $(p_{l})_{l\in I}$ of
  algebraic points of $X$ such that the net of probability measures
  $(\mu _{p_{l},v})_{l\in I}$ converges to $\mu $.
  We say that $\overline D$ satisfies the \emph{$v$-adic
    equidistribution property} if, for
  every~generic $\ov D$-small net $(p_{l})_{l\in I}$ as above, the
  net of measures~$(\nu _{p_{l},})_{l\in I}$ converges.
\end{defn}

Clearly, when the $v$-adic equidistribution property holds, there
exists a unique limit measure.

\begin{rem}
  \label{rem:8}
  When $\K$ is a number field, the analytic space $X_{v}^{\an}$ is
  metrizable and so is the space of probability measures on it. Using
  this together with Remark~\ref{rem:7} and standard arguments, one
  can reduce to sequences, instead of nets, when studying equidistribution properties
  over number fields.
\end{rem}

In the literature there are many equidistribution theorems of Galois
orbits of points of small height.  All these equidistribution results
deal with generic nets (or sequences when $\K$ is a number field) of
algebraic points satisfying the equality~\eqref{eq:54}. In view of
Proposition~\ref{prop:9}, the existence of such a net implies that the
metric is quasi-canonical. Moreover, the 
condition \eqref{eq:54} for this  net is equivalent of  being $\ov
D$-small. Thus we can reformulate a general equidistribution result in
the following form.

\begin{thm}\label{thm:4} Let $\K$ be a global field and $X$ a
  projective variety over $\K$ of dimension $n$. Let $\ov D$ be a
  semipositive metrized divisor on $X$ such that $D$ is ample. If $\ov
  D$ is quasi-canonical then, for every place $v\in
    \fM_{\K}$,
  \begin{enumerate}
  \item \label{item:13} $\overline D$ satisfies the $v$-adic
    equidistribution property;
  \item \label{item:14} the limit measure is the normalized
    Monge-Amp\`ere measure
    \begin{displaymath}
      \frac{1}{\deg_{D}(X)}\chern_{1}(D,\|\cdot\|_{v})^{\wedge n}.
    \end{displaymath}
  \end{enumerate}
\end{thm}

This result is due to Yuan \cite[Theorem
3.1]{Yuan:blbav} in the number field case and, with more general
hypotheses, to Gubler \cite[Theorem
1.1]{Gubler:Eff} in the function field case. 

Written in this form, it is clear that this equidistribution theorem
has a very restrictive hypothesis, that the metrized divisor $\ov D$ is
quasi-canonical. But it also has a very strong thesis: not only the
Galois orbits of points of small height converge to a measure, but
this limit measure can be identified with the normalized
Monge-Amp\`ere measure of the metrized divisor.

The main objective of this paper is to start the study of what happens
when the hypothesis of $\ov D$ being quasi-canonical is removed. We
will work with
toric varieties and toric metrics because, in this case, the tools
developed previously allow us to work very explicitly. In this setting,
we will see that the first  statement in Theorem~\ref{thm:4} holds in
much  great generality, but, if the metric is not
quasi-canonical,  the limit measure does
not need to agree with the normalized Monge-Amp\`ere measure.   



\section{Auxiliary results on convex analysis}
\label{sec:auxil-results-conv}

In this section we gather several definitions and results on convex
analysis that we will use in our study of toric height functions. 
For a background in convex analysis, see for instance \cite[\S
2]{BurgosPhilipponSombra:agtvmmh}.

Let $N_{\R}\simeq \R^{n}$ be a real vector space of dimension $n$ and
$M_{\R}=\Hom(N_{\R},\R)=N_{\R}^{\vee}$ its dual. The pairing between
$x\in M_{\R}$ and $u\in N_{\R}$ will be denoted either  by
$\langle x,u\rangle$ or~$\langle u,x\rangle$. 

Following \cite[\S 2]{BurgosPhilipponSombra:agtvmmh}, a convex subset
$C$ is nonempty.  The \emph{relative interior} of $C$,
denoted by $\ri(C)$, is  the interior $C$ relative to the
minimal affine subspace containing it.

Let~$C\subset M_{\R}$ be a convex subset and $g\colon C\to \R$ a
concave function. The \emph{sup-differential} of $g$ at a point $x\in
C$ is
\begin{displaymath}
  \partial g(x) =\{u\in N_{\R}\mid \langle u,z-x\rangle\ge
  g (z) - g(x) \text{ for all } z\in C \}.
\end{displaymath}
It is a convex subset of $N_{\R}$. The \emph{stability set} of $g$ is
the subset of $N_{\R}$ defined by
\begin{displaymath}
  \stab(g) = \{u\in N_{\R}\mid u-g \text{ is bounded below}\}.
\end{displaymath}
The \emph{Legendre-Fenchel dual} of $g$ is the function $g^{\vee}
\colon \stab(g) \rightarrow  \R $ defined by
\begin{equation}
  \label{eq:11}
g^{\vee}(u)=
\inf_{x\in C} \langle u,x\rangle -g(x).
\end{equation}

Let~$E\subset N_{\R}$ be a convex subset.  A nonempty
subset $F\subset E$ is a \emph{face} of $E$ if every closed segment
$S\subset E$ whose relative interior has nonempty intersection
with~$F$, is contained in~$F$.

\begin{lem}
\label{lem:flatification}
Let~$C\subset M_{\R}$ be a compact convex subset and
$g_{1},g_{2}\colon C\to \R$ two continuous concave functions. Denote
by~$C_{\max}$ the convex subset of~$C$ of the points where~$g_1 + g_2$
attains its maximum value and choose $x\in C_{\max}$.  For
$i=1,2$, consider the concave function~$\hphi_i \colon N_{\R} \to \R$
defined by
\begin{equation}
  \label{eq:7}
\hphi_i(u) =
g_i^{\vee}(u) - \langle x, u \rangle+ g_i(x).
\end{equation}
Then
\begin{enumerate}
\item \label{item:3} if $x'\in \ri(C_{\max})$, then $\partial g_{i}(x')$ is a face of $\partial g_{i}({x})$, $i=1,2$;
\item \label{item:4} $\partial g_{1}({x})\cap (-\partial g_{2}({x}))$ is
  nonempty and does not depend on the choice of $x\in C_{\max}$;
\item \label{item:28} the minimal face of  $\partial g_{1}({x})$
  containing $\partial g_{1}({x})\cap (-\partial g_{2}({x}))$ does not depend on the choice of $x\in C_{\max}$;
\item \label{item:5} the function $\hphi_{i}$ is
  nonpositive and vanishes precisely 
  on $\partial g_{i}({x}) $.
\end{enumerate}
\end{lem}

\begin{proof}
  The restriction to $C_{\max}$ of the sum $g_{1}+g_{2}$ is constant,
  and so the restrictions to this set of $g_{1}$ and $g_{2}$ are
  affine and with opposite slopes. In other words, there is $u_{0}\in
  N_{\R}$ such that, for all $x_{1},x_{2}\in C_{\max}$,
\begin{equation}\label{eq:58}
  g_{1}(x_{2})-g_{1}(x_{1})= \langle u_{0},x_{2}-x_{1}\rangle \and   g_{2}(x_{2})-g_{2}(x_{1})= -\langle u_{0},x_{2}-x_{1}\rangle.
\end{equation}
  
For the statement \eqref{item:3}, let $i=1,2$ and $u\in\partial g_{i}({x'})$. Since $x'$ is in the relative interior of $C_{\max}$, there
exists $\varepsilon >0$ such that $x'-\varepsilon(x-x') \in C_{\max}$.
By the definition of the sup-differential, for all~$z\in C$,
\begin{equation}
  \label{eq:5}
  \langle u,z-x'\rangle \ge g_{i}(z) -g_{i}(x'). 
\end{equation}
By \eqref{eq:5} and  \eqref{eq:58}, 
\begin{multline*}
  -\varepsilon \langle u,x-x'\rangle = 
  \langle u,x'-\varepsilon(x-x')-x'\rangle\\  \ge
  g_{i}(x'-\varepsilon(x-x')) -g_{i}(x')\\
  = \langle u_{0},-\varepsilon (x-x')\rangle =
  -\varepsilon (g_{i}(x)-g_{i}(x'))  .
\end{multline*}
Hence $ \langle u,x-x'\rangle \le g_{i}(x)-g_{i}(x')$. By
\eqref{eq:5} applied to $z=x$,  we have also the reverse inequality.  Thus
$\langle u,x-x'\rangle  =g_{i}(x)-g_{i}(x')$, and 
it follows from \eqref{eq:5} that, for all $z\in C$, 
\begin{displaymath}
    \langle
u,z-x\rangle \ge g_{i}(z) -g_{i}(x).
\end{displaymath}
Hence $u\in \partial g_{i}({x}) $ and so $\partial g_{i}({x'})
\subset \partial g_{i}({x}) $.  Then \cite[Proposition
2.2.8]{BurgosPhilipponSombra:agtvmmh} implies that
$\partial g_{i}({x'}) $ is a face of $ \partial g_{i}({x}) $.

To prove the statement \eqref{item:4} note that, since~$g_{1}+g_{2}$
attains its maximum value at~$x$, by \cite[Proposition
2.3.6(2)]{BurgosPhilipponSombra:agtvmmh}
\begin{displaymath}
0 \in \partial (g_{1}+g_{2})({x})  = \partial g_1({x}) 
+ \partial g_2({x}) .
\end{displaymath}
Hence $\partial g_{1}({x})\cap (-\partial g_{2}({x}))\ne\emptyset$, as
stated. Now let $u$ be a point in this intersection. Then 
\begin{equation}
  \label{eq:59}
  \langle u,z-x\rangle \ge g_{1}(z)-g_{1}(x) \and   \langle -u,z-x\rangle \ge g_{2}(z)-g_{2}(x) .
\end{equation}
Choose $x''\in C_{\max}$. Subtracting, from the inequalities \eqref{eq:59} applied to
$z=x''$, the identities \eqref{eq:58} applied to $x_{1}=x$ and
$x_{2}=x''$, we deduce that
\begin{equation*}
  \langle u-u_{0},x''-x\rangle=0.
\end{equation*}
Using this together with \eqref{eq:5}and \eqref{eq:59}, we obtain 
\begin{displaymath}
    \langle u,z-x''\rangle \ge g_{1}(z)-g_{1}(x'') \and   \langle
    -u,z-x''\rangle \ge g_{2}(z)-g_{2}(x'') .
\end{displaymath}
Hence $u\in \partial g_{1}({x''})\cap (-\partial g_{2}({x''}))$, as
stated.

For the next statement, consider the convex set $B =\partial
g_{1}({x})\cap (-\partial g_{2}({x}))$ which, thanks to \eqref{item:4},
does not depend on the choice of $x\in C_{\max}$. Denote by $F_{x}$
the minimal face of $\partial g_{1}({x})$ containing it. By
\eqref{item:3}, it is enough to consider the case when $x\in
\ri(C_{\max})$. By the same statement, the set $\partial g_{2}({x})$
does not depend on the choice of $x\in \ri(C_{\max})$, proving
\eqref{item:28}.

The statement \eqref{item:5} follows readily from \cite[Lemma
2.2.6]{BurgosPhilipponSombra:agtvmmh}. 
\end{proof}


\begin{defn} \label{def:5} Let~$C\subset M_{\R}$ be a compact convex
  subset and $g_{1},g_{2}\colon C\to \R$ two continuous concave
  functions.  Let~$C_{\max}$ be the convex subset of~$C$ of the points
  where~$g_1 + g_2$ attains its maximum value.  Given $x\in C_{\max}$,
  we define the convex subset of $N_{\R}$
  $$ 
  B(g_{1},g_{2}) = \partial g_{1}({x})\cap
  (-\partial g_{2}({x}))
  $$ 
  and the convex subset
  \begin{displaymath}
    F(g_{1},g_{2})\subset \partial g_{1}({x})
  \end{displaymath}
  as the minimal face of~$\partial g_{1}({x})$
  that contains $B(g_{1},g_{2})$.  By Lemma
  \ref{lem:flatification}, these convex subsets do not depend on the
  choice of $x$.

  We also  define the convex subset
  $A_{i}(g_{1},g_{2})=\partial g_i(x)  \subset N_{\R}$, $i=1,2$. By
  the same lemma, this convex subset does not depend 
  on the choice of the point $x$.
\end{defn}

\begin{rem}
  \label{rem:2}
  In the setting of Lemma~\ref{lem:flatification}, it is not always
  true that $\partial g_{i}({x})=\partial g_{i}({x'})$ for $x,x'\in
  C_{\max}$. An example of this situation is when $g_{i}$ is the zero
  function on the interval $[0,1]\subset\R$ and $x=0, x'=1/2$.  Hence,
  in Definition~\ref{def:5} we cannot define the sets
  $A_{i}(g_{1},g_{2})$ using an arbitrary point $x\in C_{\max}$.
\end{rem}

\begin{lem}
  \label{lemm:5}
  Let $C\subset M_{\R}$ be a compact convex subset with nonempty
  interior and $g_{1},g_{2}\colon C\to \R$ two concave functions. Then
  $B(g_{1},g_{2})$ is bounded and $F(g_{1},g_{2})$ contains no lines.
\end{lem}

\begin{proof}
  The convex set $B(g_{1},g_{2})$ is not bounded if and only if it
  contains a ray, that is, a subset of the form $\R_{\ge0} u_{1} +
  u_{2}$ with $u_{i}\in N_{\R}$, $i=1,2$, and $u_{1}\ne 0$. Suppose
  that this is the case.  This implies that, for $x\in C_{\max}$
  and all $t\ge0$,
\begin{displaymath}
 tu_{1} + u_{2}\in \partial g_{1} ({x})\quad \text{ and } \quad 
-tu_{1} - u_{2}\in \partial g_{2}({x}).  
\end{displaymath}
Hence, for all $z\in C$ and $t\ge0$,
\begin{displaymath}
-\langle u_{2},z-x\rangle +g_1(z)-g_1(x) \le t\langle u_{1},z-x\rangle
\le -\langle u_{2},z-x\rangle -g_2(z)+g_2(x).
\end{displaymath}
Letting $t\to \infty$, this implies 
$C\subset \{z \mid \langle u_{1},z-x\rangle =0\}$, contradicting
the hypothesis that $C$ has nonempty interior. Hence $B(g_{1},g_{2})$
is bounded.

Similarly, if $F(g_{1},g_{2})$ contains a line $\R u_{1}+u_{2}$, then,
for $x\in C_{\max}$ and  $t\in\R$, 
\begin{displaymath}
   tu_{1} + u_{2}\in \partial g_{1}({x}).
\end{displaymath}
 This also
implies that $C$ is contained in the affine hyperplane $ \{z\mid
\langle u_{1},z-x\rangle =0\}$ and contradicts the hypothesis that~$C$
has nonempty interior. Hence $F(g_{1},g_{2})$ contains no lines.
\end{proof}

Let $\Cb(N_{\R})$ be the space of bounded continuous functions on
$N_{\R}$, and let $\|\cdot\|$  be an auxiliary norm on $N_{\R}$ that
we fix.

\begin{defn} \label{def:8} We denote by~$\cP$ the space of Borel
  probability measures on~$N_{\R}$ endowed with the weak-$\ast$
  topology with respect to $\Cb(N_{\R})$.  This is the smallest
  topology on $\cP$ such that, for all $\varphi \in \Cb(N_{\R})$, the
  function~$\mu \mapsto \int \varphi \dd \mu$ is continuous.

  We denote by $\cE\subset \cP$ the topological subspace of
  probability measures with finite first moment, that is, the
  probability measures on $N_{\R}$ satisfying
\begin{displaymath}
  \int \|u\|\dd \mu (u) <\infty.
\end{displaymath}
For $\mu \in \cE$, the \emph{expected value} is 
\begin{displaymath}
  \exv[\mu ]=\int u\dd \mu (u)\in N_{\R}.
\end{displaymath}  
\end{defn}

The weak-$\ast$ topology of $\cP$ with respect to $\Cb(N_{\R})$ is called
the ``topologie \'etroite'' in \cite[\S 5]{Bourbaki:Iix}.  By
Proposition 5.4.10 in \emph{loc. cit.}, the topological space $\cP$ is
complete, metrizable and separable.
Latter we will consider other topologies on the underlying spaces of
$\cP$ and $\cE$. When this is the case, we will state explicitly the used topology.

For~$\mu\in \cP$, its \emph{support}, denoted by~$\supp(\mu)$, is the
set of all points in~$N_{\R}$ such that all its neighborhoods have
positive measure.  The set of probability measures on $N_{\R}$ with
finite support is contained in~$\cE$, and is dense therein.

For the rest of this section, we fix a compact convex subset $C\subset
M_{\R}$ with nonempty interior and two continuous concave functions
$g_{1},g_{2}\colon C\to \R$.  The Legendre-Fenchel dual $g_{i}^{\vee}$
is a concave function on $N_{\R}$ with stability set~$C$.

We introduce the function $\Phi\colon \cE\to \R$ given, for $\mu\in \cE$,
by 
\begin{equation}\label{eq:27}
  \Phi(\mu)= \int g_{1}^{\vee}\dd \mu +
  g_{2}^{\vee}(-\exv[\mu])+\max_{x\in C}(g_{1}(x)+g_{2}(x)).
\end{equation}
This function will play a central role in the proof of the main
results in the next section.

It follows easily from its definition that $\Phi$ is concave. In
general, this function is not continuous, as the following example
shows.

\begin{exmpl} \label{exm:1}
  Let $N_{\R}=\R$, so that $M_{\R}=\R$. Set $C=[0,1]$ and $g_{i}=0$,
  $i=1,2$. Then $g_{i}^{\vee}(u)=\min(0,u)$ for $u\in \R$. Consider
  the sequence of measures
  \begin{displaymath}
    \mu _{l}=\frac{l-1}{l}\delta _{0}+\frac{1}{l}\delta _{-l}, \quad l\ge1, 
  \end{displaymath}
where $\delta_{0}$ and $\delta_{-l}$ are the Dirac measures at the
points 0
and $-l$, respectively. 
  This sequence converges to $\delta _{0}$. However, 
  \begin{math}
    \Phi (\mu _{l})=-1 
  \end{math} for all $l$ 
  and $\Phi (\delta _{0})=0$.
\end{exmpl}

Nevertheless, we have the following result.
\begin{prop}\label{prop:3}
 The function $\Phi $ is upper semicontinuous.    
\end{prop}

To prove this proposition, we need the following lemma.

\begin{lem}\label{lemm:3}
  Let $\phi\colon N_{\R}\to \R$ be a continuous
  function. If $\phi $ is bounded above
  (respectively below), then the map $\cP \to \R\cup
  \{-\infty \}$ (respectively $\cP \to \R\cup \{\infty\}$)
    given by
    \begin{displaymath}
      \mu \longmapsto \int \phi \dd \mu 
    \end{displaymath}
    is upper semicontinuous  (respectively lower semicontinuous).
\end{lem}
\begin{proof}
  We prove the case of a function bounded above, the other case being
  analogous.  Let $\mu \in \cP$ and~$\varepsilon > 0$ be given and,
  for $l\ge 1$, put
  \begin{displaymath}
    \phi _{l}(u)=\max(\phi(u),-l). 
  \end{displaymath}
  The sequence of functions $(\phi _{l})_{l\ge1}$ is monotone and
  converges pointwise to $\phi$. So Lebesgue's monotone convergence
  theorem implies that there is $l_{0}\ge1$ such that
  \begin{displaymath}
    \int \phi_{l_{0}} \dd \mu \le\int \phi \dd
    \mu + \varepsilon. 
  \end{displaymath}

  Let $(\mu_{l})_{l\ge1}$ be a sequence in $\cP$ converging to $\mu$.
  Since $\phi _{l_0}\in \Cb(N_{\R})$, there exists $l_{1}\ge 1$ such
  that, for $l\ge l_{1}$,
\begin{displaymath}
  \int\phi\dd\mu_{l}\le   \int\phi_{l_{0}}\dd\mu_{l}\le
  \int\phi_{l_{0}}\dd\mu + \varepsilon \le
\int \phi \dd \mu + 2\varepsilon. 
\end{displaymath}
Since~$\varepsilon$ is arbitrary, $\limsup_{l\to \infty}
\int\phi\dd\mu_{l}\le \int \phi \dd \mu $,  proving the lemma.
\end{proof}

\begin{proof}[Proof of Proposition~\ref{prop:3}] 
  Set $\phi_{i}=g_{i}^{\vee}$, $i=1,2$ for short.  Fix $\mu_{0}\in\cE$
  and set $u_{0}=-\exv[\mu_{0}]\in N_{\R}$. Choose $x\in \partial \phi_{2}({u_{0}})\subset M_{\R}$ so that, for all $u\in N_{\R}$,
  \begin{equation*}
\langle x,u-u_{0}\rangle\ge     \phi_{2}(u)-\phi_{2}(u_{0}).
  \end{equation*}
  Let $\mu\in \cE$. It follows from this inequality that
  \begin{align*}
    \Phi(\mu)-\Phi(\mu_{0})&=\int \phi_{1}\dd\mu+\phi_{2}(-\exv[\mu]) -
\int \phi_{1}\dd\mu_{0}-\phi_{2}(-\exv[\mu_{0}]) \\
&\le \int \phi_{1}\dd(\mu-\mu_{0})- \langle \exv[\mu] - \exv[\mu_{0}],x\rangle \\
&= \int \phi_{1} \dd(\mu-\mu_{0})- \int \langle u,x\rangle  \dd(\mu-\mu_{0}) \\
&= \int \phi \dd(\mu-\mu_{0}) 
  \end{align*}
with $\phi=\phi_{1}-x$. Hence
\begin{equation}\label{eq:31}
  \Phi(\mu)\le \Phi(\mu_{0})- \int\phi\dd\mu_{0}+ \int\phi\dd\mu.
\end{equation}
Since $x$ belongs to
$\partial \phi_{2}({u_{0}})$ and $\partial \phi_{2}({u_{0}})\subset
\stab(\phi_{2})=\stab(\phi_{1})=C$, the function $\phi$ is bounded
above. By Lemma~\ref{lemm:3}, the right-hand side of \eqref{eq:31} is
upper semicontinuous. The inequality \eqref{eq:31} is an equality for
$\mu=\mu_{0}$. Hence $\Phi$ is upper semicontinuous at $\mu_{0}$, as
stated. 
\end{proof}

\begin{prop}\label{prop:1} 
The function $\Phi$ is nonpositive, and vanishes for $\mu\in \cE$ if
and only if
  \begin{equation}
    \label{eq:28}
\supp(\mu) \subset F(g_{1},g_{2})
\quad \text{ and } \quad
\exv[\mu ]\in B (g_{1},g_{2}),
  \end{equation}
  with $B(g_{1},g_{2})$ and $F(g_{1},g_{2})$ as in Definition
  \ref{def:5}.
\end{prop}

\begin{proof} 
Let notation be as in  Lemma~\ref{lem:flatification} and for short
put
\begin{displaymath}
A_{i}=A_i(g_1, g_2), \quad B=B(g_1, g_2), \quad F=F(g_1, g_2).  
\end{displaymath}
Fix a point $x\in \ri(C_{\max})$ and let $\hphi_i$ be as in
\eqref{eq:7}. For every~$\mu\in \cE$ we can write $\Phi (\mu)$ in terms of
the functions $\hphi_i$ as
\begin{equation}\label{eq:19}
  \Phi(\mu)= 
\int \hphi_{1}\dd \mu +
  \hphi_{2}(-\exv[\mu]).
\end{equation}
By Lemma~\ref{lem:flatification}\eqref{item:5}, the
functions~$\hphi_i$ are nonpositive and vanish precisely on the
sets~$A_{i}$.  It follows from \eqref{eq:19} that
$\Phi$ is nonpositive and vanishes for every~$\mu\in \cE$
satisfying~\eqref{eq:28}.

Conversely, let~$\mu\in \cE$ such that $\Phi (\mu )=0$.  Since
both~$\hphi_1$ and~$\hphi_2$ are nonpositive, the equality \eqref{eq:19} also
implies that
\begin{displaymath}
  \int \hphi _{1}\dd \mu =0
\quad \text{ and } \quad
\hphi_{2}(-\exv[\mu ])=0.
\end{displaymath}
Therefore~$\supp(\mu) \subset A_{1}$ and $-\exv[\mu ]\in A_{2}$.
Since~$A_1$ is convex, $\exv[\mu] \in A_1$ and so
\begin{displaymath}
  \exv[\mu]\in A_{1}\cap (-A_{2})=B,
\end{displaymath}
which gives the second condition in \eqref{eq:28}.


We next prove that the first condition in  \eqref{eq:28} is satisfied. Write
$\theta=\mu (F)$, so that $0\le \theta\le 1$ and $\mu(A_{1}\setminus F)=1-\theta$.

If $\theta <1$, we put
\begin{displaymath}
  u_2 = \frac{1}{1-\theta}\int_{A_{1} \setminus F} u \dd \mu\in A_{1} \setminus F. 
\end{displaymath}
If $\theta=0$, then $\exv[\mu ]=u_{2}$ and so
$\exv[\mu ]\in A_{1}\setminus F$, contradicting the fact that
$\exv[\mu ]\in B\subset F$.  Suppose that $0<\theta <1$ and set
\begin{displaymath}
u_1 = \frac{1}{\theta}\int_{F} u\dd \mu \in F. 
\end{displaymath}
Therefore
\begin{displaymath}
  \exv[\mu]
= \theta u_{1}+(1-\theta) u_{2}\in \ri(\ov{u_{1}u_{2}}),
\end{displaymath}
the relative interior of the segment $\overline
{u_{1}u_{2}}$.  Since $\exv[\mu]$ is in~$B$ and hence in~$F$, we have
$\ri(\ov{u_{1}u_{2}})\cap F\ne\emptyset$.  Moreover, the whole segment
is contained in~$A_{1}$, and~$F$ is a face of $A_1$. We deduce that
this segment is contained in $F$. Therefore $u_{2}\in F$,
contradicting the fact that $u_{2}\in A_{1}\setminus F$. We conclude
that $\theta=1$ and so $\supp(\mu)\subset F$.  This proves the first
condition and completes the proof.
\end{proof}

The function $\Phi $ satisfies also the following property.

\begin{lem}\label{lemm:2} There are constants $c_{1}\ge 0$ and $c_{2}> 0$ such
  that, for all $\mu \in \cE$, 
  \begin{displaymath}
    \Phi (\mu )\le c_{1}-c_{2}\int \|u\|\dd \mu.
  \end{displaymath}
\end{lem}
\begin{proof}
Let $\Psi$ be the support function of $C$, which is the function on $N_{\R}$ given by
 \begin{equation*} 
   \Psi(u)=\min_{y\in C}\langle u,y\rangle.
 \end{equation*}
 Put~$c_1 = 4\max_{y\in C}(|g_{1}(y)|,|g_{2}(y)|)$.  It follows from
 their definition that the functions $\phi_{i}=g_{i}^{\vee}$ verify,
 for $u\in N_{\R}$,
 \begin{equation}
   \label{eq:1}
    \max \left(\phi _{1}(u),\phi _{2}(u)\right) \le \Psi(u)+\frac{c_{1}}{4}. 
 \end{equation}
  Let $x$ be a point in the interior of $C$. On
  $M_{\R}$, we consider the norm induced by the 
  chosen norm~$\| \cdot \|$ in $N_{\R}$. Since $x$ is interior, we can find a
  constant $c_{2}>0$ such that $\mathcal{B}(x,c_{2})$, the closed ball
  of center $x$ and
  radius $c_{2}$, is contained in $C$. Then
  \begin{equation} \label{eq:50}
    \Psi (u)\le 
    \min_{y\in \mathcal{B}(x,c_{2})}\langle u,y\rangle =\langle u,x\rangle -c_{2}\|u\|.
  \end{equation}
Since $x\in C=\stab(\Psi )$, we have
  $(\Psi -x)(u)\le 0$.
By \eqref{eq:1} and \eqref{eq:50}, 
  \begin{align*}
    \Phi (\mu )&=\int \phi _{1}(u)\dd \mu + \phi _{2}(-\exv[\mu ]) +
    \max_{y\in C}(g_{1}(y)+g_{2}(y))\\
& \le c_{1}+\int \Psi  (u)\dd \mu + \Psi (-\exv[\mu ])\\ &=
    c_{1}+\int (\Psi-x)  (u)\dd \mu + (\Psi -x)(-\exv[\mu
    ])\\ 
&  \le c_{1}-c_{2}\int \|u\|\dd \mu,
  \end{align*}
as stated.
\end{proof}

\begin{prop}
\label{prop:2 bis}
Let~$(\mu_l)_{l \in I}$ be a net of measures in~$\cE$ such
that
\begin{displaymath}
\lim_{l} \Phi(\mu_l) = 0.  
\end{displaymath}
Then~$(\mu_l)_{l \in I}$ has at least one cluster point in~$\cP$, and
every such cluster point~$\mu$ lies in~$\cE$ and satisfies
  \begin{equation*}
\supp(\mu) \subset F(g_{1},g_{2})
\quad \text{ and } \quad
\exv[\mu ]\in B (g_{1},g_{2}).
\end{equation*}
\end{prop}

\begin{proof}
  Replacing~$(\mu_l)_{l \in I}$ by a subnet if necessary, we assume
  that $\Phi(\mu_l) \ge -1$ for all $l\in I$.  Let $c_{1}, c_{2}$ be
  the constants of Lemma~\ref{lemm:2} and set
  $K=(c_{1}+1)/c_{2}>0$. This lemma implies that each $\mu_l$ is in
  the subset of $\cE$ given by
\begin{equation*}
\Big\{\mu\in\cE\, \Big|\ \int \|u\|\dd \mu (u) \le K\Big\}.
\end{equation*}
This subset is compact thanks to Prokhorov's theorem
\cite[Th\'eor\`eme 5.5.1]{Bourbaki:Iix}, and it is  metrizable
because $\cP$ is. Hence, the net~$(\mu_{l})_{l \in I}$ has at
least one cluster point, and every such cluster point $\mu$ lies
in~$\cE$, proving the first statement.

  To prove the last statement, let~$(\mu_{k})_{k \in I' }$ be a
  subnet converging to $\mu$.  By Proposition~\ref{prop:3},
the function $\Phi$ is upper-semicontinuous and so
$$ 
\Phi(\mu)
\ge \limsup_{k} \Phi(\mu_{k}) = 0. 
$$ 
Hence~$\Phi(\mu) = 0$, and the statement follows from
Proposition~\ref{prop:1}.
\end{proof}

As we have seen in Example~\ref{exm:1}, the function $\Phi $ is not
continuous. We now consider another
topology on $\cE$ with respect to
which the function $\Phi $ is continuous.

Given~$\mu, \mu'\in \cP$, denote by~$\Gamma(\mu, \mu')$ the set
of probability measures on~$N_{\R} \times N_{\R}$ with marginals~$\mu$
and $\mu'$. That is, a probability measure $\nu$ on~$N_{\R} \times
N_{\R}$ belongs to $\Gamma(\mu, \mu')$ if and only if
\begin{displaymath}
  (p_{1})_{\ast}\nu =\mu ,\quad (p_{2})_{\ast}\nu =\mu', 
\end{displaymath}
where $p_{i}$ is the projection of $N_{\R} \times
N_{\R}$ onto its $i$-th factor, and $(p_{i})_{*}$ the direct image of
measures. 

\begin{defn}\label{def:9}
The
\emph{Kantorovich--Rubinstein distance} (or \emph{Wasserstein distance
  of order 1}) on~$\cE$ is defined, for $\mu,\mu'\in \cE$,
by
\begin{equation*}
  W(\mu, \mu')=
\inf_{\nu \in \Gamma(\mu, \mu')} \int \| u - u' \| \dd \nu(u, u').
\end{equation*}
The \emph{Kantorovich--Rubinstein topology} (or KR-topology for short)
of $\cE$ is the topology 
induced by this distance.   
\end{defn}

For a Lipschitz continuous function~$\psi \colon N_{\R} \to \R$,
denote by~$\Lip(\psi)$ its \emph{Lipschitz constant}, given by
$$ \Lip(\psi)
=
\sup_{u\ne u'} \frac{|\psi(u) -
  \psi(u')|}{\| u - u' \|}. 
$$
Lipschitz constants and the Kantorovich-Rubinstein distance are
related by the \emph{duality formula}: for $\mu,\mu'\in \cE$ and a
Lipschitz continuous function~$\psi \colon N_{\R} \to \R$, we have
\begin{equation}
  \label{eq:48}
  \bigg| \int \psi \dd \mu-\int \psi \dd \mu'\bigg|\le \Lip(\psi) \,
  W(\mu,\mu'), 
\end{equation}
see for instance~\cite[Remark~6.5]{Villani:oton}.

\begin{rem}\label{rem:4}
By \cite[Theorem 6.9]{Villani:oton}, the
KR-topology agrees with the weak-$\ast$ topology on $\cE$
with respect to the space of continuous functions
 $\varphi\colon
N_{\R}\to \R$ such that
\begin{displaymath}
 |\varphi(u)|\le c(1+\|u\|)   
\end{displaymath}
for a $c\in \R$ and all $u\in N_{\R}$. In particular, the KR-topology
is stronger than the topology of $\cE$ induced by that of $\cP$ as in
Definition \ref{def:8}.
\end{rem}


\begin{prop}\label{prop:5}
  The function $\Phi$ is continuous with respect to
  the KR-topology. In particular, if~$(\mu_l)_{l \in I}$ is a
  net of measures in~$\cE$ that converges to a measure $\mu\in
  \cE$ with respect to the KR-topology and  
  \begin{equation*}
\supp(\mu) \subset F(g_{1},g_{2})
\quad \text{ and } \quad
\exv[\mu ]\in B (g_{1},g_{2}),
\end{equation*}
then 
\begin{math}
\lim_{l } \Phi(\mu_l) = 0.  
\end{math}
\end{prop}
\begin{proof}
  Let $(\mu_l)_{l \in I}$ be a net on $\cE$ that converges to a
  measure $\mu\in \cE $ with respect to the KR-topology. By
  Remark~\ref{rem:4},
  \begin{displaymath}
    \lim_{l}\int g_{1}^{\vee}\dd \mu _{l}=
      \int g_{1}^{\vee}\dd \mu \and
    \lim_{l} g_{2}^{\vee}(-E[\mu
    _{l}])=g_{2}^{\vee}(-E[\mu]). 
  \end{displaymath}
  Therefore $\lim_{l} \Phi (\mu _{l})=\Phi (\mu)$ and so 
  $\Phi $ is continuous, proving the first statement. The second
  statement follows from the first one and Proposition~\ref{prop:1}.    
\end{proof}

We also need the following easy result, that we include here for
the lack of a suitable reference.

\begin{lem}\label{lemm:4}
  Let $E_i\subset N_{\R}$, $i=1,\dots,r$, be  convex subsets  and $E=E_1+\dots
  +E_r$ their Minkowski
  sum. For a point $u_{0}\in E$, the  following conditions are
  equivalent:
  \begin{enumerate}
  \item \label{item:10} the point $u_{0}$ is a vertex of $E$;
  \item \label{item:11} the equation $u_{0}=\sum_{i} z_i$ with $z_i\in
    E_i$ has a unique solution and, for $i=1,\dots,r$, the point
    $z_i$ in this solution is a vertex of $E_i$.
  \end{enumerate}
\end{lem}

\begin{proof}
  First assume that $u_{0}$ is a vertex of $E$. Suppose that the
  equation $u_{0}=\sum_{i} z_i$, $z_i\in E_i$, has two different
  solutions, namely $u_{0}=\sum_{i} z_i'$ and $u_{0}=\sum_{i} z_i''$
  with $z_{i_{0}}'\not =z_{i_0}''$ for some $i_{0}\in \{1,\dots,
  r\}$. Then both points
  \begin{displaymath}
    u_1=\sum_{i\not=i_0} z_i'+z_{i_0}'' \quad \text{and } \quad  u_2=\sum_{i\not=i_0}
    z_i''+z_{i_0}'  
  \end{displaymath}
  belong to $E$, they are different and satisfy
  $u_{0}=\frac{1}{2}(u_1+u_2)$, contradicting the fact that $u_{0}$ is
  a vertex of $E$. Hence the equation $u_{0}=\sum_{i} z_i$
  has a unique solution with $z_i\in E_i$.

  Now suppose that $z_{i_0}$ is not a vertex of $E_{i_0}$ for some
  $i_{0}\in \{1,\dots, r\}$. Then we can write
  $z_{i_0}=\frac{1}{2}(z'_{i_0}+z''_{i_0})$ with
  $z'_{i_0}\not=z''_{i_0}$ both in $E_{i_0}$. Hence the points
  \begin{displaymath}
    u_1=\sum_{i\not=i_0} z_i+z_{i_0}' \quad \text{and } \quad  u_2=\sum_{i\not=i_0}
    z_i+z_{i_0}''
  \end{displaymath}
  are different, belong to $E$ and $u_{0}=\frac{1}{2}(u_1+u_2)$,
  contradicting the assumption that $u_{0}$ is a vertex of $E$. Thus
  we have proved that \eqref{item:10} implies \eqref{item:11}.

  Assume now that the statement \eqref{item:11} is true but $u_{0}$ is not a
  vertex of $E$. Then there are two
  different points $u_1,u_2\in E$ with $u_{0}=\frac{1}{2}(u_1+u_2)$. Since
  $E$ is the Minkowski sum of the sets $E_i$, we can write
  \begin{displaymath}
    u_{0}=\sum_{i}z_i,\quad u_1=\sum_{i} z'_i  \and
    u_2=\sum_{i}
    z_{i}''.    
  \end{displaymath}
The equation $u_{0}=\sum_{i} z_i$ has a unique solution and so
$z_i=\frac{1}{2}(z_i'+z_i'')$ for all $i$. Since $z_i$ is a
  vertex of $E_i$, this implies $z_i'=z_i''$. Therefore $u_1=u_2$,
  contradicting the assumptions and thus proving that \eqref{item:11}
  implies \eqref{item:10}.  
\end{proof}

\section{Modulus distribution}
\label{sec:modul-distr-toric}

In this section, we study the asymptotic modulus distribution of the
Galois orbits of nets of points of small height in toric
varieties. Our approach is based on the techniques developed in the
series of papers
\cite{BurgosPhilipponSombra:agtvmmh,BurgosMoriwakiPhilipponSombra:aptv,BurgosPhilipponSombra:smthf}. These
techniques are well-suited for the study of toric metrics and their
associated height functions. In the sequel, we recall the basic
constructions and results.

Let $\K$ be a global field and $\T\simeq \G_{\textrm{m},\K}^{n}$ a split torus of
dimension $n$ over $\K$. Let
\begin{displaymath}
N=\Hom(\G_{m,\K},\T) \and M=\Hom(\T,\G_{m,\K})=N^{\vee}  
\end{displaymath}
be the lattices of cocharacters and of characters of $\T$,
respectively, and write $N_{\R}=N\otimes \R$ and $M_{\R}=M\otimes
\R$. 
We also fix an auxiliary norm $\|\cdot\|$ on $N_{\R}$. 

Let $v\in \fM_{\K}$. We denote by $\T^{\an}_{v}$ the $v$-adic
analytification of $\T$ and by $\SS_{v}$ its compact subtorus. In the
Archimedean case, $\SS_{v}$ is isomorphic to~$(S^{1})^{n}$ whereas, in the
non-Archimedean case, it is a compact analytic group, see \cite
[\S~4.2]{BurgosPhilipponSombra:agtvmmh} for a description.  Moreover,
there is a map $\val_{v}\colon \T^{\an}_{v}\to N_{\R}$, defined, in a
given splitting, by
  \begin{equation}\label{eq:44}
    \val_{v}(x_{1},\dots,x_{n})=(-\log|x_{1}|_{v},\dots,-\log|x_{n}|_{v}).
  \end{equation}
  This map does not depend on the choice of the splitting, and the
  compact torus $\SS_{v}$ coincides with its fiber over the point $0\in N_{\R}$.

Let $X$ be a proper toric variety over $\K$ with torus $\T$, described
by a complete fan $\Sigma$ on $N_{\R}$. 
To each cone
$\sigma \in \Sigma $ corresponds an affine toric variety $X_{\sigma
}$, which is an open subset of $X$, and an orbit $O(\sigma )$ of the
action of $\T$ on $X$.  The affine toric variety corresponding to the
cone $\sigma =\{0\}$ is the \emph{principal open subset} $X_{0}$. It
coincides with the orbit~$O(0)$ and  is canonically isomorphic to
the torus $\T$.

An $\R$-divisor $D$ on $X$ is \emph{toric} if it is invariant under
the action of $\T$. Such an $\R$-divisor defines a \emph{virtual
  support function} on $\Sigma$, that is a function 
\begin{displaymath}
\Psi_{D}\colon
N_{\R}\longrightarrow \R   
\end{displaymath}
whose restriction to each cone of the fan~$\Sigma$ is
linear. We  also associate to $D$ the subset of $M_{\R}$ given by 
\begin{displaymath}
 \Delta_{D}=\stab(\Psi_{D})=\{x\in M_{\R}\mid x\ge \Psi_{D}\}. 
\end{displaymath}
If $D$ is pseudo-effective, then $\Delta_{D}$ is a polytope and, otherwise,
it is the empty set.  Properties of the $\R$-divisor $D$ can
be read off from its associated virtual support function and polytope. In
particular, $D$ is nef if and only if $\Psi_{D}$ is concave, and $D$
is {big} if and only if $\Delta _{D}$ has nonempty interior.

  A quasi-algebraic metrized divisor $\ov
  D=(D,(\|\cdot\|_{v})_{v\in\fM_{\K}})$ on $X$ is \emph{toric} if and
  only if the $v$-adic metric $\|\cdot\|_{v}$ is invariant with
  respect to the action of~$\SS_{v}$, for all~$v$.  Such a toric
  metrized $\R$-divisor on $X$ defines a family of continuous
  functions $\psi_{\ov D,v}\colon N_{\R}\to \R$ indexed by the places
  of $\K$. For each $v\in \mathfrak{M}_{\K}$, this function is given,
  for~$p\in \T^{\an}_{v}$, by
\begin{equation}\label{eq:45}
  \psi_{\ov D,v}(\val_{v}(p))= \log \|s_{D}(p)\|_{v},
\end{equation}
where $s_{D}$ is the canonical rational $\R$-section of $D$ as in
\cite[\S~3]{BurgosMoriwakiPhilipponSombra:aptv}.  This adelic family
of functions satisfies that $|\psi_{\ov D,v}-\Psi_{D}|$ is bounded for
all $v$, and that $\psi_{\ov D,v}=\Psi_{D}$ for all $v$ except for a
finite number.  In particular, the stability set of each $\psi_{\ov
  D,v}$ coincides with $\Delta_{D}$.

For each $v\in \mathfrak{M}_{\K}$, we also consider the \emph{$v$-adic
  roof function} $ \vartheta_{\ov D,v}\colon \Delta_{D}\to \R$, that
is given by
\begin{displaymath}
  \vartheta_{\ov D,v}(x)= \psi_{\ov D,v}^{\vee}(x)= \inf_{u\in
    N_{\R}}(\langle u,x\rangle
  -\psi_{\ov D,v}(u)).
\end{displaymath}
This is an adelic family of continuous concave functions on
$\Delta_{D}$ which are zero except for a finite number of places. The
\emph{global roof function} $\vartheta_{\ov D}\colon \Delta_{D}\to \R$
is the weighted sum
\begin{displaymath}
  \vartheta_{\ov D}=\sum_{v\in\mathfrak{M}_{\K}}n_{v}\vartheta_{\ov D,v}.
\end{displaymath}

The essential minimum of $X$ with respect to $\ov D$ defined in
\eqref{eq:14} can be computed
as the maximum of its roof function \cite[Theorem
1.1]{BurgosPhilipponSombra:smthf}, that is
\begin{equation}
  \label{eq:10}
  \upmu^{\ess}_{\ov D}(X)= \max_{x\in \Delta_{D}}\vartheta_{\ov
    D}(x). 
\end{equation}

\begin{exmpl}\label{exm:6} Let $X$ be a proper toric variety over $\K$
  and $D$ a toric $\R$-divisor on~$X$.  The \emph{canonical metric}
  on $D$ is the metric characterized by the fact that, for each $v\in
  \mathfrak{M}_{\K}$ and $p\in
  \T^{\an}_{v}$, 
\begin{displaymath}
\log \|s_{D}(p)\|_{\can,v}=  \Psi_{D}(\val_{v}(p)), 
\end{displaymath}
see \cite[Proposition-Definition
4.3.15]{BurgosPhilipponSombra:agtvmmh}.  We denote this toric metrized
$\R$-divisor by $\ov D^{\can}$. For all $v\in \fM_{\K}$, 
\begin{displaymath}
\psi_{\ov
  D^{\can},v}=\Psi_{D} \and \vartheta_{\ov D^{\can},v}=0.
\end{displaymath}
In particular,  $\vartheta_{\ov D^{\can}}=0$ and $\upmu^{\ess}_{\ov
  D^{\can}}(X)= 0$. 
\end{exmpl}

Given a semipositive toric metrized $\R$-divisor $\ov D$ over $D$, its
associated metric functions are concave.  Conversely, every adelic
family of concave continuous functions $\psi_{v}\colon N_{\R}\to \R$,
${v\in \fM_{K}}$, 
with $|\psi_{v}-\Psi_{D}|$ bounded for all $v$ and such that
$\psi_{\ov D,v}=\Psi_{D}$ for all $v$ except for a finite number,
corresponds to a semipositive toric metrized $\R$-divisor over $D$
\cite[Proposition~4.19(1)]{BurgosMoriwakiPhilipponSombra:aptv}.  For
instance, a canonical toric metrized $\R$-divisor $\ov D^{\can}$ is
semipositive if and only if $\Psi_{D}$ is concave, which is equivalent
to the condition that $D$ is nef.

For the rest of this section, we suppose that  $X$ is a proper
toric variety over the global field $\K$ with torus $\T$, and that $\ov D$ is a
semipositive toric metrized $\R$-divisor with $D$ big.

We also fix the notation below. Recall from \S
\ref{sec:auxil-results-conv} that $\cP$ denotes the space of
probability measures on $N_{\R}$ endowed with the weak-$\ast$ topology
with respect to the space $\Cb(N_{\R})$, and that $\cE$ denotes the
subspace of probability measures with finite first moment.

\begin{notn}
  \label{def:14}
Let~$v\in \mathfrak{M}_{\K}$. We denote by $g_{i,v}$, $i=1,2$,  the
concave functions on~$\Delta_{D}$ given by
\begin{equation*}
g_{1,v} = \vartheta _{\ov D,v} \and
g_{2,v} =\sum_{w\in \fM_{\K}\setminus\{ v\}} \frac{n_{w}}{n_{v}}\vartheta _{\ov D,w}.
\end{equation*}
Thus $\vartheta _{\ov D}=n_{v}(g_{1,v}+g_{2,v})$.
We consider the convex subsets of $N_{\R}$ given by Definition~\ref{def:5}
\begin{equation}
  \label{eq:9}
   B_v = B(g_{1,v}, g_{2,v}),\  F_v = F(g_{1,v}, g_{2,v}) \and 
A_v = A_1(g_{1,v}, g_{2,v}).  
\end{equation}
Recall that $F_{v}$ is the minimal face of $A_v$ containing $B_{v}$.
We also denote by  $\Phi _{v}$  the function on
$\cE$ obtained by applying Definition \eqref{eq:27} to the set
$C=\Delta _{D}$ and the functions $g_{i,v}$, $i=1,2$.
\end{notn}

Given $v\in \fM_{\K}$ and a point~$p\in X(\ov{\K})$,
we
consider  the discrete probability measure on $N_{\R}$ defined by
\begin{equation*}
\nu_{p,v} =
(\val_{v})_* \mu_{p,v},
\end{equation*}
where~$\mu _{p,v}$ is the uniform discrete probability measure
on~$X_{{v}}^{\an}$ supported on the~$v$-adic Galois orbit of $p$ as
in~\eqref{eq:4}. This probability measure on $N_{\R}$ gives the modulus distribution of the $v$-adic Galois orbit of
the point $p$. The next result characterizes the limit behavior of this modulus
distribution for nets of points of small height.  


\begin{thm}
\label{thm:1 bis}
Let notation and hypothesis be as above. For each 
$v\in \mathfrak{M}_{\K}$ and  every $\ov{D}$-small
net~$(p_{l})_{l \in I}$ of algebraic points in the principal open
subset~$X_{0}$, the net~$(\nu_{p_l,v})_{l \in I}$ of measures in
$\cP$ has at least one cluster point. Every
such cluster point~$\nu_{v}$ lies in~$\cE$ and satisfies
\begin{equation}
  \label{eq:33}
  \supp(\nu_{v}) \subset F_v
\quad \text{ and } \quad
\exv[\nu_{v}] \in B_v.
\end{equation}
\end{thm}
The proof of Theorem~\ref{thm:1 bis} is given below, after a
definition and an auxiliary result.  

\begin{defn}
  \label{def:6}
  A \emph{centered adelic measure} $\bfnu$ on $N_{\R}$ is a collection of
  measures $\nu_{v}\in
  \cE$, $v\in \fM_{\K}$, such that $\nu_{v}=\delta_{0}$, the Dirac measure
 at the point $0\in N_{\R}$, for all but a finite number of places
  $v$, and such that
\begin{equation} \label{eq:46}
\sum_{v\in \fM_{\K}}
  n_v \exv[\nu_v] = 0.
\end{equation}
We denote by~$\cH_{\K}$ the set of all centered adelic measures on
$N_{\R}$. 
\end{defn}


We introduce the function~$\eta_{\ov{D}} \colon
\cH_{\K} \to \R$ defined by
  \begin{equation} \label{eq:36}
\eta_{\ov{D}} (\bfnu) =
- \sum_{v\in \fM_{\K}} n_v \int \psi_{\ov{D}, v} \dd \nu_v.  
  \end{equation}
  This function extends the notion of height of
  points to the space $\cH_{\K}$. Indeed, for $p\in X_{0}(\ov \K)$,
  the collection 
  \begin{equation}
    \label{eq:17}
    \bfnu_{p}=(\nu_{p,v})_{v\in \fM_{\K}}
  \end{equation}
  is a centered adelic measure on $N_{\R}$, because of the product
  formula in Proposition \ref{prop:15}\eqref{item:20}. Moreover, the
  canonical $\R$-section $s_{D}$ does not vanish at $p$ and, by
  Proposition~\ref{prop:12} and \eqref{eq:45},
\begin{multline}
  \label{eq:12}
     \h_{\ov D}(p)
 =
- \sum_{v} \frac{n_{v}}{\# \Gal(p)_{v}} \sum_{q\in
  \Gal(p)_{v}} \psi_{\ov{D},v}(\val_{v}(q))
 \\=
- \sum_{v} n_{v} \int \psi_{\ov{D},v} \dd \nu_{p,v} =\eta_{\ov
  D}(\bfnu_{p}).
\end{multline}


\begin{lem}
\label{lemm:1 bis}
For every centered adelic measure~$\bfnu=(\nu_{v})_{v\in \fM_{\K}}$,
\begin{equation}\label{eq:43}
\max_{v\in \fM_{\K}} - n_{v}\Phi_{v}(\nu_{v})
\le
\eta_{\ov{D}}(\bfnu) - \upmu^{\ess}_{\ov D}(X)
\le
 \sum_{v\in \fM_{\K}} -n_v \Phi_v(\nu_{v}).
\end{equation}
In particular, for~$p \in X_{0}(\ov{\K})$,
\begin{equation}\label{eq:47}
\max_{v\in \fM_{\K}} - n_{v}\Phi_{v}(\nu_{p,v})
\le
\h_{\ov{D}}(p) - \upmu^{\ess}_{\ov D}(X)
\le
 \sum_{v\in \fM_{\K}} -n_v \Phi_v(\nu_{p,v}).
\end{equation}
\end{lem}

\begin{proof}
  Let $\Delta_{D,\max}$ be the set of points maximizing the roof
  function $\vartheta_{\ov{D}}$ and choose $x\in
  \Delta_{D,\max}$. For each $v\in \fM_{\K}$,
  let~$\hphi_{i,v}\colon N_{\R} \to \R$, $i=1,2$, be the function
  defined by
  \begin{equation*}
\hphi_{i,v}(u) =g_{i,v}^{\vee}(u) - \langle x,u \rangle +g_{i,v}(x),
  \end{equation*}
where~$g_{i,v}$ denotes the concave function on $\Delta_{D}$
in~Notation~\ref{def:14} and $g_{i,v}^{\vee}$ its Legendre dual
as in \eqref{eq:11}.

Note that $\psi_{\ov{D},v} = g_{1,v}^{\vee}$. Using \eqref{eq:46} and
\eqref{eq:10}, we deduce that
\begin{equation*}
- \sum_{v} n_{v} \int \psi_{\ov{D},v} \dd \nu_{v}
=
\vartheta_{\ov{D}}(x) - \sum_{v} n_{v} \int \hphi_{1,v} \dd \nu_{v} =
\upmu^{\ess}_{\ov D}(X) - \sum_{v} n_{v} \int \hphi_{1,v} \dd \nu_{v}.
\end{equation*}
Thus
\begin{equation}
  \label{eq:38}
\eta_{\ov D}(\bfnu)  - \upmu^{\ess}_{\ov D}(X) =- \sum_{v} n_{v} \int \hphi_{1,v} \dd \nu_{v}.
\end{equation}

For each $v\in \fM_{\K}$, we get from the definition of~$\Phi_v$ that
$$ 
\Phi_v(\nu_{v})
=
\int \hphi_{1, v} \dd \nu_{v} + \hphi_{2, v}( - \exv[\nu_{v}]). 
$$
By Lemma~\ref{lem:flatification}\eqref{item:5}, the functions
$\hphi_{i,v}$ are nonpositive and so
\begin{equation}
  \label{eq:40}
\Phi_v(\nu_{v})
\le 
\int \hphi_{1, v} \dd \nu_{v}.
\end{equation}
The second inequality in \eqref{eq:43} then follows from \eqref{eq:38}
and~\eqref{eq:40}.

To prove the first inequality in  \eqref{eq:43}, fix~$v\in \fM_{\K}$.
By~\cite[ Propositions~2.3.1(1)
and~2.3.3(3)]{BurgosPhilipponSombra:agtvmmh},
\begin{equation}
  \label{eq:41}
\hphi_{2, v}=\boxplus_{w\ne v} \Big( \hphi_{1,w} \frac{n_{w}}{n_{v}}
\Big),  
\end{equation}
where $w$ runs over the places of
$\K$ different from $v$, the symbol~$\boxplus$ denotes the sup-convolution and, for a concave
function~$\psi$ and a nonzero constant~$\lambda$, the expression~$\psi
\lambda$ denotes the right multiplication as in~\cite[\S
2.3]{BurgosPhilipponSombra:agtvmmh}.

By the equality \eqref{eq:41}, the definitions of the sup-convolution
and the right
multiplication, and condition \eqref{eq:46}, we deduce
\begin{equation}
  \label{eq:42}
\hphi_{2, v}(-\exv[\nu_{v}]) \ge 
\sum_{w\ne v} \frac{n_{w}}{n_{v}} \hphi_{1,w} ( \exv [
  \nu_{w}]) .  
\end{equation}
By the concavity of  $\hphi_{1,w}$, we have  $\int \hphi_{1,w} \dd
\nu_{w}\le \hphi_{1,w} (\exv(\nu_{w}))$ for all $w\in \fM_{\K}$. 
Therefore, by  \eqref{eq:38} and
\eqref{eq:42}, 
\begin{multline*}
\eta_{\ov D}(\bfnu) - \upmu^{\ess}_{\ov D}(X) 
 \ge
- n_{v} \bigg( \int \hphi_{1,v} \dd \nu_{v}
+ \sum_{w\ne v} \frac{n_{w}}{n_{v}}\hphi_{1,w} (\exv(\nu_{w})) \bigg) 
\\ 
\ge
- n_{v} \bigg( \int \hphi_{1,v} \dd \nu_{v} + \hphi_{2,v} ( -
  \exv [\nu_{v}]) \bigg)
 =- n_{v} \Phi_{v}(\nu_{v}),
\end{multline*}
which proves the first inequality and completes the proof of
\eqref{eq:43}.  The inequalities in \eqref{eq:47} follow directly from
\eqref{eq:43} and \eqref{eq:12}.
\end{proof}

\begin{proof}[Proof of Theorem~\ref{thm:1 bis}]
  Let~$v\in \fM_{\K}$ and~$\Phi_v\colon \cE\to \R$ the function
  defined by~\eqref{eq:27} with $g_{1,v}$ and $g_{2,v}$ as
  in Notation~\ref{def:14}.  Since the net of points
  $(p_{l})_{l\in I}$ is $\ov D$-small,
\begin{displaymath}
  \lim_{l} \h_{\ov D}(p_{l})=\upmu^{\ess}_{\ov D}(X).
\end{displaymath}
From Lemma~\ref{lemm:1 bis}, we deduce
that
\begin{displaymath}
\lim_{l}\Phi_{v}(\nu_{p_{l},v})=0.  
\end{displaymath}
The theorem is then a direct consequence of Proposition~\ref{prop:2
  bis}.
\end{proof}

To state a partial converse of Theorem~\ref{thm:1 bis}, we need a
further definition. 

\begin{defn} \label{def:11}
The 
\emph{adelic Kantorovich--Rubinstein distance}~$W_{\K}$ on $\cH_{\K}$
is defined, for 
$\bfnu=(\nu_{v})_{v},\bfnu'=(\nu_{v}')_{v}\in \cH_{\K}$, by
$$ W_{\K}(\bfnu,\bfnu') =
\sum_{v} n_v W(\nu_v, \nu_v'),
$$
where $W$ denotes the Kantorovich--Rubinstein distance in $N_{\R}$ as
in Definition~\ref{def:9}.  By the definition of $\cH_{\K}$, 
there is only a finite number of nonzero terms in this sum.

The topology on $\cH_{\K}$ induced by this distance is called the
\emph{adelic KR-topology}.
\end{defn}

\begin{thm}\label{thm:1} 
With notation and hypothesis as before,
let $\bfnu=(\nu _{v})_{v\in
      \fM_{\K}}\in \cH_{\K}$ be a  centered adelic measure such that
    \begin{displaymath}
      \supp(\nu_{v})\subset F_{v}\and \exv[\nu_{v}]\in B_{v}
    \end{displaymath}
    for all $v$. Then
    there is a generic $\ov{D}$-small net~$(p_{l})_{l \in I}$ of
    algebraic points of~$X_{0}$ such that the net of measures
    $(\bfnu_{p_{l}})_{l\in I}$ converges to $\bfnu$ with respect to
    the adelic Kantorovich-Rubinstein distance.
\end{thm}

The proof of Theorem~\ref{thm:1} is given below, after some preliminary
results. 
The first result gives the main properties of the function $\eta_{\ov
  D}$. 

\begin{lem}
\label{lem:measure minimum}
The function~$\eta_{\ov{D}}$ is Lipschitz continuous with respect to~$W_{\K}$.
Moreover, for all $\bfnu=(\nu_{v})_{v\in \fM_{\K}}\in \cH_{\K}$,  
\begin{equation}
  \label{eq:8}
\eta_{\ov{D}}(\bfnu)
\ge
\upmu^{\ess}_{\ov D}(X),
\end{equation}
with equality if and only if $ \supp(\nu_{v}) \subset F_v$ and
$\exv[\nu_{v}] \in B_v$ for all~$v$.
\end{lem}

\begin{proof}
  Let $S\subset\fM_{\K}$ be a finite subset such that $\psi_{\ov
    D,v}=\Psi_{D}$ for all $v\notin
  S$. For~$\bfnu=(\nu_{v})_{v},\bfnu'=(\nu'_{v})_{v}\in \cH_{\K}$,
\begin{multline*}
| \eta_{\ov{D}} (\bfnu)- \eta_{\ov{D}} (\bfnu') |
 \le
\sum_v n_v \left| \int \psi_{\ov{D}, v} \dd \nu_v
- \int \psi_{\ov{D}, v} \dd \nu_v' \right|
\\ \le
\sum_v \Lip ( \psi_{\ov{D}, v} ) n_v W(\nu_v, \nu_v')
 \le \Big(\max_{x\in \Delta _{D}} \|x\| \Big) 
W_{\K} (\bfnu,\bfnu').
\end{multline*}
where the second inequality is given by the duality formula
\eqref{eq:48} and the last by the observation that $\Lip (
\psi_{\ov{D}, v})=\max_{x\in \Delta _{D}} \|x\|$ for all  $v$. This proves that 
$\eta_{\ov{D}}$ is Lipschitz continuous with respect to~$W_{\K}$.

As already remarked, the functions $\Phi_{v}$ are nonpositive.  By
Lemma~\ref{lemm:1 bis}, this implies the inequality~\eqref{eq:8}.
From the same result, it follows that the equality holds if and only
if~$\Phi_v(\nu_v) = 0$ for all $v$.  By
Proposition~\ref{prop:1}, this holds if and only if $ \supp(\nu_{v})
\subset F_v$ and $\exv[\nu_{v}] \in B_v$, completing the proof of
the lemma.
\end{proof}

From this lemma, we deduce as a direct consequence the next
characterization of algebraic points in toric varieties realizing the
essential minimum.

\begin{cor}\label{cor:3}
Let $p$ be an algebraic point of $X_{0}$. Then $\h_{\ov
  D}(p)=\upmu^{\ess}_{\ov D}(X)$ if and only if $
\supp(\nu_{p,v})\subset F_{v}$ and $\exv[\nu_{p,v}]\in B_{v}$ 
for all $v\in \fM_{\K}$.
\end{cor}

Let $H_{\K}\subset \bigoplus_{v\in
  \fM_{\K}} N_{\R}$ be the subspace defined by the equation
$\sum_{v}n_{v} u_{v}=0$. By sending the point $(u_{v})_{v}\in H_{\K}$ to the
adelic centered measure $(\delta _{u_{v}})_{v}\in \cH_{\K}$, we
identify $H_{\K}$ with a subspace of $\cH_{\K}$. 

\begin{cor}\label{cor:4} The minimum of the function $\eta_{\ov D}$ is
  equal to~$\upmu^{\ess}_{\ov D}(X)$ and it is attained at a point of the subspace $H_{\K}\subset \cH_{\K}$.
\end{cor}
\begin{proof}
  We denote by $\Delta _{D,\max}\subset \Delta _{D}$ the set of points
  where $\vartheta _{\ov D}$ attains its maximum.  Let $x\in \Delta
  _{D,\max}$. Since $\sum_{v} n_{v}\partial \vartheta _{\ov D,v}(x)
  =\partial \vartheta _{\ov D}(x)$ and $0\in \partial \vartheta _{\ov
    D}(x) $ we can find a point $\bfu =(u_{v})_{v}\in H_{\K}$ such
  that $u_{v}\in \partial \vartheta _{\ov D,v}(x) $. From the
  definition of $B_{v}$, it follows that $u_{v}\in B_{v} $ for every
  $v$. Thus, by Lemma~\ref{lem:measure minimum},
\begin{equation*}
   \upmu^{\ess}_{\ov D}(X)=\eta_{\ov D}(\bfu)=\min_{\bfnu\in \cH_{\K}}
  \eta_{\ov D}(\bfnu),
\end{equation*}
  as stated.
\end{proof}

We next show that the measures coming from algebraic points are dense
in $\cH_{\K}$.

\begin{prop}
\label{p:independence}
For every $\bfnu\in \cH_{\K}$ there is a generic net~$(p_{l})_{l\in
  I}$ of algebraic points of $X_{0}$ such that the net of associated
measures $(\bfnu_{p_{l}})_{l\in I}$ as in \eqref{eq:17}, converges to
$\bfnu$ with respect to the adelic KR-topology.
\end{prop}

\begin{proof}
  Put~$\bfnu=(\nu_v)_v$ and let ~$\varepsilon > 0$ be given.  Let $S$
  be a finite nonempty subset of~$\fM_{\K}$ such that
  $\nu_{v}=\delta_{0}$ for all $v\notin S$. By \cite[Theorem
  6.18]{Villani:oton}, we can approach, with respect to the
  KR-distance, each $\nu_{v}$, $v\in S$, by a probability measure with
  finite support. Therefore, we can find $d\ge1$ big enough and, for
  each $v\in S$, a sequence $(u_{v,1},\dots,u_{v,d})$ of points of
  $N_{\R}$, such that the probability measure $\nu'_{v}=
  \frac{1}{d}\sum_{i=1}^{d} \delta_{u_{v,i}}$ verifies
  \begin{displaymath}
W(\nu_{v},\nu'_{v}) <\frac{\varepsilon}{2\sum_{v\in S}n_{v}} \and
\exv[\nu_{v}']= \exv[\nu_{v}].
  \end{displaymath}
Set also $\nu_{v}'=\delta_{0}$ for $v\notin S$. 
Then  $\bfnu'=(\nu'_{v})_{v}\in \cH_{\K}$ and $W_{\K}(\bfnu,\bfnu')
<\frac{\varepsilon}{2}$.

Let $\F/\K$ be a finite extension of degree $d$ such that all places
in $S$ split completely, as given by \cite[Lemma
2.2]{BurgosPhilipponSombra:smthf}.  For each $v\in S$ and $w\in
\fM_{\F}$ such that $w\mid v$, we have $n_{w}=n_{v}/d$. We enumerate
the places above a given place $v\in S$ as $w(v,j)$, $j=1,\dots, d$.

Let $H_{\F}\subset \bigoplus_{w\in
  \fM_{\F}} N_{\R}$ be the subspace defined by the equation 
\begin{math}
  \sum_{w}n_{w} u_{w}=0. 
\end{math}
For each $v\in \fM_{\K}$ consider the element $\bfu\in H_{\F}$ given, for
$w\in \fM_{\F}$,  by
\begin{displaymath}
  u_{w}=\begin{cases} u_{v,j} &\text{ for } v\in S \text{ and } w=w(v,j)
\text{ with } 1\le j\le d,\\
0 & \text{ for } v\notin S \text{ and } w \mid v.    
  \end{cases}
\end{displaymath}
Consider the map $\val_{\F}\colon \T(\F) \to \bigoplus_{w\in
  \fM_{\F}}N_{\R}$ defined by $\val_{\F}=(\val_{w})_{w\in
  \fM_{\F}}$. This is a group homomorphism and so it can be extended
to a map
\begin{displaymath}
  \val_{\F}\colon  \T(\F)\otimes \Q \longrightarrow \bigoplus_{w\in
  \fM_{\F}}N_{\R}.
\end{displaymath}
By the product formula, the image of this map lies in the hyperplane
$H_{\F}$ and, by \cite[Lemma 2.3]{BurgosPhilipponSombra:smthf}, it is
dense with respect to the $L^{1}$-topology on $H_{\F}$. For
$\alpha\in \T(\F)$ and $r\in \Q$, we have 
\begin{multline}\label{eq:56}
\|\bfu-\val_{\F}(\alpha^{r}) \|_{L^{1}}=  
\sum_{v\in S} \frac{n_{v}}{d} \sum_{j=1}^{d}
\|u_{v,j}-\val_{w(v,j)}(\alpha^{r})\| + \sum_{v\notin S} \|\val_{v}(\alpha^{r})\|
\\=
\sum_{v}{n_{v}}\int \|u-u'\| \dd\lambda_{v}(u,u')
\end{multline}
for the probability measure $\lambda_{v}$ on $N_{\R}\times N_{\R}$ given by
$$\lambda_{v}=
\begin{cases}
\displaystyle   \frac{1}{d}
\sum_{j=1}^{d}\delta_{(u_{v,j},\val_{w(v,j)}(\alpha^{r}))} & \text{ if
} v\in S, \\
\delta_{(0,\val_{v}(\alpha^{r}))} & \text{ if } v\notin S.
\end{cases}
$$ 
This measure has marginals $\nu_{v}'$ and $\nu_{\alpha^{r},v}$. Hence, the quantity
in \eqref{eq:56} is an upper bound for the KR-distance
$W(\nu'_{v},\nu_{\alpha^{r},v}) $. It follows that we can choose $\alpha$ and $r$ such that for every torsion point $\omega$ of
$\T(\ov \K)$, the point $p=\omega \cdot  \alpha^{r}$ verifies
for every~$v$ that
\begin{equation*}
W(\nu'_{v},\nu_{p,v}) 
 <\frac{\varepsilon}{4\sum_{v'\in S}n_{v'}}
\end{equation*} 
and $\sum_{v\notin S} n_{v}W(\nu'_{v}, \nu_{p,v})<{\varepsilon}/{4}$.
Hence $W_{\K}(\bfnu',\bfnu_{p}) <\varepsilon/2$ and thus
$W_{\K}(\bfnu,\bfnu_{p}) <\varepsilon$.  

Since the orbit of 
$\alpha^{r}$ under the action of the group of torsion points of
$\T(\ov \K)$ is Zariski dense, we have shown that, given $\varepsilon >0$
and a nonempty open subset
$U\subset X$, we can choose $p\in U(\ov \K)$ satisfying
\begin{equation*}
  W_{\K}(\bfnu,\bfnu_{p}) <\varepsilon.
\end{equation*}

As in the proof of Proposition~\ref{prop:4}, let $I$ be the set of
hypersurfaces of $X$ ordered by inclusion. For each $Y\in
I$ choose a point $p_{Y}\in (X\setminus Y)(\ov \K)$ such that
\begin{displaymath}
    W_{\K}(\bfnu,\bfnu_{p_{Y}}) <\frac{1}{c(Y)}
\end{displaymath}
with $c(Y)$ the number of components of $Y$.
Thus, the net of algebraic points $(p_{Y})_{Y\in I}$ is generic and
the net of probability measures $(\bfnu_{p_{Y}})_{Y\in I}$ converges
to $\bfnu$ in the KR-topology, proving the result.
\end{proof}

\begin{proof}[Proof of Theorem~\ref{thm:1}] 
  Let $\bfnu=(\nu _{v})_{v}$ be a centered adelic measure on $N_{\R}$
  such that each measure $\nu_{v}$ satisfies the
  condition~\eqref{eq:33}. By Lemma~\ref{lem:measure minimum}, it
  satisfies
\begin{displaymath}
  \eta_{\ov{D}}(\bfnu) = \upmu^{\ess}_{\ov D}(X).
\end{displaymath}
Proposition~\ref{p:independence} implies that there is a
generic net~$(p_l)_{l \in I}$ of points in~$\T(\ov \K) = X_{0}(\ov \K)$
such
that~$(\bfnu_{p_l})_{l\in I}$ converges to~$\bfnu$ with respect to the distance
$W_{\K}$. 
On the other hand, by Lemma~\ref{lem:measure minimum} we also have
$$
\lim_{l} \h_{\ov{D}}(p_l) = \lim_{l}
\eta_{\ov{D}} (\bfnu_{p_l}) = \eta_{\ov{D}} (\bfnu) =
\upmu^{\ess}_{\ov D}(X), $$ 
and so the net~$(p_l)_{l \in I}$ is~$\ov{D}$-small. 
\end{proof}

\begin{cor} \label{cor:2} 
Let $v\in \fM_{\K}$. For every
  measure~$\nu_{v}\in \cE$ with $ \supp(\nu_{v})\subset F_{v}$ and $
  \exv[\nu_{v}]\in B_{v}$, there is a generic $\ov{D}$-small
  net~$(p_{l})_{l \in I}$ of algebraic points of~$X_{0}$ such that the
  net of measures $(\nu_{p_{l},v})_{l\in I}$ converges to $\nu_{v}$
  with respect to the Kantorovich-Rubinstein distance. In particular,
  $(\nu_{p_{l},v})_{l\in I}$ it also converges to $\nu_{v}$ in the weak-$\ast$
  topology with respect to $\Cb(N_{\R})$. 
\end{cor}

\begin{proof}
  Using the hypothesis that the point~$u_v = \exv[\nu_v]$ is in~$B_v$,
  for each~$v'\not = v$ we can find~$u_{v'}\in B_{v'}$ such that
  $u_{v'}=0$ except for a finite set $S'\subset \fM_{\K}$
  and
  \begin{displaymath}
\sum_{v' \in S'} n_{v'} u_{v'} = -n_{v}u_{v}.    
  \end{displaymath}
  For each~$v'\not = v$, put~$\nu_{v'} = \delta_{v'}$.  The statement
  then follows from Theorem~\ref{thm:1} applied to the centered adelic
  measure~$\bfnu=(\nu_v)_v$.
\end{proof}

Combining Theorems~\ref{thm:1 bis} and~\ref{thm:1}, we can obtain a
criterion for when the direct image under the valuation map of the
Galois orbits of a small net converges in the sense of measures. We
show that in this case, the limit measure is concentrated in a single
point.

\begin{cor}
\label{cor:3 bis}
Let $v\in \mathfrak{M}_{\K}$. The
  following conditions are equivalent:
\begin{enumerate}
\item \label{item:9 bis} for every $\ov D$-small net
  $(p_{l})_{l\in I}$ of algebraic points of $X_{0}$, the net of measures
  $(\nu_{p_l,v})_{l\in I}$ converges in the weak-$\ast$ topology with
  respect to $\Cb(N_{\R})$;
\item \label{item:27} for every generic $\ov D$-small net
  $(p_{l})_{l\in I}$ of algebraic points of $X_{0}$, the net of measures
  $(\nu_{p_l,v})_{l\in I}$ converges in the weak-$\ast$ topology with
  respect to $\Cc(N_{\R})$, the space of continuous functions
  on $N_{\R}$ with compact support;
  \item \label{item:8 bis} the face $F_{v}$
    contains only one point.  
\end{enumerate}
When  these equivalent conditions hold, the limit measures in
\eqref{item:9 bis} and \eqref{item:27} coincide with the Dirac
measure at the unique point of~$F_v$.
\end{cor}

\begin{proof}
  It is clear that \eqref{item:9 bis} implies \eqref{item:27},  and Theorem~\ref{thm:1
    bis} shows that 
  \eqref{item:8 bis} implies \eqref{item:9 bis}.  Now suppose that the
  face $F_{v}$ has more than one point. Since $F_{v}$ is the minimal
  face containing $B_{v}$, we can find distinct points
  $u_{0},u_{1},u_{2}\in F_{v}$ such that
  \begin{displaymath}
    u_{0}=\frac{u_{1}+u_{2}}{2}\in B_{v}.
  \end{displaymath}
  The probability measures $\delta_{u_{0}}$ and
  $\frac{1}{2}\delta_{u_{1}}+ \frac{1}{2}\delta_{u_{2}}$ satisfy the conditions
  \eqref{eq:33}. By Corollary~\ref{cor:2}, we can find generic $\ov
  D$-small nets $(p_{l})_{l\in I}$ and $(q_{l})_{l\in I}$ such that
  the nets of measures $(\nu_{p_{l},v})_{l\in I}$ and
  $(\nu_{q_{l},v})_{l\in I}$ respectively converge to
  \begin{displaymath}
  \delta_{u_{0}}\and \frac{1}{2}\delta_{u_{1}}+\frac{1}{2}\delta_{u_{2}}  
  \end{displaymath}
  in the KR-topology, and hence in the weak-$\ast$ topology with respect to
  $\Cc(N_{\R})$. Combining these nets, we can obtain a net
  that does not converge in this weak-$\ast$ topology. Hence the condition
  \eqref{item:27} implies the condition \eqref{item:8 bis}. 

  The last statement follows from Theorem \ref{thm:1 bis}.
\end{proof}

When any of the equivalent conditions of Corollary~\ref{cor:3 bis}
holds we say that the metrized divisor $\ov D$ satisfies the
\emph{modulus concentration property} at the place $v$. Thus
Corollary~\ref{cor:3 bis} gives us a criterion for the modulus
concentration property at a place.  We next give a criterion for the
modulus concentration property at all  places simultaneously, that
can be directly read from the roof function.  Before giving it, we
need some preliminary results and a definition.

\begin{defn}\label{def:16}
  A semipositive toric metrized $\R$-divisor $\ov D$ with $D$ big is
  called \emph{monocritical} if the minimum of $ \eta_{\ov D}$ in $\cH_{\K}$ is
  attained at a unique point. If this is the  case, by Corollary~\ref{cor:4}, 
  the minimum is attained at a point of $H_{\K}$. This point
  is called \emph{the critical point} of $\ov D$.
\end{defn}

\begin{exmpl}\label{exm:9}
  Let $\ov D^{\can}$ be a nef and big toric $\R$-divisor equipped with
  the canonical metric as in Example~\ref{exm:6}. Then all its local
  roof functions are zero. Taking a point $x$ in the interior of the
  polytope, we have $\partial \vartheta _{\ov D,v}({x})=\{0\}$ for
  every~$v$. Hence~$F_{v}=\{0\}$ for every $v$ and $\ov D$ is
  monocritical with critical point $\bfzero\in H_{\K}$.
\end{exmpl}

Recall that $\Delta _{D,\max}$ denotes the convex set of points
of $\Delta _{D}$
where $\vartheta _{\ov D}$ attains its maximum. 

\begin{prop}\label{prop:14} The following conditions are equivalent:
  \begin{enumerate}
  \item \label{item:8} the metrized $\R$-divisor $\ov D$ is monocritical;
  \item \label{item:29} for every point $x\in \Delta _{D,\max}$, the
    set
    \begin{equation}\label{eq:73}
      H_{\K}\cap\prod_{v\in \fM_{\K}} \partial \vartheta _{\ov D,v}(x) 
    \end{equation}
    contains a unique element $\bfu =(u_{v})_{v}\in H_{\K}$ and, for
    $v\in \fM_{\K}$, the point
    $u_{v}$ is a vertex of $\partial  \vartheta _{\ov D,v}(x)$;
  \item \label{item:25} for every point $x\in \Delta _{D,\max}$, the point $0$ is
    a vertex of $\partial  \vartheta _{\ov D}({x})$;
  \item \label{item:9} there exists a point $x\in \Delta _{D,\max}$ such that $0$ is
    a vertex of $\partial \vartheta _{\ov D}({x}) $;
  \item \label{item:26} for all  $v\in \mathfrak{M}_{\K}$, the set $F_{v}$ 
    contains only  one point.
  \end{enumerate}
  When these equivalent conditions hold, $F_{v}=\{u_{v}\}$ for every
  $v$ and $\bfu$ is the critical point of $\ov D$.
\end{prop}
\begin{proof}
  We prove first that \eqref{item:8} implies \eqref{item:29}. Assume
  that $\ov D$ is monocritical. If the set \eqref{eq:73} contains two
  elements $\bfu_{i} =(u_{i,v})_{v}\in H_{\K}$, $i=1,2$. Then the
  measures $\bfnu_{i}=(\delta _{u_{i,v}})_{v}\in \cH_{\K}$, $i=1,2$, satisfy
  that $ \supp(\delta _{u_{i,v}}) \subset B_v$ for each~$v$. In particular,
$ \supp(\delta _{u_{i,v}}) \subset F_v$ and
$\exv[\delta _{u_{i,v}}] \in B_v$. Thus by Lemma~\ref{lem:measure minimum}
\begin{displaymath}
  \eta_{\ov D}(\bfu_{1})=\eta_{\ov D}(\bfu_{2})=\min _{\bfnu\in
    \cH_{\K}}\eta_{\ov D}(\bfnu)
\end{displaymath}
contradicting the hypothesis that $\ov D$ is monocritical, and showing 
that~\eqref{eq:73} contains a unique element.

Assume now that the set \eqref{eq:73} contains a single point $\bfu
=(u_{v})_{v}\in H_{\K}$ and there is a place $v_{0}\in \fM_{\K}$
such that $u_{v_{0}}$ is not a vertex of $\partial  \vartheta _{\ov
  D,v_{0}}(x)$. Then we can find two points $u_{v_{0},1},u_{v_{0},2}\in \partial
\vartheta _{\ov D,v_{0}}(x)$ such that
\begin{displaymath}
  u_{v_{0}}=\frac{u_{v_{0},1}+u_{v_{0},2}}{2}.
\end{displaymath}
We consider the measure $\bfnu_{1}=(\delta _{u_{v}})_{v}$ and the
measure $\bfnu_{2}=(\nu_{v})_{v}$ defined by
\begin{displaymath}
  \nu_{v}=
  \begin{cases}
    \delta _{u_{v}}&\text{ if }v\not = v_0,\\
    {\displaystyle \frac{
      \delta_{u_{v_{0},1}}+\delta_{u_{v_{0},2}}}{2}} &\text{ if }v = v_0.
  \end{cases}
\end{displaymath}
Then~$\bfnu_{2}$ is in~\eqref{eq:73} and, again by
Lemma~\ref{lem:measure minimum}, we have that
\begin{displaymath}
  \eta_{\ov D}(\bfnu_{1})=\eta_{\ov D}(\bfnu_{2})=\min _{\bfnu\in
    \cH_{\K}}\eta_{\ov D}(\bfnu)
\end{displaymath}
contradicting the hypothesis that $\ov D$ is monocritical, and
completing the proof of~\eqref{item:29}.

  Assume that \eqref{item:29} is true and fix $x\in \Delta
  _{D,\max}$. Let $S\subset \fM_{\K}$ be the
  finite set of places where $u_{v}\not = 0$ or $\vartheta _{\ov D,v}$
  is not identically zero. We have that   
  \begin{displaymath}
    \partial \vartheta _{\ov D}({x})=\sum_{v\in S}n_{v}\partial \vartheta
    _{\ov D,v}({x}).  
  \end{displaymath}
  Moreover, \eqref{item:29} implies that the equation
  \begin{displaymath}
    0=\sum _{v\in S}n_{v}a_v \quad\text{ with }a_v\in \partial \vartheta
    _{\ov D,v}(x) 
  \end{displaymath}
  has a unique solution $a_v=u_{v}$ and this solution satisfies that
  $a_v$ is a vertex of $\partial \vartheta
    _{\ov D,v}(x)$. Therefore, by Lemma~\ref{lemm:4}
  we deduce that $0$ is a vertex of $\partial \vartheta _{\ov
    D}({x})$. Hence \eqref{item:29} implies \eqref{item:25}. 

Since $\Delta _{D,\max}$ is nonempty,
  \eqref{item:25} implies \eqref{item:9}.

  Assume now that \eqref{item:9} is true.  For each $v$, let  $g_{1,v}$
  and $g_{2,v}$ be the continuous concave functions on $\Delta_{D}$
  in Notation~\ref{def:14}.  Since $\vartheta _{\ov
    D}=n_v g_{1,v}+n_vg_{2,v}$, 
  \begin{displaymath}
    \partial \vartheta _{\ov D}({x}) = n_{v}\partial g_{1,v}({x}) +
    n_{v}\partial g_{2,v}({x}) .  
  \end{displaymath}
  Lemma~\ref{lemm:4} and the definition of the set $B_{v}$ imply that this
  set  contains one single point~$u_v$, and that
  this point is a vertex of both $\partial g_{1,v}({x}) $ and of
  $-\partial g_{2,v}({x}) $. Hence $B_{v}$ is already a face of
  $\partial g_{1,v}({x}) $. Thus 
  $F_{v}=B_{v}=\{u_{v}\}$ and so \eqref{item:9} implies
  \eqref{item:26}.

  By Lemma~\ref{lem:measure minimum} it is clear that \eqref{item:26} implies
  \eqref{item:8} finishing the proof of the equivalence. 

  Assume now that $\ov D$ is monocritical. Since by Lemma
  \ref{lem:measure minimum} the point $\bfu$ in
  \eqref{item:29} satisfies that $\eta_{\ov D}(\bfu)=\min _{\bfnu\in
    \cH_{\K}}\eta_{\ov D}(\bfnu)$, it is the critical point. Following
  the proof of the equivalence we deduce that $F_{v}=\{u_v\}$ proving
  the last statement.
\end{proof}

For a given toric metrized $\R$-divisor, the condition of being
monocritical and its critical point behave well with respect to scalar
extensions. The following result
follows from the compatibility of toric metrics with scalar extensions in
\cite[Proposition 4.3.8]{BurgosPhilipponSombra:agtvmmh}.

\begin{prop} \label{prop:13} Let~$X$ and $\ov D$ as
  before. Let~$\F\subset\ov \K$ be a finite extension of~$\K$ and
  write $\ov D_{\F}$ for the toric metrized $\R$-divisor on $X_{\F}$
  obtained by scalar extension. If $\ov D$ is monocritical with
  critical point $(u_{v})_{v\in \fM_{\K}}$, then $\ov D_{\F}$ is also
  monocritical and its critical point $(u_{w})_{w\in \fM_{\F}}$ is
  given by $u_{w}=u_{v}$ for all $v\in \fM_{\K}$ and $w$ over $v$.
\end{prop}

We now give the criterion for modulus concentration at every place.

\begin{thm} \label{thm:3}
Let~$X$ and $\ov D$ be as before. The
  following conditions are equivalent:
  \begin{enumerate}
  \item \label{item:7} for every $\ov D$-small net
  $(p_{l})_{l\in I}$ of algebraic points of $X_{0}$ and every place $v\in
  \fM_{\K}$, the net of measures 
  $(\nu_{p_l,v})_{l\in I}$ converges.
  \item \label{item:12} the metrized $\R$-divisor $\ov D$ is monocritical; 
  \end{enumerate}
  When these equivalent conditions hold,
  \begin{displaymath}
    \lim_{l\in I} \nu_{p_l,v}=\delta _{u_{v}},
  \end{displaymath}
  where $(u_{v})_{v}$ is the critical point of $\ov D$.
\end{thm}
\begin{proof} The theorem follows directly from  Corollary~\ref{cor:3
    bis} and Proposition~\ref{prop:14}.
\end{proof}

When there is modulus concentration for every place, we can show
that the convergence holds not only in the weak-$\ast$ topology with
respect to $\Cb(N_{\R})$ but even in
the stronger adelic KR-topology.

\begin{thm}
\label{thm:6}
Let~$X$ and $\ov D$ be as before. Assume
that $\ov D$ is monocritical. Let $\bfu = (u_{v})_{v}$ be the critical point
of $\ov D$ and set $\bfdelta_{\bfu} = (\delta _{u_{v}})_{v}\in \cH_{\K}$. Then, for every $\ov D$-small net $(p_{l})_{l\in I}$ of
algebraic points of  $X_{0}$, the net of centered adelic measures
$(\bfnu_{p_{l}})_{l\in I}$ converges to $\bfdelta_{\bfu}$ in the adelic
KR-topology.  In particular, for every $v\in \fM_{\K}$, the net of
measures $(\nu_{p_{l},v})_{l\in I}$ converges to $\delta_{u_{v}}$ in
the KR-topology.
\end{thm}
 
\begin{proof}
  For each $v\in \fM_{\K}$, let $f_{v}\colon N_{\R}\to \R$ be the
  function given by
\begin{displaymath}
  f_{v}(u)= \psi_{\ov D,v}(u) -\Psi_{D}(u-u_{v}).
\end{displaymath}
This is an adelic family of bounded continuous functions on $N_{\R}$
with  $f_{v}=0$ for all but a finite number of $v$. 
Consider then the function  $\eta'\colon \cH_{\K}\to \R$  given by 
\begin{equation*}
  \eta'(\bfnu)= \eta_{\ov D}(\bfnu)+\sum_{v} n_{v}\int
  f_{v}\dd\nu_{v}=-\sum_{v}n_{v}\int
  \Psi_{D}(u-u_{v}) \dd\nu_{v} .
\end{equation*}
Since the net $(p_{l})_{l\in I}$ is $\ov D$-small,  
\begin{displaymath}
\lim_{l}\eta_{\ov D}(\bfnu_{p_{l}})= \lim_{l}\h_{\ov D}({p_{l}})=\upmu_{\ov D}^{\ess}(X).  
\end{displaymath}
By Theorem
\ref{thm:3}, the net of measures $(\nu_{p_{l},v})_{l\in
  I}$ converges to $\delta_{u_{v}}$. Using Corollary~\ref{cor:4}, we deduce that 
\begin{equation}  \label{eq:24}
  \lim_{l}\eta'(\bfnu_{p_{l}})= 
\lim_{l}\eta_{\ov D}(\bfnu_{p_{l}})
+\sum_{v} n_{v}\int
  f_{v}\dd\delta_{u_{v}} = 
\upmu_{\ov D}^{\ess}(X)+ \sum_{v} n_{v}
  f_{v}({u_{v}})=0.
\end{equation}

Choose a point $x$ in the interior of $\Delta_{D}$. 
Then there is a constant
$c>0$ such that, for all $u\in N_{\R}$, 
\begin{displaymath}
 \|u\|\le -c\, (\Psi_{D}-x)(u).  
\end{displaymath}
It follows from  the definition of the Kantorovich-Rubinstein distance that, for
each~$v\in \fM_{\K}$, 
\begin{displaymath}
  W(\nu_{p_{l},v}, \delta_{u_{v}})  \le \int \|u-u_{v}\|
  \dd\nu_{p_{l},v}(u). 
\end{displaymath}
Hence
\begin{multline*}
W_{\K}(\bfnu_{p_{l}},\bfdelta_{\bfu})   \le \sum_{v}n_{v}  \int \|u-u_{v}\|
  \dd\nu_{p_{l},v}(u)\\ \le -c\sum_{v}n_{v}\int
  (\Psi_{D}-x)(u-u_{v}) \dd\nu_{p_{l},v} (u)=c\, \eta'(\bfnu_{p_{l}}),
\end{multline*}
where the last equality follows from the product formula in
Proposition \ref{prop:15}\eqref{item:20}.  By \eqref{eq:24}, this
distance converges to 0, completing the proof.
\end{proof}

\section{Equidistribution of Galois orbits and the Bogomolov property}
\label{sec:equid-galo-orbits}

We turn to the study of the limit measures of Galois orbits of $\ov
D$-small nets of algebraic points in toric varieties.  In this
section, we denote by $X$ a proper toric variety over a global field
$\K$ and $\ov D$ a toric metrized $\R$-divisor on $X$ with $D$ big.
For~$v\in \fM_{\K}$, recall that $\val_{v}\colon \T_{v}^{\an}\to
N_{\R}$ denotes the valuation map, defined in~\eqref{eq:44}.

We first describe the limit measures in the monocritical case.

\begin{defn}
  \label{def:10}
  Given $v\in \fM_{\K}$ and $u\in N_{\R}$, the probability measure
  $\lambda_{\SS_{v},u} $ on~$X_{{v}}^{\an}$ is defined as follows.

  \begin{enumerate}
  \item When $v$ is Archimedean, note that $\val_{v}^{-1}(u)=\SS_{v} \cdot p$ for
     any point $p\in \val_{v}^{-1}(u)$ and where
     $\SS_{v}=\val_{v}^{-1}(0)\simeq (S^{1})^{n}$ is the compact torus of
    $\T^{\an}_{v}$.  In this case,
    $\lambda _{\SS_{v},u}$ is the direct image under the translation by
    $p$ of the Haar probability measure of~$\SS_{v}$.
  \item When $v$ is non-Archimedean,  consider the multiplicative seminorm
    on the group algebra $\C_{v}[M]\simeq \C_{v}[x_{1}^{\pm1}, \dots,
    x_{n}^{\pm1}]$ that, to a Laurent polynomial $\sum_{m\in
      M}\alpha_{m}\chi^{m}$, assigns the value $\max_{m}
    (|\alpha_{m}|_{v}\e^{-\langle m,u\rangle})$. This seminorm gives a
    point, denoted by~$\theta(u)$, in the Berkovich space
    $X_{{v}}^{\an}$.  The point~$\theta(u)$  is of type II or III in Berkovich's
    classification and lies in the preimage $\val^{-1}_{v}(u)$. We then
    set $\lambda _{\SS_{v},u}=\delta _{\theta(u)}$, the Dirac measure at
    this point.
  \end{enumerate}
\end{defn}

The following result corresponds to Theorem~\ref{thm:10} in the
introduction, and shows that modulus concentration at every place
implies the equidistribution property at every place.  Due to the
existing equidistribution theorems in the literature, we restrict its
statement to divisors (rather than $\R$-divisors). 

\begin{thm}\label{thm:5}
  Let $X$ be a proper toric variety over $\K$ and $\overline D$ a
  semipositive toric metrized divisor on $X$ with $D$ big.
  The following conditions are equivalent:
\begin{enumerate}
  \item \label{item:16} for every generic $\ov D$-small net $(p_{l})_{l\in I}$
    of algebraic points of $X_{0}$ and every place $v\in \fM_{\K}$,
    the net of probability measures $(\mu_{p_l,v})_{l\in I}$ on $X_{{v}}^{\an}$
    converges;
  \item \label{item:18} the metrized divisor $\ov D$ is monocritical. 
\end{enumerate}

When these equivalent conditions hold, the limit measure in
\eqref{item:16} is~$\lambda _{\SS_{v},u_{v}}$, with $u_{v}\in
N_{\R}$ the $v$-adic component of the critical point of $\ov D$.
\end{thm}


The proof of Theorem~\ref{thm:5} is done by reduction to the
quasi-canonical case.
The following is the characterization of quasi-canonical toric
metrized $\R$-divisors in \cite{BurgosPhilipponSombra:smthf}.

\begin{prop}
  \label{prop:10}
  Let $X$ be a proper toric variety over $\K$ and $\ov D$ a
  semipositive toric metrized $\R$-divisor on $X$ with $D$ big. 
The following conditions are equivalent: 
\begin{enumerate}
\item \label{item:22}  $\ov D$
  is quasi-canonical (Definition~\ref{def:1});
\item \label{item:23} $\vartheta_{\ov D}$ is constant;
\item \label{item:24} there is $\bfu=(u_{v})_{v}\in H_{\K}$ and
  $(\gamma_{v})_{v}\in \bigoplus_{v\in  \fM_{\K}}\R$ such that
  \begin{equation*}
    \psi _{\ov D,v}(u)=\Psi_{D}(u-u_{v}) -\gamma_{v}
  \end{equation*}
  for all $v\in \mathfrak{M}_{\K}$ and $ u\in N_{\R}$.
\end{enumerate}
\end{prop}

\begin{proof}
  The equivalence of \eqref{item:22} and
  \eqref{item:24} is given by \cite[Corollary
  4.7]{BurgosPhilipponSombra:smthf}. The equivalence of
  \eqref{item:22} and \eqref{item:23} is given in the course of the proof of
  \cite[Proposition~4.6]{BurgosPhilipponSombra:smthf}, recalling
  that~$\vol(D) = \deg_D(X)$ and noting that, since by assumption~$\ov D$ is
  semipositive, $\avol_{\chi}(\ov D) = \h_{\ov D}(X)$.
\end{proof}

The following result gives the key step in the proof of Theorem
\ref{thm:5}. 

\begin{prop} \label{prop:6} Let $X$ be a proper toric variety over
  $\K$ and $\overline D$ a monocritical metrized $\R$-divisor on $X$
  with critical point $\bfu=(u_{v})_{v\in \fM_{\K}}$.  Let $\ov D'$ be the toric
  metrized $\R$-divisor over $D$ corresponding to the family of
  concave functions $ \psi _{\ov D',v}\colon N_{\R}\to \R$, $v\in
  \fM_{  \K}$, given by
  \begin{equation}\label{eq:16}
    \psi _{\ov D',v}(u)=\Psi _{D}(u-u_{v}) . 
  \end{equation}
  Then $\ov D'$ is quasi-canonical and every $\ov D$-small net of
  algebraic points of $X_{0}$ is also~$\ov D'$-small.
\end{prop}

\begin{proof} 
  The fact  that $\ov D'$ is quasi-canonical is given by
  Proposition~\ref{prop:10}. 

  Let $(p_{l})_{l\in I}$ be a $\ov D$-small net of algebraic points of
  $X_{0}$. By Theorem~\ref{thm:6}, the net of centered adelic measures
  $(\bfnu_{p_{l}})_{l\in I}$ converges to
  $\bfdelta_{\bfu}=(\delta_{u_{v}})_{v}$ with respect to the adelic
  KR-distance. By Lemma~\ref{lem:measure minimum}, the function
  $\eta_{\ov D'}$ is continuous with respect to this distance. Using
  \eqref{eq:12}, we deduce that
  \begin{displaymath}
    \lim_{l}\h_{\ov D'}(p_{l})= \lim_{l}\eta_{\ov
      D'}(\bfnu_{p_{l}})=\eta_{\ov D'}(\bfdelta_{\bfu})=0.
  \end{displaymath}
On the other hand, 
  $\vartheta_{\ov D',v}=u_{v}$ for each $v$. Since the critical point $\bfu$ lies
  in the subspace $H_{\K}$, we have that $\vartheta_{\ov
    D'}=\sum_{v}n_{v}u_{v}=0$. Hence, 
  \begin{displaymath}
 \upmu^{\ess}_{\ov
    D'}(X)=\max_{x\in\Delta_{D}}\vartheta_{\ov D'}(x)=0. 
  \end{displaymath}
Thus  $(p_{l})_{l\in I}$ is $\ov D'$-small, as stated.
\end{proof}

\begin{proof}[Proof of Theorem~\ref{thm:5}] 
  Suppose that the condition \eqref{item:16} holds. Given a generic
  $\ov D$-small net $(p_{l})_{l\in I}$  of algebraic points of
  $X_{0}$ and $v\in \fM_{\K}$, the net of measures
  $(\mu_{p_{l},v})_{l\in I}$ converges weakly with respect to the
  space $\cC(X_{{v}}^{\an})$.  Hence, the net of direct
  images $(\nu_{p_{l},v})_{l\in I}$ converges weakly with respect to
  the space $\Cc(N_{\R})$. By Corollary~\ref{cor:3 bis}, for
  each $v$,  the
  face $F_{v}$ contains only one point. 
  Proposition~\ref{prop:14} then implies that~$\ov D$ is
  monocritical, giving the condition~\eqref{item:18}.

  Now suppose that the condition \eqref{item:18} holds.  Let $(Y,E)$
  be the polarized toric variety associated to the polytope
  $\Delta_{D}$.  By the characterization of semipositive toric metrics
  in \cite[Theorem 4.8.1]{BurgosPhilipponSombra:agtvmmh}, the metric
  in $\ov D$ induces a semipositive toric metric on $E$, and we denote
  by $\ov E$ the corresponding toric metrized divisor. We have that
  $\psi_{\ov E,v}=\psi_{\ov D,v}$ for all $v$, and so $\ov E$ is also
  monocritical with the same critical point as~$\ov D$.

Let
$$
\ov E'=(E,\|\cdot\|_{v}')_{v\in\fM_{\K}}
$$ 
be the ample divisor $E$ on $Y$ equipped with the quasi-canonical
toric metric given by Proposition~\ref{prop:6}, with~$\ov D$ replaced
by~$\ov E$.  Let $(p_{l})_{l\in
  I}$ be a generic $\ov D$-small net of algebraic points of
$X_{0}=\T=Y_{0}$.  It is also a generic $\ov E$-small net of algebraic
points of $Y_{0}$. By Proposition~\ref{prop:6} with~$\ov D$ replaced
by~$\ov E$, it is also $\ov
E'$-small.

By Theorem~\ref{thm:4}, for each place~$v$ the net $(\mu
_{p_{l},v})_{l\in I}$ converges to the normalized Monge-Amp\`ere
measure $\mu_{v}=\frac{1}{\deg_{E}(Y)}{\chern_{1}(E,
  \|\cdot\|'_{v})^{\wedge n}}$ on $Y_{{v}}^{\an}$.  Consider the real
Monge-Amp\`ere measure $\cM(\psi _{\ov E',v})$ associated to the
$v$-adic metric in $\ov E'$ as in
\cite[Definition~2.7.1]{BurgosPhilipponSombra:agtvmmh}. By the
explicit formula \eqref{eq:16} and
\cite[Example~2.7.5]{BurgosPhilipponSombra:agtvmmh},
  \begin{displaymath}
    \cM(\psi _{\ov
    E',v})=\vol_{M}(\Delta_{D}) \delta _{u_{v}}=\frac{\deg_{E}(Y)}{n!}\delta _{u_{v}}.
  \end{displaymath}
Then \cite[Theorem 4.8.11]{BurgosPhilipponSombra:agtvmmh} implies that
 $\mu_{v}=\lambda_{\SS_{v},u_{v}}$.
 Therefore, the net of measures $(\mu _{p_{l},v})_{l\in I}$ on
 $X_{{v}}^{\an}$ converges to $\lambda_{\SS_{v},u_{v}}$, giving
 the condition~\eqref{item:16} and the last statement in the theorem.
\end{proof}

\begin{exmpl}\label{exm:4}
  Let $\ov D^{\can}$ be a big and nef toric divisor on~$X$ equipped with the
  canonical metric. Following Example~\ref{exm:9}, this toric metrized
  divisor is monocritical with critical point $\bfzero \in
  H_{\K}$. Hence, it satisfies the $v$-adic equidistribution property
  with limit measure $\lambda_{\SS_{v},0}$, for every $v\in \fM_{\K}$.
\end{exmpl}

In \cite{Bilu:ldspat}, Bilu gave an equidistribution theorem for
Galois orbits of sequences of points of small canonical height.  This
result is restricted to number fields and Archimedean places.
However, and in contrast with the previous example, this result holds
not just for generic, but for \emph{strict} sequences of points, that
is, sequences that eventually avoid any given proper torsion
subvariety.  This stronger version of the equidistribution property
was used in a crucial way in \emph{loc. cit.} to prove the Bogomolov
property for the canonical height.

Here we extend this version of the equidistribution property to
monocritical metrized $\R$-divisors on toric varieties
(Theorem~\ref{thm:7}) and deduce from it the Bogomolov property
(Theorem~\ref{thm:12} in the introduction, or Theorem~\ref{thm:2}
below).  Our proofs are similar to Bilu's and use Fourier analysis.
Hence, for the rest of the section we restrict to the case when~$\K$ is a
number field and we only study the equidistribution over the Archimedean
places.  Following Remark~\ref{rem:8}, we restrict without loss of
generality to sequences, instead of nets.

To formulate this extension, we have to modify slightly the notion of
strict sequence. First we recall some standard terminology: a
\emph{subtorus} of $\T$ is an algebraic subgroup of $\T$ that is
geometrically irreducible, a \emph{translate of a subtorus} is a
subvariety of $\T_{\ov \K}$ that is the orbit of a point
$p\in \T(\ov \K)$ by a subtorus, and a \emph{torsion subvariety} is
a translate of a subtorus by a torsion point of the group $\T(\ov
\K)\simeq (\ov \K^{\times})^{n}$.

\begin{defn}
  \label{def:12}
  A sequence $(p_{l})_{l\ge1}$ of algebraic points of
  $\T$ is \emph{strict} if, for every  translate of a subtorus $U
  \subsetneq \T_{\ov \K}$, there is $l_{0}\ge 1$ such that
  \begin{math}
p_{l}\notin U (\ov \K)
\end{math} for all $l\ge l_{0}$.  Equivalently, $(p_{l})_{l\ge 1}$ is
strict if, for every $m\in M\setminus \{0\}$ and every point $q\in
X_{0}(\ov \K)$, there is $l_{0}\ge 1$ such that
  \begin{math}
    \chi^{m}(p_{l})\ne  \chi^{m}(q) 
  \end{math}
 for all $l\ge l_{0}$.  
\end{defn}

\begin{thm}\label{thm:7}
  Let $X$ be a proper toric variety over a number field $\K$ and
  $\overline D$ a monocritical metrized $\R$-divisor on $X$. Then, for
  every strict $\ov D$-small sequence~$(p_{l})_{l\ge 1}$ of algebraic
  points of $X_{0}$ and every Archimedean place $v\in \fM_{\K}$, the
  sequence $(\mu_{p_l,v})_{l\ge1}$ converges to the
  probability measure $ \lambda _{\SS_{v},u_{v}}$, with $u_{v}\in
  N_{\R}$ the $v$-adic component of the critical point of $\ov D$.
\end{thm}

\begin{proof}
Let $(p_{l})_{l\ge 1}$ be a
  strict  $\ov D$-small sequence of algebraic points of
  $X_{0}$. 
  For each $m\in M\setminus \{0\}$ consider the character
  \begin{displaymath}
    \chi^{m}\colon \T\longrightarrow \G_{\textrm{m},\K}.
    \end{displaymath}
    Since $(p_{l})_{l\ge 1}$ is strict, the sequence
    $(\chi^{m}(p_{l}))_{l\ge 1}$ is generic.

    We embed $\G_{{\textrm{m},\K}}\hookrightarrow \P^{1}_{\K}$ as the
    principal open subset. Let 
    $D_{0}=\div(x_{0})$ be the divisor at infinity on $\P^{1}_{\K}$,
    that we equip with the quasi-canonical toric metric corresponding
    to the adelic family of functions $\psi_{\ov D_{0}^{m},v}\colon
    \R\to \R$ given by
\begin{displaymath}
  \psi_{\ov
 D_{0}^{m},v}(u)= \min(0, u-\langle m,u_{v}\rangle). 
\end{displaymath}
For each $v\in \fM_{\K}$, there is a commutative diagram
\begin{displaymath}
  \xymatrix{ \T_{v}^{\an} \ar[r]^{\chi^{m}}\ar[d]_{\val_{v}} & 
    \G_{\textrm{m},{v}}^{\an} \ar[d]^{\val_{v}} \\
    N_{\R} \ar[r]_{m}& \R
}
\end{displaymath}
The commutativity of this diagram implies that
$\nu_{\chi^{m}(p_{l}),v} = m_{*} \nu_{p_{l},v}$. By
Theorem~\ref{thm:6}, the sequence $(\bfnu_{p_{l}})_{l\ge 1}$ converges
in the adelic KR-topology to the centered adelic measure
$(\delta_{u_{v}})_{v}$ on $N_{\R}$. Hence, the sequence $(
\bfnu_{\chi^{m}(p_{l})})_{l\ge 1}$ converges in the adelic KR-topology
to the centered adelic measure $(\delta_{\langle m,u_{v}\rangle})_{v}$
on $\R$. By~\eqref{eq:12} and  Lemma~\ref{lem:measure minimum}, the sequence of points
$(\chi^{m}(p_{l}))_{l\ge 1}$ is $\ov D_{0}^{m}$-small.

Summarizing, the sequence $(\chi^{m}(p_{l}))_{l\ge 1}$ of algebraic
points of $\P^{1}_{0}$ is generic and small with respect to the
quasi-canonical toric metrized divisor $\ov D_{0}^{m}$.  Theorem
\ref{thm:4} then implies that the sequence of measures
$(\mu_{\chi^{m}(p_{l}),v})_{l\ge 1}$ on the analytification
$\P_{{v}}^{1,\an}\simeq \P^{1}(\C)$ converges
to~$\lambda_{\SS_{v},\langle m,u_{v}\rangle }$.

Assume now that $v$ is Archimedean. Since the space of probability
measures on $X(\C)$ is sequentially compact, by restricting to a
subsequence we can suppose without loss of generality
that~$(\mu_{p_{l},v})_{l \ge 1}$ converges to a measure~$\mu$.
Since the sequence of direct images
$((\val_{v})_{\ast}\mu_{p_l,v})_{l\ge 1}$
converges in the KR-topology to the Dirac measure on the point
$u_{v}\in N_{\R}$, we deduce that
\begin{displaymath}
 \supp (\mu )\subset \val_{v}^{-1}(u_{v})=\SS_{v}\cdot \e^{-u_{v}}. 
\end{displaymath}
Let $z$ be the standard affine coordinate of $\P^{1}(\C)$.  For each
$m\in M\setminus \{0\}$, let $z_{m}$ be a continuous function on
$\P^{1}(\C)$ that agrees with $z$ on a neighborhood of $S^{1}\cdot
\chi^{m}(\e^{-u_{v}})$. Hence $(\chi^{m})^{\ast}(z_{m})$ agrees with
the character $\chi^{m}$ on a neighborhood of $\SS_{v}\cdot
\e^{-u_{v}}$.  Then
\begin{multline*}
  \int \chi^{m} \dd\mu = 
  \int (\chi^{m})^{\ast}(z_{m}) \dd\mu=
  \lim_{l} \int (\chi^{m})^{\ast}(z_{m})  \dd\mu_{p_{l},v}
\\  =\lim_{l} \int z_{m}\dd (\chi^{m})_{\ast}\mu_{p_{l},v}=
  \lim_{l} \int z_{m}\dd \mu_{\chi^{m}(p_{l}),v}\\
  =\int z_{m}\dd
  \lambda _{S^{1},\langle m,u_{v}\rangle} = \int z\dd
  \lambda _{S^{1},\langle m,u_{v}\rangle} =0,
\end{multline*}
where the last equality comes from Cauchy's formula. 
Hence $ \int \chi^{m} \dd\mu=0$ for all $m\in M\setminus \{0\}$.  By
Fourier analysis, the only probability measure supported on
$\SS_{v}\cdot \e^{-u_{v}}$  satisfying this condition is $\lambda
_{\SS_{v},u_{v}}$. Thus $\mu = \lambda _{\SS_{v},u_{v}}$, concluding
the proof.
\end{proof}

\begin{rem}
  \label{rem:5}
  Our notion of strict sequence is stronger than the one in
  \cite{Bilu:ldspat}. Nevertheless, for the canonical height on a
  projective space, a small sequence of points is strict in our sense
  if and only if it eventually avoids any fixed translate of a
  subtorus with essential minimum equal to $0$. Such a translate of a
  subtorus is necessarily a torsion subvariety, see for instance
  Example \ref{exm:2}. Hence, a small sequence of points that is
  strict in the sense of Bilu \cite{Bilu:ldspat} is also strict in the
  sense of Definition \ref{def:12}. Thus Theorem~\ref{thm:7} applied
  to the canonically metrized divisor at infinity on a projective
  space specializes to \cite[Theorem 1.1]{Bilu:ldspat}.
\end{rem}

\begin{rem}
  \label{rem:10}
  To the best of our knowledge, even for the canonical metric it is
  still not know  if the equidistribution property for strict sequences
  holds for the non-Archimedean places of a global field.
\end{rem}

The toric Bogomolov conjecture can be stated as follows: let $X$ be a
toric variety and $D$ an ample divisor on $X$. Let $V\subset X_{0,\ov
  \K}$ be a subvariety which is not torsion.  Then there exists $\varepsilon >0$ such that the subset
of algebraic points of $V$ of canonical height bounded above by
$\varepsilon$, is not dense in $V$. Equivalently, if $V\subset X_{0,\ov
  \K}$ is a subvariety such that $\upmu^{\ess}_{\ov D^{\can
  }}(V)=0$, then $V$ is a torsion subvariety. 

This conjecture was proved by Zhang in the number field case
\cite{Zhang:plbas}.
Bilu obtained a proof of Zhang's theorem based
on his equidistribution theorem. In
what follows, we extend his approach to the general monocritical case
over a number field.

Recall that  $X$ denotes  a proper toric variety over a number field $\K$ and $\ov D$ a toric metrized
$\R$-divisor on $X$.  For a subset $V\subset X(\ov \K)$, we denote by
$\upmu^{\abs}_{\ov D}(V)$ the absolute minimum of the height of its algebraic
points. The fact that $\ov D$ is toric implies 
\begin{equation}
  \label{eq:78}
\upmu^{\ess}_{\ov D}(X)=\upmu^{\abs}_{\ov D}(X_{0}),
\end{equation}
see \cite[Lemma 3.9(1)]{BurgosPhilipponSombra:smthf}.  Therefore, for
any subvariety $V\subset X_{0,\ov \K}$,
\begin{equation} \label{eq:79}
  \upmu^{\ess}_{\ov D}(V)\ge \upmu^{\abs}_{\ov D}(V)\ge
  \upmu^{\abs}_{\ov D}(X_{0})=\upmu^{\ess}_{\ov D}(X). 
\end{equation}
This motivates the following definition.

\begin{defn}
\label{def:13}
A subvariety  $V\subset  X_{0,\ov \K}$ is \emph{$\ov
  D$-special} if 
\begin{displaymath}
  \upmu^{\ess}_{\ov D}(V)=\upmu^{\ess}_{\ov D}(X).
\end{displaymath}
In particular, an algebraic point $p$ of $X_{0}$ is $\ov D$-special if
and only if $\h_{\ov D}(p)=\upmu^{\ess}_{\ov D}(X)$.
\end{defn}

We also propose the following terminology. 

\begin{defn}\label{def:15} 
  The toric metrized $\R$-divisor $\ov D$ satisfies the
  \emph{Bogomolov property} if every $\ov D$-special subvariety of
  $X_{0,\ov \K}$ is a translate of a subtorus.
\end{defn}

We consider the problem of deciding if a given toric metrized
$\R$-divisor satisfies the Bogomolov property. The following result
corresponds to Theorem~\ref{thm:12} in the introduction, and shows
that the answer  is affirmative for monocritical metrics.

\begin{thm}
  \label{thm:2}
  Let $X$ be a proper toric variety over a number field $\K$ and
  $\overline D$ a monocritical metrized $\R$-divisor on $X$ with
  critical point $\bfu=(u_{v})_{v\in \fM_{\K}}$. Let $V$ be a $\ov
  D$-special subvariety of $ X_{0,\ov \K}$. Then $V$ is a translate of
  a subtorus.

Furthermore, if  $ u_v\in \val_{v}(\T(\K))\otimes \Q$ for all $v$, then
  $V$ is the translate of a subtorus by a $\ov D$-special point.
\end{thm}

Before giving the proof of this theorem, we study special points and, more
generally, special translates of subtori in the monocritical case.  We
first give a criterion for the existence of such points.

\begin{prop}\label{prop:16}
  Let $ X$ be a proper toric variety over $\K$ and $\ov D$ a
  monocritical metrized $\R$-divisor on $X$ with critical point
  $\bfu=(u_v)_{v\in \fM_{\K}}$. Then there exists a $\ov D$-special
  point if and only if
\begin{equation}\label{eq:76}
  u_v\in \val_{v}(\T(\K))\otimes \Q \quad \text{for all} \quad  v\in \fM_{\K}.
\end{equation}
If this is the case, then every $\ov D$-special point is of the form
$q^{{1}/{\ell}}$ with  $q\in X_{0}(\K)$ and $\ell\ge 1$.
\end{prop}

\begin{proof}
  Suppose that there is a $\ov D$-special point $p\in X_{0}(\ov
  \K)$. Choose a finite normal extension $\F\subset \ov \K$ of $\K$ where $p$ is
  defined. 
Consider the norm of $p$ relative to this extension, given by 
\begin{displaymath}
  \norm_{\K}^{\F}(p)=\prod_{\tau \in \Gal(\F/\K)}
  \tau\big(p^{[\F:\K]_{i}}\big) 
\end{displaymath}
where $\Gal(\F/\K)$ and  $[\F:\K]_{i}$ are  the Galois group and
the  inseparable degree of the extension,  respectively. 

Let $v\in \fM_{\K}$. For every $\tau \in \Gal(\F/\K)$, there is a
place $w\in \fM_{\F}$ over $v$ such that $\val_{v}(\tau(p))=
\val_{w}(p)$.  By Corollary~\ref{cor:3} and Proposition~\ref{prop:13},
we have that $\val_{w}(p)=u_v$ for any such place.  It follows that
$\val_{v}(\tau(p))= u_{v}$ for all $\tau$.  Using that
$\#\Gal(\F/\K)\cdot [\F:\K]_{i}=[\F:\K]$, we deduce that
\begin{displaymath}
    \val_{v}( \norm_{\K}^{\F}(p)) = 
\sum_{\tau}
\val_{v}\big(\tau\big(p\big)^{ [\F:\K] _{i}}\big)= [\F:\K] u_{v} .
 \end{displaymath}
 Since $\norm_{\K}^{\F}(p) \in \T(\K)$, we get that $[\F:\K]
 u_v\in \val_{v}(\T(\K))$, proving the implication.

 Conversely, assume that the condition \eqref{eq:76} holds.
Let $S\subset \fM_{\K}$ be a finite set containing the
 Archimedean places and those places~$v$ where $u_v\not = 0$.  Set
 \begin{displaymath}
   \T(\K)_{S}=\{p\in \T(\K)\mid \val_{v}(p)=0 \text{ for all } v\notin
 S\}
 \end{displaymath}
 and let $H_{\K,S}$ be the subspace of $\bigoplus_{v\in S}N_{\R}$
 defined by the equation $\sum_{v\in S}n_{v}z_{v}=0$.
Moreover, consider the lattice
\begin{displaymath}
   \Gamma= H_{\K,S} \cap \bigoplus_{v\in S} \val_{v}(\T(\K)) 
  \end{displaymath}
and the map $ \val_{S}\colon \T(\K)_{S}\to \Gamma$ given by
  \begin{math}
\val_{S}(p)=(\val_{v}(p))_{v\in S}.
  \end{math}
  By Dirichlet's unit theorem \cite[Chapter IV, \S4, Theorem
  9]{Weil:bnt}, the image~$\Lambda$ of this map is a sublattice that is
  commensurable to~$\Gamma$. Thus $\Lambda\otimes\Q =\Gamma\otimes\Q$.
Condition \eqref{eq:76} implies that $(u_v)_{v\in S}\in
  \Gamma\otimes\Q = \Lambda \otimes \Q$.
Hence, there is an integer $\ell\ge 1$ such that
  \begin{displaymath}
     (\ell u_{v})_{v\in S}\in \Lambda.
  \end{displaymath}
In other terms,  there is $q\in \T(\K)_S$ such that $\val_{v}(q)=\ell u_{v}$
  for all $v\in S$. By Corollary~\ref{cor:3}, the point  $p=q^{1/\ell}\in
  \T(\ov \K)$ is $\ov D$-special, proving the reverse implication.

  To prove the last statement, suppose that the condition
  \eqref{eq:76} holds and consider an arbitrary $\ov D$-special point
  $p'\in X_{0}(\ov \K)$. Let $p$ be the $\ov D$-special point
  constructed above and $\F\subset \ov\K$ a finite extension of $\K$
  so that $p,p'\in \T(\F)$. Then $\val_{w}(p'p^{-1})=0 $ for all
  $w\in \fM_{\F}$. By Kronecker's theorem, the point $p'p^{-1}$ is torsion. We conclude that some positive power of $p'$ lies in
  $\T(\K)$, as stated.
\end{proof}

Next we characterize the translates of subtori that are $\ov
D$-special.  Let $U=T_{\ov \K}\cdot p$ be the translate of a subtorus
$T\subset \T$ by a point $p\in X_{0}(\ov \K)$.  The subtorus~$T$
corresponds to a saturated sublattice~$Q$ of~$N$; we denote by $\iota\colon Q\hookrightarrow
N$ the corresponding inclusion map.
Let $\F\subset \ov \K$ be a finite extension of $\K$ where~$p$ is
defined. For each $w\in \fM_{\F}$, we consider the affine
subspace of~$N_{\R}$ given by
\begin{displaymath}
  A_{U,w}= \val_{w}(p)+Q_{\R}.
\end{displaymath}
Indeed $  A_{U,w}=\val_{w}(U^{\an}_{{w}})$ and so this affine subspace
depends only on $U$ and not on a particular choice for the translating point $p$. 

As explained in \cite[\S 3.2]{BurgosPhilipponSombra:agtvmmh}, 
the normalization of the closure of $U$ in $X_{\ov \K}$ can be given a
structure of toric variety. 
Let $\Sigma $ be the  fan on $N_{\R}$ corresponding to~$X$ and
$\Sigma_{Q}$ the fan on $Q_{\R}$ obtained by  restricting
$\Sigma$ to this latter linear space. Then the inclusion $\iota\colon
Q_{\R}\hookrightarrow N_{\R}$ induces an equivariant map of toric
varieties
\begin{equation*}
  \varphi_{p,\iota}\colon
X_{\Sigma _{Q},\ov \K}\to X_{\ov \K}
\end{equation*}
extending the inclusion $U\hookrightarrow \T_{\K}$.

\begin{prop}\label{prop:11}
  Let $X$ be a proper toric variety over a number field $\K$
  and~$\overline D$ a monocritical metrized $\R$-divisor on $X$ with
  critical point $\bfu=(u_{v})_{v\in \fM_{\K}}$.  Let $U=T_{\ov
    \K}\cdot p\subset X_{0,\ov \K}$ be the translate of a subtorus
  $T\subset \T$ by a point $p\in X_{0}(\ov \K)$ defined over a finite
  extension $\F\subset \ov \K$ of~$\K$.
For a place~$w$ in~$\fM_{\F}$ denote by~$v(w)$ the place in~$\fM_{\K}$
below~$w$.
Then we have the following properties.
\begin{enumerate}
  \item \label{item:15} The translate $U$ is $\ov D$-special if and only if $u_{v(w)}\in A_{U,w}$ for all~$w\in \fM_{\F}$.
  \item \label{item:17} If the translate $U$ is $\ov D$-special, then the
    metrized $\R$-divisor $\varphi^{\ast}_{p,\iota}\ov D$ is
    monocritical and its critical point is~$(u_{v(w)} -
    \val_w(p))_{w \in \fM_{\F}}$.
  \end{enumerate}
\end{prop}

\begin{proof}
  By passing to a suitable large finite extension of $\K$ and applying
  Proposition~\ref{prop:13}, we can reduce to the case when $U$ is the
  translate of a $\K$-rational point, that is, $U=T_{\ov \K}\cdot p$
  with $p\in X_{0}(\K)$. With this assumption, $\F=\K$ and we set
  $v\coloneqq w=v(w)$. 
 
  Since $\ov D$ is a semipositive toric metrized divisor with $D$ big,
  the virtual support function $\Psi _{D}$ is concave and its
  associated polytope has dimension $n$. Hence, there is $m\in M_{\R}$
  such that $\langle m, u\rangle >\Psi_{D} (u)$ for all $u\not =
  0$. Moreover, the metric functions~$\psi _{\ov D,v}$ are concave for
  all $v\in \fM_{\K}$.

  Consider the toric metrized $\R$-divisor $\ov E\coloneqq
  \varphi_{\iota,p}^{\ast}\ov D$ on the toric variety
  $X_{\Sigma_{Q}}$. By \cite[Proposition
  4.3.19]{BurgosPhilipponSombra:agtvmmh}, its virtual support function
  and metric functions are given, for $z\in Q_{\R}$, by
  \begin{displaymath}
    \Psi _{ E}(z)=\Psi_{D} (\iota(z)),\qquad
    \psi _{\ov E,v}(z)
=
\psi_{\ov D,v}    (\val_{v}(p)+ \iota(z)).
  \end{displaymath}
  Therefore $\Psi _{E}$ is concave and satisfies $\langle \iota^{\vee}
  m, z\rangle >\Psi_{E}(z)$ for all $z\in Q_{\R}\setminus
  \{0\}$. Hence, the $\R$-divisor $E$ is big.  Moreover, the metric
  functions $\psi _{\ov E,v}$ are concave and so $\ov E$ is
  semipositive.

  Since $U$ is identified with a dense open subset of $X_{\Sigma
    _{Q},\ov \K}$, we have 
  \begin{displaymath}
  \upmu^{\ess}_{\ov
    D}(U)=\upmu^{\ess}_{\ov E}(X_{\Sigma _{Q}}).  
  \end{displaymath}
Consider the affine
  subspace $\bfA_{U}=\bigoplus_{v}A_{U,v}$ of $\bigoplus_{v}N_{\R}$.
  By  Corollary~\ref{cor:4},
  \begin{multline*}
    \upmu^{\ess}_{\ov E}(X_{\Sigma
    _{Q}})=\min_{\bfu'\in H_{\F}\cap \bfA_{U}}
  \sum_{v}-n_{v}\psi_{\ov D, v}(u_{v}'),\qquad
  \upmu^{\ess}_{\ov D}(X)=\min_{{\bfu'\in H_{\F}}}
  \sum_{v}-n_{v}
\psi_{\ov D,v}(u_{v}').
  \end{multline*}
  Since $\ov D$ monocritical, the minimum in the right equality is
  attained only at~$\bfu' = ( u_{v} )_{v}$.
We conclude
  that
  $\upmu^{\ess}_{\ov E}(U)=\upmu^{\ess}_{\ov D}(X)$ if and only if $u_{v}\in
  A_{U,v}$ for all~$v\in \fM_{\K}$, proving both statements.
\end{proof}

\begin{cor}
  \label{cor:6}
  Let $X$ be a proper toric variety over a number field $\K$ and
  $\overline D$ a monocritical metrized $\R$-divisor on $X$ with
  critical point $\bfu=(u_{v})_{v\in \fM_{\K}}$, and suppose that $
  u_v\in \val_{v}(\T(\K))\otimes \Q$ for all $v\in \fM_{\K}$.  Then a
  translate of a subtorus of $X_{0}$ is $\ov D$-special if and only if
  it is the translate of a subtorus by a $\ov D$-special point.
\end{cor}

\begin{proof}
Clearly, the translate of a subtorus by a $\ov D$-special point
is~$\ov D$-special.
To prove the reverse implication, let $U$ be a $\ov D$-special
translate of a subtorus and write
  $U=T_{\ov \K}\cdot p$ as in the statement of
  Proposition~\ref{prop:11}. By this result, the toric metrized
  $\R$-divisor $\ov E=\varphi_{p,\iota}^{*}\ov D$ is monocritical and, for each $v\in \fM_{\K}$ and $w\in \fM_{\F}$
  over $v$, 
  \begin{displaymath}
    u_{v}\in   A_{U,w}\cap \val_{v}(\T(\K))\otimes \Q \subset  A_{U,w}
    \cap\val_{w}(\T(\F))\otimes \Q .
  \end{displaymath}
  Since $p\in X_{0}(\F)$,
  \begin{displaymath}
   A_{U,w} \cap
  \val_{w}(\T(\F))\otimes \Q
=
\val_{w}(p) + \val_{w}(T(\F)) \otimes \Q . 
  \end{displaymath}
 Hence $u_{v} - \val_{w}(p) \in \val_{w}(T(\F)) \otimes \Q.$
Extending the base field to $\F$ and restricting to $X_{\Sigma_{Q}}$,
Proposition~\ref{prop:16} implies that this toric variety
contains an $\ov E$-special point. Hence $U$ contains a $\ov D$-special
point and it is the translate of $T$ by this 
point, as stated.
\end{proof}

\begin{exmpl}
  \label{exm:2}
  Let $\ov D^{\can}$ be a nef and big toric $\R$-divisor on the
  proper toric
  variety~$X$, equipped with the canonical metric.  By 
  Example~\ref{exm:9}, it is monocritical with critical point
  $\bfzero\in H_{\K}$.
%
Hence, $p\in X_{0}(\ov \K)$ is $\ov D^{\can}$-special if and only if
$\val_{v}(p)=0$ for every $v\in \fM_{\K}$.
By Kronecker's theorem, this is also equivalent to the fact that $p$ is
   torsion. Hence, Corollary~\ref{cor:6} shows that a translate of a subtorus
  that is $\ov D^{\can}$-special is necessarily the translate of a
  subtorus by a torsion point, that is, a torsion subvariety. 
\end{exmpl}

\begin{proof}[Proof of Theorem~\ref{thm:2}]
    Let $U\subset X_{0,\ov \K}$ be the minimal translate of a subtorus
  containing the subvariety $V$ and let~$Q$ and~$\Sigma_{Q}$ be
  defined before Proposition~\ref{prop:11}.
By \eqref{eq:78} and \eqref{eq:79}, we
  have $\upmu^{\abs}_{\ov D}(U)=  \upmu^{\ess}_{\ov D}(U) $ and
  \begin{displaymath}
\upmu^{\ess}_{\ov D}(X)
=
\upmu^{\abs}_{\ov D}(X_0)
\le  \upmu^{\abs}_{\ov D}(U)\le \upmu^{\abs}_{\ov D}(V)\le
  \upmu^{\ess}_{\ov D}(V) = \upmu^{\ess}_{\ov D}(X).   
  \end{displaymath}
  Therefore, $U$ is $\ov D$-special.  By Proposition
  \ref{prop:11}\eqref{item:17}, 
  $\ov D$ pulls back to a monocritical metrized $\R$-divisor on
  $X_{\Sigma_{Q}}$, the normalization of the closure of $U$ in~$X_{\ov
    \K}$. Replacing~$X$ by this toric variety, we reduce to the case
  where $U=X_{0,\ov \K}$.

  Using Proposition~\ref{prop:4}, we choose a sequence $(p_{l})_{l\ge
    1}$ of algebraic points of $V$, that is generic in $V$ and satisfies
\begin{displaymath}
\lim_{l}\h_{\ov
  D}(p_{l})= \upmu^{\ess}_{\ov D}(V). 
\end{displaymath}
Since $V$ is not contained in any proper translate of a subtorus, this
sequence is strict and, since $V$ is $\ov D$-special, it is also $\ov
D$-small.

Applying Theorem~\ref{thm:7} to an Archimedean place $v\in \fM_{\K}$, we
obtain that the sequence of measures $(\mu_{p_{l},v})_{l\ge 1}$
converges to a measure 
whose support is the translate $\SS_{v}\cdot \e^{-u_{v}}$ of the
compact subtorus, with $u_{v}$ the $v$-adic coordinate of the critical
point of $\ov D$.  Hence $ \SS_{v}\cdot \e^{-u_{v}}\subset
V^{\an}_{{v}}$. Since $\SS_{v}\cdot \e^{-u_{v}}$ is dense in
$X_{v}^{\an}$ with respect to the Zariski topology, it follows that
$V=X_{0, \ov \K}$, proving the result.
\end{proof}

By Theorem~\ref{thm:2} and Example~\ref{exm:2}, the canonical toric
metrized $\R$-divisor~$\ov D^{\can}$  satisfies the Bogomolov property, and every $\ov
D^{\can}$-special subvariety is  torsion. Hence, Theorem
\ref{thm:2} extends Zhang's theorem to the general monocritical
case. 
On the other hand, in \S\ref{sec:count-bogom-prop} we will give examples of
non-monocritical metrized divisors  not satisfying the Bogomolov property.

\section[Examples]{Examples} \label{sec:examples}

The obtained criteria can be applied in concrete situations, to decide
if a given semipositive toric metrized $\R$-divisor satisfies
properties like modulus concentration or equidistribution.  In this
section, we consider translates of subtori with the canonical height,
and  toric metrized $\R$-divisors equipped with positive smooth
metrics at the Archimedean places and canonical metrics at the
non-Archimedean ones. We also give a family of
counterexamples to the Bogomolov property in the non-monocritical
case.

\subsection{Translates of subtori with the canonical
  height} \label{sec:transl-subt-with}

Let $X$ a proper toric variety of dimension $n$ over a global field
$\K$ and $D$ a big and nef toric $\R$-divisor on~$X$.
Let~$\Psi _{D}$ be its virtual support function.

We denote by $\ov D^{\can}$ this $\R$-divisor equipped with the
canonical metric as in Example \ref{exm:6}.  This toric metrized
$\R$-divisor satisfies that, for all $v\in \fM_{\K}$,
 \begin{displaymath}
\psi_{\ov
  D^{\can},v}=\Psi_{D} \and \vartheta_{\ov D^{\can},v}=0.
\end{displaymath}
Since $D$ is big, $\Delta _{D}$ has dimension $n$. Every point $x$ in the
interior of $\Delta _{D}$ maximizes the global roof function and 
$\partial \vartheta _{\ov D^{\can},v}({x})=\{0\}$. Therefore, for all
$v\in \fM_{\K}$, 
\begin{displaymath}
  B_{v}=\{0\}\and F_{v}=\{0\}.
\end{displaymath}
By Proposition \ref{prop:14}, the canonical metric is monocritical and
so, by Theorem \ref{thm:5}, $\ov D^{\can}$ satisfies the
equidistribution property at every place (Example \ref{exm:4}).

We next study the toric metrics on $D$ that are obtained as inverse
image by an equivariant map 
of a canonical metrized toric divisor on an projective space.

Let $v\in\fM_{\K}$.  If $v$ is Archimedean, we set $\lambda _{v}=1$
whereas, if $v$ is non-Archimedean, we set $\lambda _{v}$ as the
positive generator of the discrete subgroup $\val_{v}(\K^{\times})$ of
$ \R$. A piecewise affine function is said to be \emph{$\lambda
  _{v}$-rational} if all its defining affine functions $\langle
x,u\rangle +b$ satisfy $x\in M_{\Q}$ and $b\in \lambda _{v}\Q$.

Let $\P^{r}_{\K}$ be a standard projective space over $\K$ with
homogeneous coordinates $(z_0:\dots:z_r)$ and~$H$ the hyperplane at
infinity, defined by the equation $z_0=0$. Denote by~$\ov H^{\can}$
this toric divisor equipped with the canonical metric.  As seen in
\cite[Example 3.7.11]{BurgosPhilipponSombra:agtvmmh}, if $\psi \colon
N_{\R}\to \R$ is a concave $\lambda_{v}$-rational piecewise affine
function with $|\psi-\Psi_{D}|$ is bounded, then there is an integer
$r>0$ and a toric morphism $\iota \colon X \to \P^{r}_{\K}$ such that
\begin{displaymath}
  \psi =\psi _{\iota ^{\ast} \ov H^{\can},v}.
\end{displaymath}
Hence, any such function $\psi $  can be realized as the $v$-adic
metric function of a toric metrized divisor on $D$. This allows us to
construct many examples, both  monocritical and non-monocritical, of metrized
toric divisors.

In the next examples, we fix $\K=\Q$ and, as before, we denote by
$\ov H^{\can}$ the hyperplane at infinity with the canonical metric.

\begin{exmpl}\label{exm:5}
  Let $\iota\colon \G_{\textrm{m},\Q}\to \P^{2}_{\Q}$ be the map given
  by 
  \begin{displaymath}
    \iota(t)=(1:t/2:t).
  \end{displaymath}
  Let $X$ the normalization of the closure of
  $\iota(\G_{\textrm{m},\Q})$ and $\ov D=\iota^{\ast}(\ov
  H^{\can})$. Then $X=\P^{1}_{\Q}$ and $D$ is the divisor at infinity.

  We have $\Delta _{D}=[0,1]$.  As explained in \cite[Example
  5.1.16]{BurgosPhilipponSombra:agtvmmh}, for each $v\in \fM_{\Q}$ the
  graph of the local roof function associated to $\ov D$ is given by
  the upper envelope of the extended polytope
  \begin{displaymath}
    \conv((0,0),(1,\log|1/2|_{v}),(1,\log|1|_{v})) \subset
    \R\times \R.
  \end{displaymath}
  The graphs of these functions are represented in
  Figure~\ref{fig:localrf1}.
\captionsetup[subfigure]{labelformat=empty}
\begin{figure}[ht]
  \centering
  \begin{subfigure}[1]{0.3\textwidth}
    \input{ex_71_21.pdf_t}
    \caption{$v=2$}
  \end{subfigure}
  \begin{subfigure}[2]{0.3\textwidth}
    \input{ex_71_infty1.pdf_t}    
    \caption{$v=\infty$}
  \end{subfigure}
  \begin{subfigure}[1]{0.3\textwidth}
    \input{ex_71_other1.pdf_t}
    \caption{$v\ne\infty, 2$}
  \end{subfigure}
\caption{Local roof functions in Example \ref{exm:5}}
\label{fig:localrf1}
\end{figure}
Thus, for $x\in [0,1]$ we have $ \vartheta _{2}(x)=x\log(2)$ and $
\vartheta _{v}(x)=0$ for $v\not=2$. The global roof function is
$\vartheta (x)=x\log(2)$ and the only point that maximizes it is
$x=1$. Moreover, $    \partial \vartheta _{2}({1}) =(-\infty,\log(2)]$
and $ \partial \vartheta_{v} ({1})    =(-\infty,0]$ for $v\not=2$.
With Notation \ref{def:14}, we have 
\begin{alignat*}{2}
  B_{2}&= [0,\log(2)],&\quad
  F_{2}&= [-\infty,\log(2)],\\
  B_{v}&=[-\log(2),0],& F_{v}&=[-\infty,0]\ \text{ for }v\not = 2.
\end{alignat*}
By Corollary \ref{cor:3 bis}, this metrized divisor does not satisfy
the modulus concentration property at any place. \emph{A fortiori}, it
does not satisfy the equidistribution property at any place. 

Indeed, by \eqref{eq:10} we have $    \upmu^{\ess}_{\ov
  D}(X)=\log(2)$.
  Let $(\omega _{l})_{l\ge 1}$ be a sequence given by a choice of a
  primitive $l$-th root of the unity, $a\not =2$ a positive prime number and
  $r$ an integer with $\log(a)\le r\log(2)$. Choose any $r$-th
  root of $a$ and consider the generic sequences of points 
  \begin{displaymath}
    p_{l}=(1:\omega _{l})\and q_{l}=(1:2a^{-1/r}\omega_{l}) \quad
    \text{ for } l\ge 1.
  \end{displaymath}
For every $v\in \fM_{\Q}$, $l\ge 1$, $p\in
  \Gal(p_{l})_{v}$ and $q\in
  \Gal(q_{l})_{v}$  we have $    (\val_{v})_{\ast}(p)=0$ and 
  \begin{displaymath}
 (\val_{v})_{\ast}(q)=
 \begin{cases}
   \log(2)&\text{ if }v=2,\\
   \frac{-1}{r}\log(a)&\text{ if }v=a,\\
   -\log(2)+\frac{1}{r}\log(a)&\text{ if }v=\infty,\\
   0&\text{ if }v\not =2,a, \infty.\\
 \end{cases}
  \end{displaymath}
Either by computing the local roof functions of $\ov D$ or the Weil
height of the image of these points under the inclusion $\iota$, we
deduce  that
\begin{displaymath}
\h_{\ov D}(p_{l})=\log(2) \and \h_{\ov D}(q_{l})=\log(2).  
\end{displaymath}
Therefore both sequences are $\ov D$-small. For any place $v$, the
sequence $\mu _{p_{l},v}$ converges to $\lambda _{\SS_{v},0}$. By
contrast, if we denote $u_v=(\val_v)_\ast(q)$ for any $q\in
\Gal(q_{l})_{v}$ , then $\mu_{q_{l},v}$ converges to $\lambda
_{\SS_{v},u_v}$. This shows that neither the modulus concentration nor
the equidistribution properties hold for the places $2, a,
\infty$. Varying $a$, we deduce that these properties do not hold at
any place of $\Q$.

The metric of $\ov D$ at the Archimedean place is the canonical one.
The metrics at the non-Archimedean places can be interpreted in terms
of integral models. Let~$\cX$ be the blow up of $\P^{1}_{\Z}$ at the
point $(1:0)$ over the prime $2$. The fibre of the structural map
$\cX\to \Spec (\Z)$ over the point $2$ has two components: the
exceptional divisor of the blow up, that we denote by $E$, and the
strict transform of the fibre of $\P^{1}_{\Z}$, that we denote $Y$.
Consider the divisor
  \begin{displaymath}
 \cD=\ov \infty + Y,   
  \end{displaymath}
  where $\ov \infty$ denotes the 
  closure in $\cX$ of the point $(0:1)\in \P^{1}(\Q)$. The pair
  $(\cX,\cD)$ 
  is a model of $(X,D)$. For each non-Archimedean place $v$, this model
  induces an algebraic metric on $D$ that agrees with the $v$-adic
  metric of $\ov D$.   
\end{exmpl}

\begin{exmpl} \label{exm:8} Consider now the map
    $\iota\colon \G_{\textrm{m},\Q}\to \P^{2}_{\Q}$ given
  by
  \begin{displaymath}
   \iota(t)=(t^{-1}:1/2:t). 
  \end{displaymath}
Let $X$ be the  normalization of the closure
  of $\iota(\G_{\textrm{m},\Q})$ and $\ov D=\iota^{\ast}(\ov
  H^{\can})$. In this case, $X=\P^{1}_{\Q}$ and $D$ is the divisor at
  infinity plus the divisor at zero. 

  We have $\Delta _{D}=[-1,1]$. As before, we compute the local roof
  functions using \cite[Example
  5.1.16]{BurgosPhilipponSombra:agtvmmh}. Their graphs are  represented in
  Figure \ref{fig:localrf2}.
\captionsetup[subfigure]{labelformat=empty}
\begin{figure}[ht]
  \centering
  \begin{subfigure}[1]{0.3\textwidth}
    \input{ex_73_21.pdf_t}
    \caption{$v=2$}
  \end{subfigure}
  \begin{subfigure}[2]{0.3\textwidth}
    \input{ex_73_infty1.pdf_t}    
    \caption{$v=\infty$}
  \end{subfigure}
  \begin{subfigure}[1]{0.3\textwidth}
    \input{ex_73_other1.pdf_t}
    \caption{$v\ne\infty, 2$}
  \end{subfigure}
\caption{Local roof functions in Example \ref{exm:8}}
\label{fig:localrf2}
\end{figure}
For $x\in[0,2]$, we have $\vartheta _{2}(x)=(1 - |x|)\log(2)$ and $\vartheta _{v}(x)=0$ for $v\not = 2$.
  Thus, the global roof function is~$\vartheta(x)= (1 - |x|)\log(2)$.
  Its maximum is attained only at the point $x = 0$. In this case,
  $    \partial \vartheta _{2}(0)=[-\log(2),\log(2)]$ and $   \partial
  \vartheta _{v}(0)=\{0\}$ for $v\not = 2$. We deduce that
  \begin{equation}\label{eq:3}
    B_{2}=\{0\},\ F_{2}=[-\log(2),\log(2)]\and
    B_{v}=\{0\},\ F_{v}=\{0\} \text{ for }v\not = 2.
  \end{equation}
  By Corollary \ref{cor:3 bis}, $\ov D$ satisfies modulus
  concentration for all places except the place~$2$. This toric
  metrized divisor is not monocritical, and so we cannot apply
  Theorem~\ref{thm:5} in this case. Indeed, later we will see that
  $\ov D$ does not satisfy the equidistribution property at any other
  place of $\Q$ (Example~\ref{exm:3}).

  As in the previous example, the metric of $\ov D$ at the Archimedean
  place is the canonical one, and those at the non-Archimedean places
  can be interpreted in terms of integral models. Let $\cX$ be the
  blow up of $\P^{1}_{\Z}$ at the points $(1:0)$ and $(0:1)$ over the
  prime $2$. The fibre of the structural map $\cX\to \Spec (\Z)$ over
  the point $2$ has three components.
  Consider the divisor
  \begin{displaymath}
   \cD=\ov \infty + \ov 0 ,
  \end{displaymath}
 where $\ov \infty$ denotes the 
  closure in $\cX$ of the point $(0:1)\in \P^{1}(\Q)$ and $\ov 0$ the 
  closure of the point $(1:0)$. The pair
  $(\cX,\cD)$ 
  is a model of $(X,D)$. For each non-Archimedean place $v$, this model
  induces an algebraic metric on $D$ that agrees with the $v$-adic
  metric of $\ov D$. 
\end{exmpl}

\begin{exmpl}\label{exm:7} This time we consider the map $\iota\colon
  \G_{\textrm{m},\Q}\to \P^{3}_{\Q}$ given 
  by
  \begin{displaymath}
   \iota(t)=(1:t/2:t^2/2:t^3). 
  \end{displaymath}
Let $X$ be the  normalization of
  the closure 
  of $\iota(\G_{\textrm{m},\Q})$ and $\ov D=\iota^{\ast}(\ov
  H^{\can})$. In this case, $X=\P^{1}_{\Q}$ and $D$ is three times the divisor at
  infinity. 

We have  $\Delta _{D}=[0,3]$ and the local roof functions are
  represented in Figure \ref{fig:localrf3}. 
\captionsetup[subfigure]{labelformat=empty}
\begin{figure}[ht]
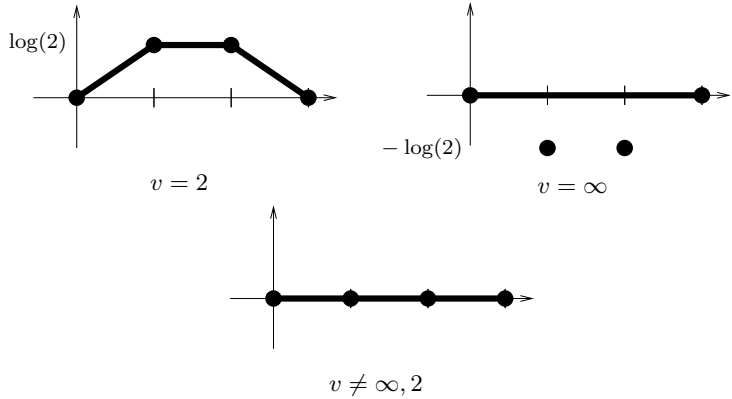

  \centering
  \begin{subfigure}[1]{0.40\textwidth}
    \input{ex_74_21.pdf_t}
    \caption{$v=2$}
  \end{subfigure}
  \begin{subfigure}[2]{0.40\textwidth}
    \input{ex_74_infty1.pdf_t}    
    \caption{$v=\infty$}
  \end{subfigure}
  \begin{subfigure}[1]{0.40\textwidth}
    \input{ex_74_other1.pdf_t}
    \caption{$v\ne\infty, 2$}
  \end{subfigure}
 \caption{Local roof functions in Example \ref{exm:8}}
 \label{fig:localrf3}
\end{figure}
They are given by $ \vartheta _{2}(x)= \log(2) \min (x,1,
3-x)$ and $ \vartheta _{v}(x)=0 $ for $v\not = 2$.
The global roof function is thus $\vartheta(x)= \log(2) \min (x,1,3-x)$, which is maximized at any point of the interval
$[1,2]$. Choosing the maximizing point $x=3/2$, we have $\partial
\vartheta _{v}({3/2})=\{0\}$ for all $v$.  

Thus $\ov D$ is monocritical, by Proposition~\ref{prop:14}.
By Corollary \ref{cor:3 bis} and Theorem~\ref{thm:5}, it satisfies both the modulus concentration and the
equidistribution properties for any place.  By Theorem \ref{thm:2}, it
also satisfies the Bogomolov property.
\end{exmpl}

\subsection{Positive Archimedean metrics} \label{sec:posit-arch-metr} 

The following result covers many of the examples 
considered in \cite{BurgosPhilipponSombra:agtvmmh,
  BurgosMoriwakiPhilipponSombra:aptv, BurgosPhilipponSombra:smthf}:
twisted Fubiny-Study metrics on projective spaces, metrics from
polytopes, Fubini-Study metrics on toric bundles, $\ell^{p}$-metrics
on toric varieties, and Fubini-Study metrics on weighted projective
spaces.  All of them consist of toric varieties over $\Q$ with a toric
divisor equipped with a positive smooth metric at the Archimedean
place and the canonical metric at the non-Archimedean ones.

\begin{thm}\label{thm:8} Let $X$ a proper
  toric variety over a number field $\K$ and $\ov D=(D,(\|\cdot\|_{v})_{v\in
    \fM_{\K}})$ a semipositive toric metrized $\R$-divisor with $D$
  big. 
We assume that, when  $v$ is Archimedean,   $\|\cdot\|_{v}$ is a
positive smooth metric on the principal open subset $X^{\an}_{0,{v}}$
whereas, when $v$ is non-Archimedean, it is the $v$-adic canonical metric of
$D$. Then $\ov D$ is monocritical. In particular, it satisfies the equidistribution
property for every place of $\K$.

When $\K=\Q$, the $v$-adic limit measure is $\lambda_{\SS_{v},0}$ for every
$v\in \fM_{\Q}$. 
\end{thm}

\begin{proof}
  Since the metric is smooth and positive for $v$ Archimedean, the
  proof of \cite[Proposition 4.4.1]{BurgosPhilipponSombra:agtvmmh}
  implies that the metric function $\psi _{\ov D,v}$ is smooth and
  strictly concave, in the sense that its Hessian is negative
  definite.
Therefore $\psi _{\ov D,v}$ is of Legendre type in the
  sense of \cite[Definition 2.4.1]{BurgosPhilipponSombra:agtvmmh} and,
  by \cite[Theorem 2.4.2(2)]{BurgosPhilipponSombra:agtvmmh}, the local
  roof function $\vartheta _{\ov D,v}$ is of Legendre type.  In
  particular,  $\vartheta _{\ov D,v}$ is smooth and strictly concave  on the interior of
  $\Delta_{D}$ and the sup-differential at any point of the border
  of the polytope is empty. 

  For the non-Archimedean places, the metrics are canonical and so
  their local roof functions are zero.  Hence
\begin{displaymath}
  \vartheta_{\ov D}=\sum_{v\mid\infty}n_{v}\vartheta_{\ov D,v},
\end{displaymath}
this function is smooth and strictly concave on the
interior of~$\Delta_D$, and its sup-differential at any point of
the border of~$\Delta_D$ is empty. This implies that there is a unique maximizing point
$x_{\max}\in \Delta_{D}$, that it lies in the interior of the
polytope, and that $\partial
\vartheta_{\ov D}(x_{\max})=\{0\}$.  

The first assertion then follows from Proposition \ref{prop:14}. The
rest of the statement follows from Theorem \ref{thm:5}. 
\end{proof}

\begin{exmpl}\label{exm:11} Let $X=\P^{1}_{\Q}$ and $\ov D$ the
  divisor at infinity equipped with the Fubini-Study metric at the
  Archimedean place and the canonical metric at the non-Archimedean
  ones. By Theorem \ref{thm:8}, this toric metrized divisor satisfies
  the equidistribution property at every place. Moreover, the limit
  measure of the Galois orbits of any generic $\ov D$-small sequence
  is $\lambda_{\SS_{v},0}$.

  Recall that the canonical metric at the non-Archimedean places
  corresponds to the canonical model of $(\P^{1}_{\Q},\infty)$ given
  by $(\P^{1}_{\Z},\ov \infty)$, where $\ov \infty$ is the closure of
  the point $(0:1)\in \P^{1}(\Q)$. If we change the integral model,
  different phenomena may occur. For instance, if we consider the
  integral model of Example \ref{exm:5}, then the maximum of the
  global roof function is attained at the interior of the
  polytope. Since the global roof function is differentiable in the
  interior of the polytope, we deduce that the sup-differential is
  reduced to one point. By Proposition \ref{prop:14}, this new toric
  metrized divisor is also monocritical.

  By contrast, if we consider the divisor $D'=0+\infty$ with the
  Fubini-Study metric at the Archimedean place and the metrics induced
  by the integral model of Example~\ref{exm:8}, then the maximum of
  the global roof function is attained at the point zero and the
  sup-differential at this point is $[-\log(2),\log(2)]$. Since zero is
  not a vertex of this set, by Proposition \ref{prop:14} this divisor
  is not monocritical. Hence it does not satisfy the equidistribution
  property at the Archimedean place.
\end{exmpl}

\subsection{Counterexamples to the Bogomolov
  property} \label{sec:count-bogom-prop}

In this section, we give examples of toric metrized
divisors not satisfying the Bogomolov property. For simplicity,
we restrict to the case $\K=\Q$. As in~\S\ref{sec:transl-subt-with},
we denote by $\ov H^{\can}$  the canonical metrized divisor at
infinity on a projective space.

\begin{exmpl}\label{exm:10}
  Consider the map $\iota\colon
  \G_{\textrm{m},\Q}\times \G_{\textrm{m},\Q} \to \P^{3}_{\Q}$ given by
  \begin{displaymath}
    \iota(t_{1},t_{2})=(1:2:t_{1}:t_2).
  \end{displaymath}
  As in the examples in the previous section, we denote by~$X$ the
  normalization of the closure of the image of $\iota$ and $\ov
  D=\iota^{\ast}(\ov H^{\can})$. In this case, $X=\P^{2}_{\Q}$ and $D$
  is the divisor at infinity.

We have that  $\Delta _{D}$ is the standard simplex of $N_{\R}=\R^{2}$
and $\Psi_{D}\colon \R^{2}\to \R$ is the function given by 
  \begin{displaymath}
    \Psi _{D}(u_{1},u_{2})=\min(0,u_1,u_2).
  \end{displaymath}
  By
  \cite[Example 4.3.21]{BurgosPhilipponSombra:agtvmmh}, the local metric
  functions are given, for $(u_{1},u_{2})\in \R^{2}$, by
  \begin{displaymath}
    \psi _{\ov D,v}(u_1,u_2)=
    \begin{cases}
      \Psi_{D} \big(u_1+\log(2),u_2+\log(2)\big)-\log(2)&\text{ if }v=\infty,\\
      \Psi_{D} (u_1,u_2)&\text{ if }v\not =\infty.
    \end{cases}
  \end{displaymath}
  By \cite[Example 5.1.16]{BurgosPhilipponSombra:agtvmmh}, the local
  roof functions are given, for $(x_1,x_2)\in \Delta _{D}$, by
\begin{displaymath}
  \vartheta _{\ov D, v}(x_1,x_2)=
  \begin{cases}
    (1-x_1-x_2)\log(2)&\text{ if }v=\infty,\\
    0&\text{ if }v\not =\infty.
  \end{cases}
\end{displaymath}
Hence the global roof function agrees with $\vartheta _{\ov
  D,\infty}$. Its only maximizing point is $x_{\max}=(0,0)$.  Hence
$ \partial \vartheta _{\ov D,\infty}{(0,0)}
=(-\log(2),-\log(2))+\R^2_{\ge 0}$ and $ \partial \vartheta _{\ov
  D,v}{(0,0)} = \R^2_{\ge 0}$ for $v\not = \infty$. Thus
\begin{alignat*}{2}
  B_{\infty}&=[-\log(2),0]^2,&\quad
  F_{\infty}&=(-\log(2),-\log(2))+\R^2_{\ge 0},\\
  B_{v}&=[0,\log(2)]^2,& F_{v}&=\R^2_{\ge 0}\ \text{ for }v\not = \infty.
\end{alignat*}
We also have $ \upmu^{\ess}_{\ov D}(X)=\vartheta(0,0)=\log(2). $

Let $(z_0:z_1:z_2)$ be homogeneous coordinates of $X$ and consider the
curve $C$ of equation $z_0+z_1+z_2=0$. In what follows,  we will see that this
curve is a $\ov D$-special subvariety.  Since $C$ is not a translate
of a subtorus, this will show that $\ov D$ does not satisfy the
Bogomolov property.

For $l\ge 1$ choose a primitive $l$-th root of the unity $\omega
_{l}$. Let $z_{1,l}$ be a solution of the equation $z^{2}+z+\omega
_{l}=0$ and put $z_{2,l}=\omega _{l}/z_{1,l}$ for the other
solution. Then 
\begin{equation}\label{eq:34}
  z_{1,l}+z_{2,l}+1=0 \and z_{1,l}z_{2,l}=\omega _{l}. 
\end{equation}
In particular,  $p_{l}=(1:z_{1,l}:z_{2,l})$ is an algebraic point of
$C$. 

Let $v\in
\fM_{\Q}$ and $q=(1:q_{1}:q_{2})\in \Gal(p_{l})_{v}$. 
If $v\ne \infty$, then the conditions~\eqref{eq:34} imply that
\begin{equation}\label{eq:68}
  \val_{v}(q)=(0,0)\in B_{v}.
\end{equation}
If $v=\infty$,  then these same
conditions~\eqref{eq:34} give $ \max(|q_{1}|_{\infty},|q_{2}|_{\infty})\le
\frac{1+\sqrt{5}}{2}$. Thus
\begin{equation}
  \label{eq:69}
  \val_{\infty}(q)\in
  \Big(-\log\Big(\frac{1+\sqrt{5}}{2}\Big),-\log\Big(\frac{1+\sqrt{5}}{2}\Big)\Big)+ 
  \R^{2}_{\ge 0}\subset F_{\infty}.
\end{equation}
Moreover, by the product formula and \eqref{eq:68}, we have
\begin{equation}
  \label{eq:70}
\exv[\nu_{p_{l},\infty}]= \frac{1}{\# \Gal(p_{l})_{\infty}} \sum_{q\in  \Gal(p_{l})_{\infty}}\val_{\infty}(q)=(0,0)\in B_{\infty}. 
\end{equation}
By Corollary \ref{cor:3}, the conditions \eqref{eq:68}, \eqref{eq:69}
and \eqref{eq:70} imply that $ \h_{\ov D}(p_{l})=\upmu^{\ess}_{\ov
  D}(X).$ Since the sequence $(p_{l})_{l\ge 1}$ is generic in $C$,
we deduce $\upmu^{\ess}_{\ov D}(C)=\upmu^{\ess}_{\ov D}(X)$ and so $C$
is a $\ov D$-special subvariety.
\end{exmpl}

We generalize this example to a family of metrics on toric varieties
of dimension greater than or equal to $2$.

\begin{prop} \label{prop:2}
Let~$X$ be a proper toric variety over~$\Q$ of dimension $n\ge 2$
and~$D$ a big and nef $\R$-divisor on~$X$. 
Let $u_0\in N_{\R}$ and consider the metrized divisor~$\overline{D}^{u_0}$
over~$D$ defined by
$$ \psi_{\overline{D}^{u_0}, v}(u)
=
\begin{cases}
\Psi_D(u - u_0)
& \text{if } v = \infty,
\\
\Psi_D(u)
& \text{if } v \neq \infty.
\end{cases} $$
Then~$\overline{D}^{u_0}$ satisfies the Bogomolov property if and only
if~$u_0 = 0$.
\end{prop}

\begin{proof}
  When~$u_0 = 0$ we have $\overline{D}^{u_0}=\ov D^{\can}$. By Theorem
  \ref{thm:2} and Example \ref{exm:2}, this toric metrized divisor
  satisfies the Bogomolov property.

Suppose~$u_0 \neq 0$. The local roof functions of $\ov D^{u_{0}}$ are
given, for $x\in \Delta _{D}$, by
\begin{displaymath}
  \vartheta _{\ov D^{u_{0}},v}(x)=
  \begin{cases}
    \langle x,u_{0}\rangle& \ \text{ if }v=\infty,\\
    0& \ \text{ if }v\neq \infty.
  \end{cases}
\end{displaymath}
In particular, the global roof function~$\vartheta_{\ov D}$ coincides with~$\vartheta_{\ov D^{u_{0}},
  \infty}$.
Fix~$x_0$ in the relative interior of the convex subset of~$\Delta_D$ on which~$\vartheta_{\ov D}$ attains it
maximum value.
If we denote by~$\vartheta_0$ the constant function
equal to~$0$ defined on~$\Delta_D$, then~$\sigma_0 = \partial
\vartheta_0 (x_0)$ is a cone in~$N_{\R}$ containing~$-u_0$ in its
relative interior.
Moreover,
\begin{displaymath}
  \partial \vartheta _{\ov  D^{u_{0}},\infty}({x_0}) = u_0 +
  \sigma_0\and
  \partial \vartheta _{\ov D^{u_{0}},v}({x_0}) = \sigma_0 \text{ for }v\not = \infty.
\end{displaymath}
It follows that $0\in B_{v}$ for every $v$,
that $F_{\infty} = u_{0}+\sigma _{0}$ and that $F_{v} = \sigma
_{0}$ for $v\neq \infty.$

As in Example \ref{exm:10}, to prove that $\ov D^{u_0}$ does not
satisfy the Bogomolov property, it is enough to exhibit a curve~$C$
in~$X$ which is $\ov D$-special but not a translate of a subtorus.


We identify~$N_{\R}\simeq \R^n$.  Since~$X$ is proper and~$\sigma_0$
is a cone of the fan of $X$, there is a primitive vector $n_0 \in N$
in~$\sigma_0$.
It follows that there
is~$\varepsilon_0 > 0$ such that
$$ 
\ell_0
:=
\{ x n_0 \mid - \varepsilon_0 \le x \le \varepsilon_0 \}
\subset
u_0+\sigma_0.
$$ 
Choose a primitive vector~$a_0\in N$
such that $a_0$ and $n_0$ generate a saturated sublattice $V$ of
$N$. Put $b_0= n_0 + a_0$. Then $a_0$ and~$b_0$ form an integral basis of
$V$. Fix an integer $ k_0 \ge \varepsilon_0^{-1} $ and consider the
linear map~$L \colon V_{\R} \to \R^2$ defined by
$$ 
L(s a_0 + t b_0)= k_{0} \cdot (s, t). 
$$ 
Let $S$ be the toric surface in~$X_0$ associated to the saturated
sublattice $V$. The linear map $L$ induces a toric morphism $\iota
\colon S \to \G_{\textrm{m},\Q}^2$.  Let~$C$ be the curve
in~$\G_{\textrm{m},\Q}^2$ of equation $x + y + 1=0$ and denote by~$C_0$ the closure in~$X$ of the curve~$\iota^{-1}(C)$.

As in Example \ref{exm:10}, for $l\ge 1$ choose a primitive $l$-th
root of unity root $\omega _{l}$.
Let~$z_{1,l}$ be a solution of the equation $z^{2}+z+\omega _{l}=0$ and put $z_{2,l}=\omega
_{l}/z_{1,l}$. Hence
\begin{equation*}
  z_{1,l}+z_{2,l}+1=0 \and z_{1,l}z_{2,l}=\omega _{l}. 
\end{equation*}
In particular, $(z_{1,l},z_{2,l})\in C(\ov \Q)$.
Choose a point $p_{l}\in C_0(\ov \Q)$ such that $\iota(p_l)=(z_{1,l},z_{2,l})$. The sequence of points $(p_{l})_{l\ge
  0}$ is generic in $C_0$.

For every place $v$ there is a commutative diagram
\begin{displaymath}
  \xymatrix{
(\G_{\textrm{m}}^{2})^{\an}_{{v}}\ar[d]^{\val_{v}} & 
S^{\an}_{{v}} \ar[d]^{\val_{v}} \ar@{^{(}->}[r] \ar[l]_{\iota}&
X^{\an}_{0,{v}}\ar[d]^{\val_{v}} \\
\R^{2} & V_{\R} \ar@{^{(}->}[r] \ar[l]_{L}& N_{\R}
}
\end{displaymath}
Since $n_{0}=b_{0}-a_{0}$, we have
\begin{displaymath}
  \ell := L_{\R}(\ell_0) = \{(x,-x)\mid |x|\le \varepsilon _{0}k_{0}\}.
\end{displaymath}
Arguing as in Example \ref{exm:10}, for every non-Archimedean place
$v$ and every point $q\in \Gal(p_{l})_{v}$, we have
\begin{displaymath}
  \val_{v}(\iota (q))=0.
\end{displaymath}
Since $L$ is injective, $\val_{v}(q) = 0$ and therefore~$\nu_{p_l, v}
= \delta_0$.
In particular,
$$ \supp(\nu_{p_l, v}) = \{ 0 \} \subset F_v
\text{ and }
\exv [ \nu_{p_l, v}] = 0 \in B_v. $$
When $v=\infty$, the product formula implies that 
\begin{displaymath}
\exv[\nu_{p_{l},\infty}]= \frac{1}{\# \Gal(p_{l})_{\infty}} \sum_{q\in
  \Gal(p_{l})_{\infty}}\val_{\infty}(q)
=
0 \in B_{\infty}. 
\end{displaymath}
On the other hand, the facts that $|z_{1,l}|_{\infty}\, |z_{1,l}|_{\infty}=1$ and that   
$$ \frac{\sqrt{5}-1}{2}
\le \min \{ |z_{1,l}|_{\infty}, |z_{2,l}|_{\infty} \} \le \max \{
|z_{1,l}|_{\infty}, |z_{2,l}|_{\infty} \} \le \frac{1 +
  \sqrt{5}}{2}, $$ imply that $\val_{\infty}(\iota (q)) \in \ell$ for
every $q\in \Gal(p_{l})_{\infty}$. Thus
\begin{displaymath}
  \val_{\infty}(q)\in \ell_0 \subset u_0 + \sigma_0 = F_{\infty}.
\end{displaymath}
This implies that~$\supp(\nu_{p_l, \infty}) \subset F_{\infty}$.
By Lemma \ref{lem:measure minimum}, we have $\h_{\ov D}(p_{l})=\upmu_{\ov
  D}^{\ess}(X)$. Being the sequence $(p_{l})_{l\ge 1}$ generic in~$C_0$, we deduce  that $C_0$ is $\ov D$-special. Since $C_0$ is not a translate of a
subtorus, we conclude that $\ov D$ does not satisfy the Bogomolov
property, as stated. 
\end{proof}

\section{Potential theory on the projective line and small points} \label{sec:potent-theory-proj}

In this section, we apply potential theory on the projective line over
a number field, and in particular Rumely's Fekete-Szeg\H{o} theorem,
to produce interesting sequences of small points in the non-monocritical case.

In the absence of modulus concentration, this allows to produce a
wealth of non-toric measures that are limit measures of Galois orbits
of generic sequences of points of small height. These techniques also
allow to show that the absence of modulus concentration at a place can
affect the equidistribution property at another place.

\subsection{Limit measures in the absence of modulus concentration}\label{sec:limit-meas-absence-1}

We recall the basic objects of potential theory on the projective
line. For most of the details and precise definitions, we refer the
reader to \cite{Tsuji:ptmft} and \cite{BakerRumely:ptdbpl} for the
Archimedean and non-Archimedean cases, respectively.

Let~$\K$ be a number field and fix a place $v\in \fM_{\K}$.  
For a subset $E\subset \C_{v}$, we denote by $\ov E$ its closure in
$\A^{1,\an}_{{v}}$. Moreover, for~$r > 0$, put
\begin{displaymath}
 \ball_v(E, r) = \Big\{ z \in \C_v \, \Big|\  \inf_{a\in E}|z -
a|_v \le r \Big\}.
\end{displaymath}
In particular, for $a\in \C_{v}$ the set $\ball_{v}(a,r)$ is the closed
ball with center $a$ and radius~$r$. We set~$O_v = \ball_v(0, 1)$.

We denote by
\begin{displaymath}
 \delta_{v} \colon \A^{1,\an}_{{v}} \times
\A^{1,\an}_{{v}}\to \R
\end{displaymath}
the function defined by $ \delta_{v}(z,z')=|z-z'|_{v}$ for~$v$
Archimedean, and as the unique upper semicontinuous extension
of the function on~$\C_v \times \C_v$ defined by $(z, z') \mapsto |z-z'|_{v}$ for~$v$ non-Archimedean,
see \cite[Proposition 4.1]{BakerRumely:ptdbpl}.  

Given  a probability measure  $\mu$ on $\A^{1,\an}_{v}$, the \emph{energy
  integral} (with respect to the point at infinity) of $\mu$ is defined as
\begin{equation}\label{eq:18}
  I_{v}(\mu)=\int_{\A^{1,\an}_{v}\times \A^{1,\an}_{v}}
  -\log (\delta_{v}(z,z'))\dd \mu(z) \dd\mu(z').
\end{equation}
Let $K\subset \A^{1,\an}_{v}$ be a measurable subset. The
\emph{$v$-adic Robin constant} and \emph{capacity} (with respect to
the point at infinity) of $K$ are respectively defined as
\begin{equation}\label{eq:37}
  V_{v}(K)=\inf\{ I_{v}(\mu) \mid \supp(\mu)\subset K\} \quad
  \text{ and } \quad \capacity _{v}(K)= \e^{-V_{v}(K)}.
\end{equation}

In the non-Archimedean case, $\C_{v}$ is a proper subset of
$\A^{1,\an}_{v}$. In general, 
\begin{displaymath}
 \capacity_{v}(K\cap \C_{v})\le
\capacity_{v}(K), 
\end{displaymath}
but there is a particular case where the equality holds. If $K$ is a
closed strict affinoid domain, then $K\cap \C_{v}$ is an RL-domain in
the sense of \cite{Rumel:ctac.14012}. Moreover, $\ov{K\cap \C_{v}}=K$
and, by \cite[Corollary~6.26]{BakerRumely:ptdbpl} and \cite[Theorem
4.3.3]{Rumel:ctac.14012}, we have 
\begin{equation*}
 \capacity_{v}(K\cap \C_{v})=  \capacity_{v}(K).
\end{equation*}
For instance,  $O_v$ is a closed strict affinoid domain and 
\begin{equation}
  \label{eq:22}
  \capacity_{v}(O_{v})=\capacity_{v}(\overline{O}_v)=1.
\end{equation}
In the Archimedean case, $O_{v}=\overline{O}_v$ and these equalities are
also valid.

If $K$ is compact and $\capacity_{v}(K)>0$, then there exists a unique
probability measure, denoted by~$\rho _{K}$, supported on $K$
and realizing the infimum in \eqref{eq:37}, see \cite[\S III.2 and Theorem
III.32]{Tsuji:ptmft} for the Archimedean case and \cite[Propositions
6.6 and~7.21]{BakerRumely:ptdbpl} for the non-Archimedean one.
Hence
\begin{equation*}
  I_{v}(\rho _{K})=V_{v}(K).
\end{equation*}
This measure is called the \emph{equilibrium measure} of $K$. 
For~$K=\overline{O}_v$, it agrees with~$\lambda_{\SS_v,0}$, the Haar
probability measure on the unit circle when $v$ is Archimedean, and
the Dirac measure at the Gauss point of $\A^{1,\an}_{v}$ when $v$ is
non-Archimedean.

\begin{defn}
  \label{def:7}
  An \emph{adelic set} is a collection $\bfE=(E_{v})_{v\in \fM_{\K}}$
  such that $E_{v}$ is a subset of $ \C_{v}$ invariant under the
  action of the absolute $v$-adic Galois group $\Gal(\ov
  \K_{v}/\K_{v})$ for all $v$, and such that $E_{v}=O_{v}$ for all but
  a finite number of $v$.  We say that ${\bfE}$ is \emph{bounded}
  (respectively \emph{closed}, \emph{open}) if $E_{v}$ is bounded
  (respectively closed, open) for all $v$.
\end{defn}

Given an adelic set $\bfE=(E_{v})_{v\in \fM_{\K}}$, its \emph{(global)
  capacity} is defined as
\begin{displaymath}
  \capacity(\bfE)=\prod_{v\in \fM_{\K}}\capacity_{v}(E_{v})^{n_{v}}.
\end{displaymath}
By \eqref{eq:22}, this product actually runs over a finite set and so
the global capacity is well-defined. 

The following result shows that, in the non-monocritical case, there is
a wealth of measures not invariant under the action of the compact
torus, that can be obtained as limit measures of Galois orbits of
generic sequences of points of small height.

\begin{thm}
  \label{thm:13}
  Let $X=\P^{1}_{\K}$ and $\ov D$ the divisor at infinity
  equipped with a semipositive toric metric.  Let $B_{v},F_{v}$ be the
  associated subsets of $N_{\R}=\R$ as in~\eqref{eq:9}.  Let
  $\bfE=(E_{v})_{v\in\fM_{\K}}$ be a closed bounded adelic set with
  $E_v$ an RL-domain for every non-Archimedean place $v$, and such
  that $\capacity(\bfE)=1$. Assume that the following conditions hold:
  \begin{enumerate}
  \item \label{item:2} $\supp((\val_{v})_{*}\rho_{E_v})\subset F_{v}$  for all $v\in \fM_{\K}$;
  \item \label{item:6} $\exv[(\val_{v})_{*}\rho_{E_v}]\in B_{v}$  for all $v\in \fM_{\K}$;
\item \label{item:1} $\sum_{v\in
    \fM_{\K}}n_{v}\exv[(\val_{v})_{*}\rho_{E_v}]=0$. 
  \end{enumerate}
  Then there is a generic $\ov D$-small sequence $(p_{l})_{l\ge1}$
  of algebraic points of $X_{0}=\G_{{\rm m},\K}$ such that, for every $v\in
  \fM_{\K}$, the sequence of probability measures
  $(\mu_{p_{l},v})_{l\ge 1}$
  converges  to $\rho_{E_{v}}$.
\end{thm}

The proof of this theorem will be given after two preliminary
propositions. 
The next statement is a direct consequence of Rumely's version of the
Fekete-Szeg\H{o} theorem in~\cite[Theorem~2.1]{Rumely:fstsc}.

\begin{prop} \label{prop:7} Let $\bfE=(E_{v})_{v\in\fM_{\K}}$ be a
  closed bounded adelic set such that  $\capacity(\bfE)\ge 1$. There
  exists a generic sequence $(p_{l})_{l\ge1}$ of points of $\ov
  \K^{\times}$ satisfying
  \begin{equation*}
  \Gal(p_l)_v\subset \ball_{v}\Big(E_v,\frac{1}{l}\Big)
  \end{equation*}
  for all~$l\ge1$ and $v\in \fM_{\K}$. In particular,
  $\Gal(p_l)_v\subset E_v$ for every non-Archimedean place $v$ such
  that $E_{v}=O_{v}$.
 \end{prop}

\begin{proof}
For $l\ge 1$,  consider the bounded adelic neighbourhood
$\bfU_{l}=(U_{l,v})_{v\in \fM_{\K}}$ of~$\bfE$ given by 
\begin{equation*}
    U_{l,v}= \ball_{v}\Big(E_v,\frac{1}{l}\Big).
\end{equation*}
By \cite[Theorem 2.1]{Rumely:fstsc}, there is an infinite number of
points $p\in \ov \K^{\times}$ such that
$\Gal(p)_{v}\subset U_{l,v}$ for all $v$. Inductively, for each
$l\ge1$ we choose $p_{l}$ as one of these points which is different
from $p_{l'} $ for $l'\le l-1$.
\end{proof}

In the notation of Proposition \ref{prop:7}, when the adelic set
$\bfE$ has capacity 1, the sequence of $v$-adic Galois orbits of the
points $p_{l}$ equidistribute according to the equilibrium measure of
the closure $\ov E_{v}$.

\begin{prop}
  \label{prop:integers equidistribution}
  Let $\bfE=(E_{v})_{v\in\fM_{\K}}$ be a closed bounded adelic set
  with $E_v$ an RL-domain for every non-Archimedean place $v$ and such
  that $\capacity(\bfE)=1$. Let $(p_{l})_{l\ge1}$ be a generic
  sequence of points of $\ov \K^{\times}$ with $ \Gal(p_l)_v\subset
  \ball_{v}(E_v,\frac{1}{l})$ for all~$l\ge1$ and~$v\in
  \fM_{\K}$. Then, for all~$v\in\fM_{\K}$, the sequence~$(\mu_{p_l,
    v})_{l \ge 1}$ converges to the equilibrium measure of~$\ov
  {E}_v$.
\end{prop}

\begin{proof}
  By taking a subsequence, we can suppose without loss of generality
  that, for each~$v$ in~$\fM_{\K}$, the sequence~$(\mu_{p_l, v})_{l
    \ge 1}$ converges to a probability measure~$\mu_v$.


  For each~$l \ge 1$ and~$v\in \fM_{\K}$, put for short
  $G_{l,v}=\Gal(p_{l})_{v}$ and set
$$ 
d_{l, v}
= \frac{1}{\# G_{l,v} (\# G_{l,v} - 1)} \sum_{\substack{ q,  q' \in
  G_{l,v} \\  q \neq  q'}} \log | q -
 q'|_v .
$$
Consider also the probability measure on $\P^{1,\an}_{{v}}\times
\P^{1,\an}_{{v}}$ given by
$$ \nu_{l, v}
=\frac{1}{\# G_{l,v} (\# G_{l,v} - 1)} \sum_{\substack{ q, q' \in
    G_{l,v} \\ q \neq q'}} \delta_{ q} \times \delta_{q'}, $$ and note
that~$(\nu_{l, v})_{l \ge 1}$ converges to~$\mu_v \times \mu_v$.  The
function~$\log (\delta_v(\cdot, \cdot))$ is 
bounded from above on~${\ball_v(E_{v}, 1)} \times {\ball_v(E_{v},
  1)}$. Similarly as in the proof of Lemma~\ref{lemm:3}, this property
implies that
\begin{multline}
  \label{eq:32}
\limsup_{l \to \infty}  d_{l, v}
=
\limsup_{l \to \infty} \int_{\P^{1, \an}_{v} \times \P^{1, \an}_{v}}
\log (\delta_{v}(z, z')) \dd \nu_{l, v}(z, z')
\\ \le
- I_v(\mu_v)
\le
\log \capacity_v({E}_v).
\end{multline}

By the product formula, $\sum_{v \in \fM_{\K}} n_vd_{l, v} = 0$.  Let
$S\subset \fM_{\K}$ be a finite set of places containing the
Archimedean places and those where $E_{v}\ne O_{v}$. In particular,
$d_{l,v}\le 0$ for $v\notin S$.  Hence, for $v\in \fM_{\K}$,
\begin{displaymath}
  \begin{split}
    \liminf_{l \to \infty} d_{l, v}
& =
\liminf_{l \to \infty} \sum_{w \in \fM_{\K} \setminus \{ v \}} -\frac{n_w}{n_v}d_{l,
  w}
\\ & \ge 
\liminf_{l \to \infty} \sum_{w \in S \setminus \{ v \}} -\frac{n_w}{n_v}d_{l,
  w}
\\ & \ge 
-\sum_{w \in S \setminus \{ v \}} \frac{n_w}{n_v}\limsup_{l \to \infty} d_{l,
  w} 
\\ & \ge 
-\sum_{w \in  S\setminus \{ v \}} \frac{n_w}{n_v}\log(\capacity_w({E}_w))
\\ & \ge \log(\capacity_v({E}_v)).
  \end{split}
\end{displaymath}
Together with~\eqref{eq:32}, this implies~$I_v(\mu_v) = - \log
\capacity_v(E_v)$.  It follows that~$\mu_v$ is the
equilibrium measure of~${E}_v$,  completing the proof.
\end{proof}

\begin{proof}[Proof of Theorem \ref{thm:13}]
  Let $(p_{l})_{l\ge1}$ be a generic sequence of points of $\ov
  \K^{\times}$ as in Proposition \ref{prop:integers equidistribution},
  that exists thanks to Proposition \ref{prop:7}.
  Then Proposition \ref{prop:integers equidistribution} implies that,
  for every $v\in \fM_{\K}$, the sequence 
  of probability measures $(\mu_{p_{l},v})_{l\ge 1}$ converges
  to $\rho_{E_{v}}$.  Here we have to show that, under the present
  hypotheses, this sequence of points is $\ov D$-small.

  Let $s_{D}$ be the canonical section of $\mathcal{O}(D)$ with
  $\div(s_D)=D$. This is a global section vanishing only at
  infinity. Hence its $v$-adic Green function 
$$
g_{\ov  D,v}=-\log\|s_{D}\|_{v}
$$ 
is a continuous real-valued function on $
  \A^{1,\an}_{v}$. Let $S\subset \fM_{\K}$ be a finite set of places
  containing the Archimedean places, the places where the metric
  $\|\cdot\|_{v}$ differs from the canonical one, and those where
  $E_{v}\ne O_{v}$.

  By construction, $\Gal(p)_{v}\subset \ball_{v}(E_{v},1)$ for all
  $v$. In particular, for all $v\notin S$ we have $\Gal(p)_{v}\subset
  O_{v}$ and so $g_{\ov D,v}(q)=0$ for all $q\in \Gal(p)_{v}$. Hence
  \begin{displaymath}
    \h_{\ov D}(p_{l})=\sum _{v\in \mathfrak{M}_{\K}}\frac{n_{v}}{\#
      \Gal (p_{l})_{v}} \sum_{q\in \Gal(p_{l})_{v}}g_{\ov D,v}(q)=
  \sum_{v\in S}n_v\int  \wt g_{\ov D,v} \dd \mu _{p_{l},v}
  \end{displaymath}
  for any continuous function~$\wt g_{\ov D,v}$ on~$\P^{1,\an}_{{v}}$
  coinciding with $g_{\ov D,v}$ on the bounded subset $
  \ball_{v}(E_{v},1)$.

  The measures~$\mu _{p_{l},v}$ converge to~$\rho _{E_v}$ and are
  supported on the closure $\ov{ \ball_{v}(E_{v},1)}$.  Also, for all
  $v\notin S$, we have $\rho_{E_{v}}=\lambda_{\SS_{v},0}$ and $g_{\ov
    D,v}$ vanishes on the support of this measure. Hence
\begin{equation} \label{eq:49}
  \lim_{l\to \infty} \h_{\ov D}(p_{l})= \sum_{v\in S}n_v\int \wt
  g_{\ov D,v} \dd \rho _{E_v}
=
\sum_{v\in \mathfrak{M}_{\K}}n_v\int
  g_{\ov D,v} \dd \rho _{E_v}.
\end{equation}

By the condition~\eqref{item:1} and the fact that ${\bfE}$ is an adelic
set, we deduce that the collection  $\bfnu=((\val_{v})_{\ast}\rho
_{E_v})_{v\in \fM_{\K}}$ is a centered adelic measure (Definition~\ref{def:6}). 
Moreover,  $g_{\ov D,v}=-\psi _{\ov D, v}\circ \val_{v}$ on $\A^{1,\an}_{v}\setminus
  \{0\}$. 
By \eqref{eq:49} and the fact that the equilibrium measure does not
charge any individual point, we have
\begin{displaymath}
    \lim_{l\to \infty} \h_{\ov D}(p_{l})=-\sum_{v\in \mathfrak{M}_{\K}}n_v\int
  \psi _{\ov D,v} \dd (\val_{v})_{\ast}\rho _{E_v} =  \eta_{\ov
    D}(\bfnu). 
\end{displaymath}
 Lemma \ref{lem:measure minimum} together with  the conditions~\eqref{item:2}
and~\eqref{item:6} implies that $\eta_{\ov D}(\bfnu)=\upmu^{\ess}_{\ov
  D}(X)$. Hence the sequence~$(p_l)_{l \ge 1}$ is $\ov
D$-small, as stated, finishing the proof of the theorem.
\end{proof}

\subsection{Local modulus concentration and  
equidistribution}\label{sec:limit-meas-absence}


Corollary \ref{cor:3 bis} gives a criterion for a semipositive toric
metrized $\R$-divisor to satisfy the modulus concentration property at
a given place. Applying it, one can immediately give examples where
modulus concentration fails at that place. If this happens, then the
equidistribution property also fails at that place.

Can this absence of modulus concentration affect the
equidistribution property at another  place? The next result on the
projective line over a number field shows that this can be the case
under a rationality hypothesis. 

\begin{prop}\label{prop:8}
  Let $X=\P^{1}_{\K}$ be the projective line over a number field $\K$,
  $\ov D$ the divisor at infinity equipped with a semipositive toric metric, and
  $v_{0}\in \fM_{\K}$. For each $v\in \fM_{\K}$, let $B_v$ be the set introduced
  in Notation \ref{def:14}. Assume that there is a point $p\in
  X_{0}(\ov \K)=\ov \K^{\times}$ such that $\val_{v}(p)\in B_{v}$ for
  all $v\in \fM_{\K}$ and $\val_{v_{0}}(p)\in \ri(B_{v_{0}})$. 

  If $\ov D$ does not satisfy the modulus
  concentration property at~$v_{0}$, then $\ov D$ does not satisfy the
  equidistribution property at any place of $\K$.
\end{prop}

\begin{proof} 
  Assume that $\ov D$ does not satisfy the modulus concentration
  property at $v_{0}$. Let $v\in \fM_{\K}$. If $v=v_0$ then clearly
  $\ov D$ does not satisfy the equidistribution property at $v$, so we
  can suppose that $v\ne v_{0}$.  Extending scalars to a suitable
  large number field and translating by the point $p$, we can also
  reduce to the case when $0\in \ri(B_{v_{0}})$ and $0\in B_{w}$ for
  all $w\in \fM_{\K}$.

Let $F_{v_0}$ and $A_{v_{0}}$ be the convex sets given in Notation
\ref{def:14}. 
By Corollary \ref{cor:3 bis}, the set $F_{v_{0}}$ is not a single
point.  
Since $0\in \ri(B_{v_{0}})$ and 
  $F_{v_{0}}$ is the minimal face of $A_{v_{0}}$ containing $B_{v_0}$,
  there is $\delta>0$ such  
  that the set $F_{v_{0}}$ contains the interval~$[-\delta, \delta]$. Set
  $$
  c=\frac{\e^\delta+\e^{-\delta}}{2}>1
  $$
  and consider the closed bounded adelic set~$\bfE = ( E_w)_{w \in \fM_{\K}}$ given by
\begin{align*}
E_{v_{0}} =&
\begin{cases} 
[-2c,2c] &\text{if~$v_{0}$ is Archimedean}, \\[1mm]
\ball_{v_{0}}(2,c) &\text{if $v_{0}$ is non-Archimedean},
\end{cases} \\
E_{v} =&
\begin{cases} \displaystyle
[-2/c,2/c] & \text{if~$v$ is Archimedean}, \\[1mm]
\displaystyle
\ball_v(2,1/c)& \text{if $v$ is non-Archimedean},
\end{cases} 
\end{align*}
and, for $w\not=v_0,v$,
$$
E_{w} =
\begin{cases} 
[-2,2] & \text{if~$w$ is Archimedean}, \\[1mm]
O_w=\ball_{w}(0,1) & \text{if $w$ is non-Archimedean}.
\end{cases}
$$
The local capacities of these sets are
$$ 
\capacity_{v_{0}}(E_{v_{0}}) = c, \quad 
\capacity_{v}(E_{v}) = 1/c \and \capacity_{w}(E_{w})=1 \quad \text{for
}w\ne v_{0},v,
$$
see for instance \cite[\S 3]{Rumely:fstsc}. Hence,  the global capacity of~$\bfE$ is~$1$.

Consider the map $R\colon \P^{1}_{\K}\to \P^{1}_{\K}$ given by
$R(z)=z+\frac{1}{z}$. Using the expression~$R(z)-2=\frac{(z-1)^2}{z}$,
one checks that, for $w$ non-Archimedean,
$$R^{-1}(E_w) \subset 
\begin{cases}
\{z\in\C_{v_0}\mid|z-1|_{v_0}^2\le c|z|_{v_0}\} & \text{ if } w=v_0,\\
\{z\in\C_v\mid|z-1|_v^2\le c^{-1}|z|_v\}& \text{ if } w=v,\\
\{z\in\C_w\mid|z^2+1|_w\le|z|_w\}& \text{ if } w\not=v_0,v.
\end{cases}
$$
Using that $z=\frac{1}{2}(R(z)\pm\sqrt{R(z)^2-4})$, one also
checks that, for  $w$ Archimedean, 
\begin{displaymath}
R^{-1}(E_w)\subset 
\begin{cases}
\{z\in\C_{v_0}\mid c-\sqrt{c^2-1}\le |z|_{v_0} \le c+\sqrt{c^2-1} \} & \text{ if } w=v_0,\\
\{z\in\C_w\mid |z|_{w} = 1 \} & \text{ if } w\not=v_0.\\
\end{cases}
\end{displaymath}
Note that  $c-\sqrt{c^2-1}=\e^{-\delta}$ and
$c+\sqrt{c^2-1}=\e^\delta$. 
We
deduce from the previous analysis that, regardless  whether $v_{0}$
or $w$ are Archimedean or not, we have 
$$
R^{-1}(E_{v_0})\subset \val_{v_0}^{-1}([-\delta,\delta]) \and
R^{-1}(E_w)\subset\val_w^{-1}(0) \text{ for } w\not=v_0.
$$

We represent in Figure \ref{fig:Rinv} the inverse images by $R$ of the
sets $E_{v_0}$, $E_{v}$ and $E_{w}$ in the Archimedean case.  The
point $x$ therein is~$x= c^{-1} + i \sqrt{1 - c^{-2}}$.

\captionsetup[subfigure]{labelformat=empty}
\begin{figure}[ht]
  \centering
  \begin{subfigure}[1]{0.40\textwidth}
    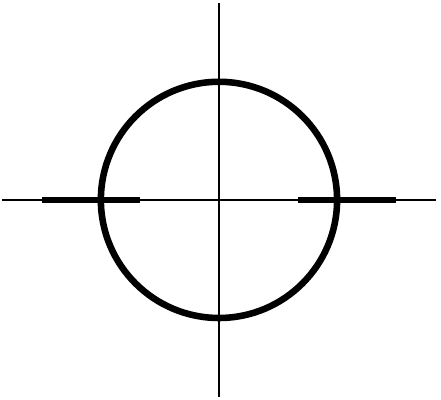
    \caption{$R^{-1}(E_{v_0})$}
  \end{subfigure}
   \begin{subfigure}[1]{0.40\textwidth}
     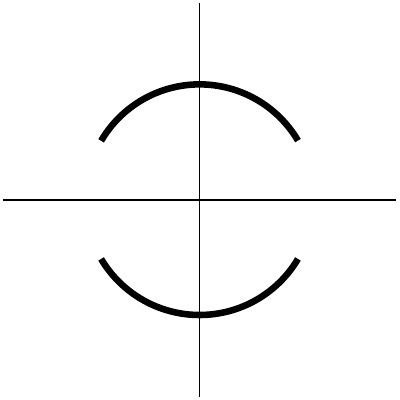    
     \caption{$R^{-1}(E_{v})$}
   \end{subfigure}
   \begin{subfigure}[1]{0.40\textwidth}
     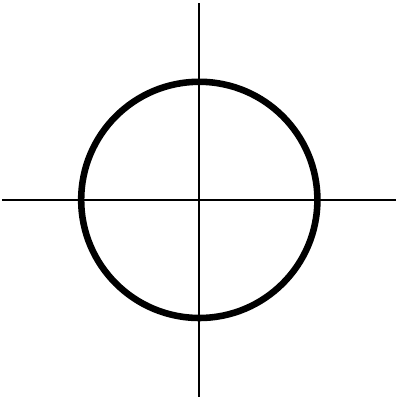
     \caption{$R^{-1}(E_{w})$}
   \end{subfigure}
 \caption{Inverse images by $R$ of the sets $E_{v_0}$, $E_{v}$ and $E_{w}$ for
  $v_{0}$, $v$ and  $w\ne v,v_0$  Archimedean}
 \label{fig:Rinv}
\end{figure}

Let~$(p_l)_{l \ge 1}$ be a generic sequence as given by
Proposition~\ref{prop:7} applied to the adelic set $\bfE$.  For each
$l\ge1$, choose a point $q_{l}\in R^{-1}(p_{l})$. After restricting to
a subsequence, we can assume that the sequence $(\mu_{q_{l},w})_{l\ge
  1}$ converges to a probability measure $\mu_{w} $ on $\P^{1,\an}_{v}$, for all
$w\in \fM_{\K}$.  By construction, for each $w$ the supports of the
direct image measures $\nu_{q_{l},w}=(\val_{w})_{*} \mu_{q_{l},w}$,
$l\ge 1$, are contained in a fixed bounded subset of
$N_{\R}=\R$. Therefore, this sequence of measures converges in the
KR-topology to the direct image~$(\val_{w})_{\ast}\mu_{w} $.

Let $S\subset \fM_{\K}$ be the finite subset consisting of the
Archimedean places plus $v_0$ and $v$. If $w \not = v_0$, then
$\Gal(q_{l})_{w}\subset \val_{w}^{-1}(0)$ and $\exv[\nu_{q_{l},w}]=0$.
Thus
\begin{displaymath}
\exv[(\val_{w})_{\ast}(\mu_{w} )] =\lim_{l}\exv[\nu_{q_{l},w}]=0 .  
\end{displaymath}
Hence, thanks to the convergence in the KR-topology and the product
formula,
\begin{displaymath}
  \exv[(\val_{v_0})_{\ast}(\mu_{v_0} )]=\lim_{l}
  \exv[\nu_{q_{l},v_0}]= \lim_{l}\sum_{\substack{w\in
      S\\w\not = v_0}}-\exv[\nu_{q_{l},v_0}]=0.
\end{displaymath}
Thus $\exv[(\val_{w})_{\ast}(\mu_{w} )]=0\in B_{w}$ for all $w\in
\fM_{\K}$.  By construction, it is also clear that
$\supp((\val_{w})_{\ast}\mu _{w})\subset F_{w}$ for all $w$. By Lemma
\ref{lem:measure minimum}, the sequence $(q_l)_{l\ge 1}$ is $\ov
D$-small.

We have thus constructed a generic $\ov D$-small sequence such that
its $v$-adic Galois orbit converges to a measure $\mu_{v}$ whose
support is contained in the closure $\ov{R^{-1}(E_{v})}$.  Observe
that
\begin{displaymath}
  \ov{R^{-1}(E_{v})}\subsetneq \SS_{v}=\val^{-1}(0)
\end{displaymath}
because, in
the Archimedean case, it does not contain the points $1$ and $-1$,
whereas in the non-Archimedean case, it does not contain the Gauss
point. 

On the other hand, the sequence $(\omega _{l})_{l\ge 1}$ given by the
choice of a primitive $l$-th root of unity is also $\ov D$-small, but
its $v$-adic Galois orbit converges to the measure~$ \lambda_{\SS,0}$. The support of this measure is strictly bigger than that of
$\mu_{v} $.  We deduce that $\ov D$ does not satisfy the $v$-adic
equidistribution property, as stated.
\end{proof}

\begin{exmpl}
  \label{exm:3}
  Let $X=\P^{1}_{\Q}$ and $\ov D$ the divisor at infinity plus the
  divisor at zero, equipped with the semipositive toric metric from
  Example \ref{exm:8}.  As explained therein, $\ov D$ does not satisfy
  modulus concentration at the place $v_{0}=2$ and, by \eqref{eq:3},
  we have $0\in \ri(B_{v})$ for all $v\in \fM_{\Q}$.
  Theorem~\ref{thm:13} implies that $\ov D$ does not satisfy the
  equidistribution property for any place of $\Q$.
\end{exmpl}

\begin{rem} \label{rem:3} A rationality hypothesis like the condition
  that the sets $B_{v}$ contain the image by the valuations map of an algebraic
  point, is necessary for the conclusion of Proposition \ref{prop:8}
  to hold.  Indeed, suppose that, for a given non-Archimedean place
  $v$, we have $B_{v}=F_{v}=\{u_{v}\}$ with $u_v\not \in \val_{v}(\ov
  \K_{v}^{\times})$.  By the tree structure of the Berkovich
  projective line, this implies that $\val_{v}^{-1}(u_v)$ consists of
  single point, of type III in Berkovich's classification \cite [\S
  1.4]{BakerRumely:ptdbpl}. Hence, the $v$-adic modulus concentration
  at $v$ given by Corollary \ref{cor:3 bis}, easily implies that the
  $v$-adic Galois orbits of $\ov D$-small sequences of algebraic
  points concentrate around this point of type III, regardless of the
  structure of the set $B_{v_{0}}$.
\end{rem}


\newcommand{\noopsort}[1]{} \newcommand{\printfirst}[2]{#1}
  \newcommand{\singleletter}[1]{#1} \newcommand{\switchargs}[2]{#2#1}
  \def\cprime{$'$}
\providecommand{\bysame}{\leavevmode\hbox to3em{\hrulefill}\thinspace}
\providecommand{\MR}{\relax\ifhmode\unskip\space\fi MR }
\providecommand{\MRhref}[2]{%
  \href{http://www.ams.org/mathscinet-getitem?mr=#1}{#2}
}
\providecommand{\href}[2]{#2}

\end{document}